\documentclass[12pt,reqno]{amsart}
\usepackage[margin=1in]{geometry}

\usepackage{hyperref}
\usepackage{amsthm}
\usepackage{amssymb}
\usepackage{amsxtra}
\usepackage{mathtools}
\usepackage{kopp} 
\usepackage{epsfig}
\usepackage{verbatim}
\usepackage{cleveref}
\usepackage{mathrsfs}
\usepackage{cite}
\usepackage{tikz-cd}
\usepackage{autonum} 

\numberwithin{equation}{section}

\theoremstyle{plain}
\newtheorem{theorem}{Theorem}[section]
\newtheorem{thm}[theorem]{Theorem}
\newtheorem{prop}[theorem]{Proposition}
\newtheorem{cor}[theorem]{Corollary}
\newtheorem{lem}[theorem]{Lemma}

\newtheorem{conj}[theorem]{Conjecture}

\theoremstyle{definition}
\newtheorem{defn}[theorem]{Definition}

\newtheorem{egz}[theorem]{Example}

\theoremstyle{remark}
\newtheorem{rmk}[theorem]{Remark}

\AddToHook{env/prop/begin}{\crefalias{theorem}{prop}}
\AddToHook{env/cor/begin}{\crefalias{theorem}{cor}}
\AddToHook{env/lem/begin}{\crefalias{theorem}{lem}}
\AddToHook{env/que/begin}{\crefalias{theorem}{que}}
\AddToHook{env/conj/begin}{\crefalias{theorem}{conj}}
\AddToHook{env/assu/begin}{\crefalias{theorem}{assu}}
\AddToHook{env/cla/begin}{\crefalias{theorem}{cla}}
\AddToHook{env/defn/begin}{\crefalias{theorem}{defn}}
\AddToHook{env/nota/begin}{\crefalias{theorem}{nota}}
\AddToHook{env/cond/begin}{\crefalias{theorem}{cond}}
\AddToHook{env/egz/begin}{\crefalias{theorem}{egz}}
\AddToHook{env/intro/begin}{\crefalias{theorem}{intro}}
\AddToHook{env/rmk/begin}{\crefalias{theorem}{rmk}}

\begin{document}

\allowdisplaybreaks

\title[The Shintani--Faddeev modular cocycle]{The Shintani--Faddeev modular cocycle: Stark units from $q$-Pochhammer ratios}
\author{Gene S. Kopp}

\address{Department of Mathematics, Louisiana State University, Baton Rouge, LA, USA}
\email{kopp@math.lsu.edu}

\begin{abstract}
We give a new interpretation of Stark units associated to real quadratic fields as \textit{real multiplication values} of a \textit{modular cocycle}. The cocycle of interest is a meromorphic factor describing the modular transformations of the $q$-Pochhammer symbol and is related to the Shintani--Barnes double sine function and the Faddeev quantum dilogarithm. We prove a refinement of Shintani's Kronecker limit formula that relates square roots of Stark class invariants to real multiplication values of the cocycle, which are cohomological invariants.
\end{abstract}

\date{May 2, 2025}



\maketitle

\tableofcontents

\section{Introduction}\label{sec:intro}

In the theory of singular moduli, special values of modular functions generate abelian Galois extensions of imaginary quadratic fields. Such modular functions may be specified as ratios of products of $q$-Pochhammer symbols (for example, eta quotients and theta quotients)---in particular, theta quotients yield elliptic units. In this paper, we prove that certain other ratios of $q$-Pochhammer symbols meromorphically continue as a function of $\tau$ (with $q = e^{2\pi i \tau}$) to a region that includes an open subset of the real line, where their special values at real quadratic points are closely related to certain special values of $L$-functions at $s=0$. Then, under the assumption of a version of the Stark conjectures, limiting ratios of $q$-Pochhammer symbols generate abelian extensions of \textit{real} quadratic fields.

The central objects of interest are the \textit{Shintani--Faddeev Jacobi cocycle} $(\m,A) \mapsto \sfj{\m,A}{z}{\tau}$ and the closely related \textit{Shintani--Faddeev modular cocycle (with characteristics $\r$)} $A \mapsto \sf{\r}{A}{\tau}$.\footnote{The notation $\sf{\r}{A}{\tau}$ uses the Hebrew letter \textit{shin}. 
See \Cref{sec:typesetting} for typesetting information.} The Shintani--Faddeev Jacobi cocycle is a mapping from the Jacobi group $\Z^2 \semidirect \SL_2(\Z)$ to complex meromorphic functions that serves as a generalized factor of automorphy describing the modular transformation law of the $q$-Pochhammer symbol. The function $\sfj{\m,A}{z}{\tau}$ has been identified by Dimofte (in a slightly different form) as a partition function of a certain topological quantum field theory on a squashed lens space \cite{dimofte}. It has been called both the \textit{rarefied hyperbolic gamma function} and the \textit{generalized noncompact quantum dilogarithm} by Sarkissian and Spiridonov \cite{sarkissian}; it has also been studied by Garoufalidis and Wheeler \cite{gw} and Wheeler \cite{wheelerthesis}. It generalizes the Shintani's \textit{double sine function} \cite{shintani,shintanicertain} (named by Kurokawa \cite{kurokawa}), which was rediscovered by Faddeev \cite{faddeev} and is called the \textit{noncompact quantum dilogarithm} in the physics literature. Our name---the Shintaini--Faddeev Jacobi cocycle---was chosen to emphasize the basic algebraic role of the function (as a meromorphic cocycle for the Jacobi group) and its dual history (and importance) in number theory and physics. 

Our main theorem, \Cref{thm:main}, expresses a Stark class invariant as the square of a \textit{real multiplication (RM) value} of the Shintani--Faddeev modular cocycle (or equivalently, an RM value of the Shintani--Faddeev Jacobi cocycle), up to an explicit root of unity. The Stark class invariant is the value $\exp(-Z_\A'(0))$ for a Dirichlet series $Z_\A(s)$ that is defined as the difference of two ray class partial zeta functions of a real quadratic field $F$. Stark conjectured that the value $\exp(-Z_\A'(0))$ is an algebraic unit in an abelian extension of $F$ \cite{stark3, starkrealquad, stark4}. 
Tate's refinement of Stark's conjectures includes 
the further prediction that the square root of this invariant is in an abelian extension of $F$ \cite{tate}.
The square roots we obtain from RM values of $\shin^\r$ are sometimes positive and sometimes negative, and the sign defines a new class invariant. The sign plays a key role in an application to quantum information theory, the construction of symmetric, informationally complete, positive operator-valued measures (SIC-POVMs) \cite{afk}.

\subsection{The Shintani--Faddeev Jacobi and modular cocycles}

Before stating our main theorems, we give a short, self-contained description of the transcendental functions of interest. Here and throughout the paper, we use the notation $\ee{z} := e^{2\pi i z}$ for the complex exponential.

For $\m = \smcoltwo{m_1}{m_2} \in \Q^2$ and $A = \smmattwo{a}{b}{c}{d} \in \SL_2(\Z)$, we will consider the following infinite product defined for $(z,\tau) \in \C \times \HH$:
\begin{equation}
\sfj{\m,A}{z}{\tau} 
= \prod_{k=0}^\infty \frac{1-\ee{\frac{z}{c\tau+d}+(m_2+k)\frac{a\tau+b}{c\tau+d}-m_1}}{1-\ee{z+k\tau}}.
\end{equation}
This product will often be written as a ratio of two infinite $q$-Pochhammer symbols
\begin{equation}
\sfj{\m,A}{z}{\tau} 
= \frac{\left(\ee{\frac{z}{c\tau+d}+m_2\frac{a\tau+b}{c\tau+d}-m_1},\ee{\frac{a\tau+b}{c\tau+d}}\right)_\infty}{(\ee{z},\ee{\tau})_\infty},
\end{equation}
using the notation $(w,q)_\infty = \prod_{k=0}^\infty(1-wq^k)$. 
We will also write $\sfj{\m,A}{z}{\tau}$ for a meromorphic continuation of this function, and we will interpret
\begin{equation}
\sfj{\m,A}{z_0}{\tau} = \lim_{z \to z_0}\sfj{\m,A}{z}{\tau}
\end{equation}
when necessary, for example, when there are zeros in the numerator and denominator of the product formula.

The \textit{Shintani--Faddeev Jacobi cocycle} is the function from $\Z^2 \semidirect \SL_2(\Z)$ to a space of meromorphic functions given\footnote{For $\m \in \Z^2$, $\sigma_{\m,A}$ does not depend on $m_1$ (but does depend on $m_2$); indeed, this vestigial variable does not appear in treatments \cite{dimofte, sarkissian}. We keep it for bookkeeping purposes, to emphasize the Jacobi group action.} 
by $(\m,A) \mapsto \sigma_{\m,A}$.
For $(\m,A) \in \Z^2 \semidirect \SL_2(\Z)$, the function $\sfj{\m,A}{z}{\tau}$ has a meromorphic continuation to the domain $\C \times \DD_{\!A}$, where
\begin{equation}
\DD_{\!A} = 
\begin{cases}
\C \setminus (-\infty,-d/c] & \mbox{if } c > 0,\\
\C & \mbox{if } c=0 \mbox{ and } d>0,\\
\HH & \mbox{if } c=0 \mbox{ and } d<0,\\
\C \setminus [-d/c,\infty) & \mbox{if } c < 0.
\end{cases}
\end{equation}

Let $\r \in \Q^2$, and consider the congruence group
\begin{equation}
\Gamma_\r = \{A \in \SL_2(\Z) : A\r - \r \in \Z^2\}.
\end{equation}
Define the ``shin'' function to be
\begin{equation}
\sf{\r}{A}{\tau} = \sigma_{\r,A}(0,\tau);
\end{equation}
it continues meromorphically to $\tau \in \DD_{\!A}$.
The \textit{Shintani--Faddeev modular cocycle} is the function from $\Gamma_\r$ to a space of meromorphic functions given by $A \mapsto \shin_{\!A}$. 

\subsection{Partial zeta functions and the Stark conjectures}

We now describe the definitions of the ray class groups and zeta functions needed for the statement of the main theorem.
Let $F$ be a number field and $\OO$ be an order in $F$. Let $\mm$ be an ideal of $\OO$ and $\rS$ a subset of the set of real embeddings of $F$. The following algebraic structures are defined in \cite{kopplagarias,kopplagariasmonoids} and generalize the standard definitions from class field theory to arbitrary orders.

The \textit{ray class group of the order $\OO$ modulo $(\mm, \rS)$} is
\begin{equation}
\Cl_{\mm,\rS}(\OO) = \frac{\rJ_{\mm}^\ast(\OO)}{\rP_{\mm,\rS}(\OO)},
\end{equation}
where
\begin{align}
\rJ_{\mm}^\ast(\OO) &= \{\mbox{invertible fractional ideals of $\OO$ coprime to $\mm$}\}, \mbox{ and} \\
\rP_{\mm,\rS}(\OO) &= \{\alpha\OO \mbox{ such that } \alpha \con 1 \Mod{\mm} \mbox{ and } \rho(\alpha)>0 \mbox{ for } \rho \in \rS\}.
\end{align}
The \textit{\rcmia} is
\begin{equation}
\Clt_{\mm,\rS}(\OO) = \frac{\rJf_\mm(\OO)}{\sim_{\mm,\rS}},
\end{equation}
where
\begin{align}
\rJf_\mm(\OO) &= \{\aa \in \rJ_\OO^\ast(\OO) : \aa\OO[S_\mm^{-1}] \subseteq \OO[S_\mm^{-1}]\} \mbox{ with} \\
S_\mm &= \{\alpha \in \OO : \alpha\OO + \mm = \OO\},
\end{align}
and the equivalence relation $\sim_{\mm,\rS}$ is defined by
\begin{equation}
\aa \sim_{\mm,\rS} \bb \iff \begin{array}{c}\exists \ccc \in \rJf_{\mm}(\OO) \mbox{ and } \alpha, \beta \in \OO[S_\mm^{-1}] \mbox{ such that } \aa = \alpha\ccc, \bb = \beta\ccc, \\ 
\alpha - \beta \in \mm\OO[S_\mm^{-1}],
\sgn(\rho(\alpha)) = \sgn(\rho(\beta)) \mbox{ for all } \rho \in \rS.\end{array}
\end{equation}
The \textit{submonoid of zero classes} is
\begin{equation}
\ZClt_{\mm,\rS}(\OO) = \{[\dd] \in \Clt_{\mm,\rS}(\OO) : \dd \subseteq \mm\}.
\end{equation}
If the real embeddings of $F$ are labeled $\rho_1, \ldots, \rho_r$ and $\rS = \{\rho_{j_1}, \ldots, \rho_{j_k}\}$, the pair $(\mm,\rS)$ may be abbreviated as $\mm\infty_{j_1} \cdots \infty_{j_k}$.

Let $\A \in \Clt_{\mm,\rS}(\OO)$, and let $\sR$ be the element of $\Cl_{\mm,\rS}(\OO)$ defined by
\begin{equation}
\sR := \{\alpha\OO : \alpha \equiv -1 \Mod{\mm} \mbox{ and } \rho(\alpha)>0 \mbox{ for all } \rho \in \rS\}.
\end{equation} 
For $\re(s)>1$, define the \textit{ray class partial zeta function} and  the \textit{differenced ray class partial zeta function}, respectively, by
\begin{align}
\zeta_{\mm,\rS}(s,\A) &= \sum_{\aa \in \A} \Nm(\aa)^{-s}, \mbox{ and} \\
Z_{\mm,\rS}(s,\A) &= \zeta_{\mm,\rS}(s,\A) - \zeta_{\mm,\rS}(s,\sR\A).
\end{align}
 
The Takagi existence theorem associates to the ray class group $\Cl_{\mm,\rS}(\OO_F)$ a \textit{ray class field}. This correspondence is extended to nonmaximal orders by \cite[Thm.\ 1.1]{kopplagarias}, which associates to $\Cl_{\mm,\rS}(\OO_F)$ the \textit{ray class field $H_{\mm,\rS}^{\OO}$ of $\OO$ modulo $(\mm,\rS)$}, a particular abelian Galois extension of $F$ with certain properties.

A famous series of conjectures of Stark connects the leading term of the Taylor series expansion of partial zeta functions at $s=1$ to units in abelian extensions. More generally, the Stark conjectures describe generalized regulators for Artin $L$-functions of possibly non-abelian Galois extensions \cite{stark1,stark2,stark3,starkrealquad,stark4}. The present paper is concerned only with the rank $1$ Stark conjecture in the abelian case with real quadratic base field.

For a real quadratic field $F$ and $\A \in \Cl_{\mm\infty_2}(\OO_F)$, the Stark conjectures predict that the real number
\begin{equation}
\exp(-Z_{\mm\infty_2}'(0,\A))
\end{equation}
is an algebraic unit $\e_\A$. Except in trivial cases 
(when $\exp(-Z_{\mm\infty_2}'(0,\A)) = \e_\A = 1$),
Stark conjectured that $H_{\mm\infty_2}^{\OO_F} = F(\e_\A)$. He also conjectured a compatibility with the Artin map: $(\Art(\A))(\e_{\id}) = \e_\A$. Tate's refinement of the Stark conjectures \cite{tate} includes the claim that $F(\e_\A^{1/2})$ is abelian over $F$, which he attributes to Stark.

\subsection{Eta-multiplier and theta-multiplier characters}

The Shintani--Faddeev cocycles are intimately connected to the Dedekind eta function and Jacobi theta functions, which are half-integral weight modular forms with character. The properties of these functions are reviewed in \Cref{sec:prelim}. For now, we describe their characters briefly in order to state the main theorem.

For $\tau \in \HH$, the \textit{Dedekind eta function} is
\begin{equation}
\eta(\tau) = \ee{\frac{\tau}{24}}\prod_{k=1}^\infty \left(1-\ee{k\tau}\right).
\end{equation}
Under the \textit{fractional linear transformation action} $A \cdot \tau = \frac{a\tau+b}{c\tau+d}$ of $A = \smmattwo{a}{b}{c}{d} \in \SL_2(\Z)$, the eta function transforms by
\begin{equation}
\eta\!\left(\frac{a\tau+b}{c\tau+d}\right)
= \psi\!\left(A,\sqrt{c\tau+d}\right)\sqrt{c\tau+d}\ \eta(\tau).
\end{equation}
Here $\psi$ is a character of the metaplectic group (see \Cref{sec:covering}) taking values is the group of complex roots of unity $\mu_\infty(\C)$. (In fact, its values are all $24$-th roots of unity.) Its square is a bona fide character of the modular group: $\psi^2 : \SL_2(\Z) \to \mu_\infty(\C)$. A formula for $\psi$ due to Rademacher is given as \Cref{thm:psiformula}.

For $\r = \smcoltwo{r_1}{r_2} \in \Q^2$ and $\tau \in \HH$, the \textit{Jacobi theta function with characteristics $\r$} is
\begin{equation}
\theta_\r(\tau) 
= \sum_{n=-\infty}^\infty \ee{\foh \left(n+r_2+\foh\right)^2\tau + \left(n+r_2+\foh\right)\left(-r_1+\foh\right)}.
\end{equation}
Under the fractional linear transformation action of $A = \smmattwo{a}{b}{c}{d} \in \Gamma_\r$, this theta function transforms by
\begin{equation}\label{eq:thetatransintro}
\theta_\r(A\cdot\tau) = \psi\!\left(A,\sqrt{c\tau+d}\right)^3\chi_\r(A)\sqrt{c\tau+d}\ \theta_\r(\tau).
\end{equation}
The character $\chi_\r : \Gamma_\r \to \mu_\infty(\C)$ is given by the formula
\begin{align}\label{eq:chifmlaintro}
\chi_\r(A) &:= -(-1)^{\delta_2(\!A\r-\r)}\ee{-\frac{1}{2}\det\!\rowtwo{A\r}{\r}}, \mbox{ where} \\
\delta_2(\q) &:= \begin{cases}
1, & \mbox{if } \q \in 2\Z^2, \\
0, & \mbox{if } \q \nin 2\Z^2.
\end{cases}
\end{align}
For proofs of \eqref{eq:thetatransintro} and \eqref{eq:chifmlaintro}, see \Cref{thm:thetamod} and \Cref{lem:charsimp}.

\subsection{Main result:\ a limit formula}

The main theorem relates a generalized ``Sark class invariant'' $\exp(Z_{\mm\infty_2}'(0,\A))$ to a limit of ratios of $q$-Pochhammer symbols. It may be considered a refinement of Shintani's Kronecker limit formula in the real quadratic case \cite{shintani}. Its statement does not require Shintani decomposition or continued fractions.

\begin{thm}\label{thm:main}
Let $\OO$ be an order in a real quadratic field $F \subset \R$, with Galois conjugation map $x \mapsto x'$, and let $\mm$ be a nonzero $\OO$-ideal.
Let $\A \in \Clt_{\mm\infty_2}(\OO) \setminus \ZClt_{\mm,\rS}(\OO)$, let $\A_0$ be the class of $\A$ in $\Cl(\OO)$, choose some $\bb \in \A_0^{-1}$ coprime to $\mm$, and write $\bb\mm = \alpha(\beta\Z+\Z)$ for some $\alpha,\beta \in F$ such that $\alpha$ is totally positive and $\beta > \beta'$. Choose $\r = \smcoltwo{r_1}{r_2} \in \Q^2$ such that $\alpha(r_2\beta-r_1)\OO \in \bb\A$ and $r_2\beta'-r_1>0$.
Write
\begin{equation}
\{B \in \Gamma_\r : B \cdot \beta = \beta\} = \langle A \rangle \mbox{ or } \langle -I, A \rangle
\end{equation}
with $A \smcoltwo{\beta}{1} = \lambda \smcoltwo{\beta}{1}$ for $\lambda > 1$.
Let $n = \frac{2}{\abs{\phi^{-1}(\A)}}$, where $\phi : \Clt_{\mm\infty_1\infty_2}(\OO) \to \Clt_{\mm\infty_2}(\OO)$ is the natural quotient map.
Then 
\begin{align}\label{eq:main}
\exp\!\left(n Z_{\mm\infty_2}'(0,\A)\right)
& = (\psi^{-2}\chi_\r^{-1})(A) \ \sf{\r}{A}{\beta}^2,
\end{align}
where $\shin^{\r}$ denotes the Shintani--Faddeev modular cocycle. Explicitly, the value
\begin{align}
\sf{\r}{A}{\beta} 
= \lim_{y \to 0^+} \frac{(\widetilde{w}_y,\widetilde{q}_y)_\infty}{(w_y,q_y)_\infty}
= \lim_{y \to 0^+} \prod_{k=0}^\infty \frac{1-\widetilde{w}_y\widetilde{q}_y^k}{1-w_yq_y^k},
\end{align}
where the parameters in the product are
$q_y=\ee{\beta+y i}$, 
$w_y=\ee{r_2(\beta+y i)-r_1)}$, 
$\widetilde{q}_y=\ee{A\cdot(\beta+y i)}$, 
and 
$\widetilde{w}_y=\ee{r_2(A\cdot(\beta+y i))-r_1}$.
\end{thm}

The proof of \Cref{thm:main} relies on most of the lemmas in the paper and will be completed in \Cref{sec:proof}. The key idea of the proof is to use the cocycle condition to ``telescope'' a variant of Shintani's formula based on a continued fraction expansion.

The statement of \Cref{thm:main} could be made slicker in several ways, at the expense of hiding some of the complexity behind further definitions. 
\begin{itemize}
\item
The value $\sf{\r}{\!A}{\tau}$ in \Cref{thm:main} is a \textit{real multiplication (RM) value} of the cocycle $\shin^{\r}$, as defined in \Cref{sec:stable}. It may be written as
\begin{equation}
\shin^\r[\beta] := \sf{\r}{\!A}{\tau},
\end{equation}
as it only depends on the real quadratic number $\beta$. Indeed, it also depends only on the class of $\shin^\r$ in a certain cohomology group. 
\item 
The RM values of the Jacobi cocycle coincide with those of the modular cocycle:
\begin{equation}
\sigma[r_2\beta-r_1,\beta] = \shin^\r[\beta],
\end{equation}
although this equivalence hides a shift in the elliptic variable in the definition of the RM value. Thus, \Cref{thm:main} may also be understood as a result about the RM values of Shintani--Faddeev Jacobi cocycle. From this perspective, the theorem exhibits all Stark class invariants of ray class fields over real quadratic fields as RM values of a single object, the Jacobi cocycle $\sigma$.
\item
The characters could be absorbed into the definition of either the modular or the Jacobi cocycle. In particular, if one defines the ``samech cocycle'' to be
\begin{equation}
\samech^\r_{\!A}(\tau) = (\psi^{-2}\chi_\r^{-1})(A) \ \sf{\r}{A}{\tau}^2,
\end{equation}
then \eqref{eq:main} becomes
\begin{equation}
\exp\!\left(n Z_{\mm\infty_2}'(0,\A)\right) = \samech^\r[\beta],
\end{equation}
and the matrix $A$ may be left out of the theorem statement.
\end{itemize}

We give a corollary to \Cref{thm:main} that is particularly essential to the construction of SIC-POVMs in \cite{afk}. 
\begin{cor}
If $\beta$ is a real quadratic number, $\r \in \Q^2$, and
\begin{equation}
\{B \in \Gamma_\r : B \cdot \beta = \beta\} = \langle A \rangle \mbox{ or } \langle -I, A \rangle
\end{equation}
such that $A \smcoltwo{\lambda}{1} = \lambda \smcoltwo{\beta}{1}$ for $\lambda > 1$,
then $\samech^\r[\beta] = (\psi^{-2}\chi_\r^{-1})(A) \, \sf{\r}{A}{\beta}^2$ is a positive real number.
\end{cor}
\begin{proof}
By \Cref{thm:correspondence}, 
every such pair $(\r,\beta)$ corresponds to some $\A \in \Clt_{\mm\infty_2}(\OO)$ in the manner of \Cref{thm:main}. Since $Z_{\mm\infty_2}'(0,\A) \in \R$, it follows from \eqref{eq:main} that $(\psi^{-2}\chi_\r^{-1})(A) \, \sf{\r}{A}{\beta}^2$ is a positive real number.
\end{proof}

\subsection{Conditional results and conjectures:\ algebraicity}

If one assumes an appropriate version of a Stark conjecture, our main theorem implies that the RM values of the Shintani--Faddeev cocycle lie in abelian extensions of real quadratic fields, at least in the maximal order case.

\begin{thm}\label{thm:field}
Assume \Cref{conj:stark2} (a consequence of Tate's refinement of the Stark conjectures). 
Let $\beta \in \R$ such that $a\beta^2+b\beta+c=0$ with $a,b,c \in \Z$, $b^2-4ac$ not a square, and let $\r \in \Q^2$. 
\begin{enumerate}
\item There exists some $n \in \N$ such that $\shin^{\r}[\beta]^n$ is an algebraic unit in an abelian extension of $F=\Q(\beta)$.
\item If $b^2-4ac$ is a fundamental discriminant, then we may take $n=1$. Moreover, if $\cc = \beta\Z+\Z$, and $\mm$ is the kernel of the $\OO_F$-module map $\OO_F \to ((r_2\beta-r_1)\OO_F+\cc)/\cc$ given by $1 \mapsto r_2\beta-r_1$, then $\samech^\r[\beta] \in H_{\mm\infty_2}$.
\end{enumerate}
\end{thm}

We conjecture that the assumption in \Cref{thm:field}(2) that the discriminant is fundamental (equivalently, that $\colonideal{\beta\Z+\Z}{\beta\Z+\Z} = \OO_F$) is unnecessary. 
It is not clear whether that conjecture, stated below as \Cref{conj:field}, follows from the Stark conjectures and their existing refinements. Obstructions to proving this include discrepancies in Euler factors between ``Galois-theoretic'' and ``ray class-theoretic'' $L$-functions in the non-maximal order case, as well as the failure of some pairs $(\r,\beta)$ to be in the image of primitive classes under any of the maps $\Upsilon_{\mm}$ in \Cref{thm:correspondence}. 
Numerical evidence for \Cref{conj:field} is given in \cite{afk}.
\begin{conj}\label{conj:field}
If $\beta \in \R$ such that $a\beta^2+b\beta+c=0$ with $a,b,c \in \Z$, $b^2-4ac$ not a square, and $\r \in \Q^2$, then $\shin^{\r}[\beta]$ is an algebraic unit in an abelian extension of $\Q(\beta)$. 
Moreover, if $\mm$ is an $\OO$-invertible ideal such that $(\r,\beta) \in \M_{\OO,\mm}$ in the notation of \Cref{thm:correspondence},
then $\samech^\r[\beta] \in H_{\mm\infty_2}$. 
\end{conj}

In addition to these results and conjectures, it is natural to ask for:
\begin{itemize}
\item[(1)] a precise description of the field generated by $\sqrt{\samech^\r[\beta]} = \sqrt{(\psi^{-2}\chi_\r^{-1})(A)} \ \sf{\r}{A}{\beta}$, as a ray class field or a subfield thereof, and
\item[(2)] an analogue of the Shimura reciprocity law, that is, a complete description of the action of $\Gal(\ol{\Q}/\Q(\beta))$ on the values $\shin^\r[\beta]$.
\end{itemize}
It is possible to deduce a Shimura reciprocity law for the \textit{squares} $\shin^{\r}[\beta]^2$, in the fundamental discriminant case, from the Stark conjectures and the results of the present paper. 
However, we are hopeful that the study of SIC-POVMs and related objects will help produce conjectural answers to (1) and (2), and we postpone these lines of inquiry to future work.

\subsection{Prior work}\label{sec:prior}

The function $\sigma_{\m,A}(z,\tau)$ defining the Shintani--Faddeev Jacobi cocycle is studied under other names by Dimofte \cite{dimofte} and Sarkissian and Spiridonov \cite{sarkissian} in mathematical physics and by Garoufalidis and Wheeler \cite{gw} and Wheeler \cite{wheelerthesis} in the context of low-dimensional topology and quantum modularity, with Chern--Simons theory playing a central role for both sets of authors. 

Yamamoto first suggested that Shintani's formula for the Stark unit would ``telescope'' down to a limit of an absolute value of a ratio of $q$-Pochhammer symbols, showing this in one example and suggesting that the general case could be handled using continued fractions \cite{yamamotofactorization}.  
Our \Cref{thm:main} brings Yamamoto's idea to its fulfillment while also removing the absolute value and giving a cocycle interpretation.

The present author has previously given another limit formula for ray class partial zeta function of real quadratic fields that also does not require Shintani decomposition or continued fractions \cite{koppklf}. That formula, which is substantially more complicated, relied on a different continuous interpolation of the arithmetic zeta functions, using Mellin transforms of non-holomorphic indefinite theta functions.

Outside of the lines of research opened by Shintani and Faddeev, variants of the double sine function have appeared in other areas of mathematics. One such variant is the \textit{quantum exponential function} in the theory of quantum groups; see Woronowicz \cite{wor}. A related function also appears in the work of Malyuzhinets on wave diffration in a wedge-shaped region; see the review \cite{osipov} and the references therein.

Several constructions in the number theory literature have both similar names and an indirect mathematical relationship to the Shintani--Faddeev cocycles. The \textit{Shintani cocycle} constructed by Solomon \cite{solomon1, solomon2} is a cocycle for $\SL_2(\Z)$ (or, with appropriate modification, $\GL_2(\Q)$) valued in power series that may be used to evaluate ray class zeta functions of real quadratic fields at nonpositive integers; Hill's generalization applies to totally real fields \cite{hill}.
Sczech's \textit{Eisenstein cocycle} \cite{sczech,charollois1,charollois2} is related to the Shintani cocycle but encodes Shintani's Kronecker formula for the derivative of a ray class partial zeta function of a totally real field at $s=0$ by means of an integration pairing with cycles. The Shintani and Eisenstein cocycles allow for $p$-adic interpolation, and the Eisenstein cocycle plays an important role in the work of Dasgupta, Kakde, and Ventullo on the Gross--Stark and Brumer--Stark conjectures \cite{dkv,dk1,dk2}.
The \textit{rigid meromorphic cocycles} of Darmon and Vonk \cite{vonk,dpv} (and, in particular, the \textit{Dedekind--Rademacher cocycle}) are $p$-adic rigid meromorphic modular $1$-cocycles for $\SL_2\!\left(\Z\!\left[\frac{1}{p}\right]\right)$ whose real multiplication values are conjectured (and in some cases, proven) to be algebraic numbers in abelian extensions of a real quadratic field. Finally, papers of Radchenko and Zagier \cite{rz} and Choie and Kumar \cite{ck} discuss (but do not formalize) a cocycle interpretation of the Herglotz--Zagier function, which (conjecutrally) produces rank $2$ Stark regulators in the real quadratic case.

\subsection{Applications and future work}

Work of Marcus Appleby, Steven Flammia, and the author uses the Shintani--Faddeev modular cocycle as part of a conjectural construction of symmetrically complete positive operator-valued measures (SIC-POVMs) \cite{afk}. The Shintani--Faddeev modular cocycle provides the ``correct'' signed square roots of the Stark units needed to extend the present author's previous results on the problem \cite{koppsic} to arbitrary dimension. The construction of SIC-POVMs may be viewed as a geometric interpretation of Stark units.

Many future research directions naturally stem from this paper. Elucidating the precise mathematical relationships between the Shintani--Faddeev cocycle and the various cocycles discussed in \Cref{sec:prior} is an important (and far from trivial) undertaking. 
Investigating RM values of other, generally noncommutative, ``quantum modular cocycles'' studied by Garoufalidis and Wheeler \cite{gw} and Wheeler \cite{wheelerthesis} is sure to prove interesting. More broadly, one hopes to make sense of the myriad possible connections to quantum modularity, knot theory, $3$-manifolds, topological quantum field theory, and gauge theory. Another future direction of research may be to align the constructions in this paper with Manin's noncommutative geometry perspective on real multiplication \cite{manin}.

The author is hopeful the results of this paper will eventually be extended to arbitrary (not just totally real) number fields. For example, in the complex cubic case, the analogous $1$-cocycle for $\SL_3$ seems to come from the elliptic gamma function \cite{felder,gerbe}---this has been partially shown by the impressive work of Bergeron, Charollois, and Garc\'{i}a \cite{bcg}, but further work would be needed to define the cocycle precisely, remove the $\aa$-smoothing,
and characterize arbitrary ``stable values'' of the cocycle. One can make some educated guesses that hint at a program for general number fields to realize Stark units and variants thereof as stable values of function-valued $(r_1+r_2-1)$-cocycles for $\SL_n$. However, much work is needed to pin down the details in the totally real ($r_2=0$) and almost totally real ($r_2=1$) cases, and genuinely new ideas will be needed in the general case.

\subsection{Structure of this paper}

We briefly outline the format of the paper. 

\Cref{sec:prelim} proves needed basic results about $q$-Pochhammer symbols, eta functions, theta functions, and the characters $\psi$ and $\chi_\r$. Content that can be found elsewhere is summarized with references to proofs provided. The identities we desire for the character $\chi_\r$ do not appear to be in the literature in the form we need and so must be proven in detail.

\Cref{sec:moduli} provides background on ray class monoids and describes how ray classes are associated to ``real multiplication points'' in a moduli space.

\Cref{sec:wannabe} introduces the Shintani--Faddeev modular and Jacobi cocycles and defines their stable values (including real multiplication values). It relates them to the theta functions, eta functions, and characters describes in \Cref{sec:prelim}. A more sophisticated perspective on these cocycles is provided in \Cref{sec:cohomology} by a form of equivariant cohomology. Specifically, $\shin^\r$ is an element of a first cohomology group of the global sections of $\Gamma_\r$-invariants of a certain complex of sheaves. This approach is reminiscent of Bekki's work on Eisenstein and Shintani--Barnes cocycles \cite{bekki} but is not directly compatible.

\Cref{sec:partial} defines the partial zeta functions of interest, proves relations between different types of partial zeta functions, and connects a few different versions of the Stark conjectures. \Cref{sec:partialzero} focuses on evaluating partial zeta functions at $s=0$. It uses continued fractions following the approach of Tangedal \cite{tangedal}. Several technical results on continued fractions and related quantities are required, and their proofs constitute a large portion of that section.

The proofs of \Cref{thm:main} and \Cref{thm:field} are completed in \Cref{sec:completing}. Some further implications are discussed, and a numerical example is provided.

\subsection{List of notation}

The following list describes some of the notation used in the paper that is nonstandard, uncommon, new, or holds some potential for confusion.
\begin{itemize}
\item $\ee{z} = e^{2\pi i z}$ for $z \in \C$.
\item $\C \cup \{\infty\}$ is the Riemann sphere (which can be identified with $\Pj^1(\C)$).
\item $\HH :=  \{\tau \in \C : \im(\tau)>0\}$ is the upper half plane.
\item Vectors are assumed to be column vectors unless otherwise stated; the transpose of a vector $\v$ is $\v^\top$.
\item If $\v = \smcoltwo{v_1}{v_2}$ and $\w = \smcoltwo{w_1}{w_2}$ are vectors, the ``standard'' symplectic form is denoted $\symp{\v}{\w} := -\det\smmattwo{v_1}{w_1}{v_2}{w_2} = v_2w_1 - v_1w_2$.
\item If $\r = \smcoltwo{r_1}{r_2} \in \R^2$ and $\tau \in \C$, then $\sympt{\r}{\tau} := \symp{\r}{\smcoltwo{\tau}{1}} = r_2\tau-r_1$.
\item If $A = \smmattwo{a}{b}{c}{d} \in \SL_2(\R)$ and $\tau \in \C$, then $j_{\!A}(\tau) := c\tau+d$, and $s_{\!A}(\tau) = \sgn(j_{\!A}(\tau))$.
\item If $A = \smmattwo{a}{b}{c}{d} \in \SL_2(\R)$ and $\tau \in \C \cup \{\infty\}$, then $A \cdot \tau := \frac{a\tau+b}{c\tau+d}$.
\item If $(\m,A) \in \Q^2 \semidirect \SL_2(\Q)$ (the rational Jacobi group) and $(z,\tau) \in \C \times (\C \setminus \Q)$, then $(\m,A) \cdot (z,\tau) := \left(\frac{z}{j_{\!A}(\tau)} + \sympt{\m}{A\cdot\tau},A\cdot\tau\right)$, and $A \cdot (z,\tau) := (\mathbf{0},A) \cdot (z,\tau)$.
\item For $\r \in \Q^2$, the group $\Gamma_\r = \{A \in \SL_2(\Z) : A\r - \r \in \Z^2\}$.
\item If $F$ is a number field, then $\rho_1, \ldots, \rho_{r_1}$ are its real embeddings, and $\infty_1, \ldots, \infty_{r_1}$ are formal symbols denoting the associated real places. If $F$ is a real quadratic field, we will sometimes fix a choice of real embedding $F \subset \R$ and denote $\rho_1(x) = x$ and $\rho_2(x) = x'$ for $x \in F$, so that $x \mapsto x'$ is the nontrivial Galois automorphism.
\item $\Cquad$ is the set consisting of those irrational complex numbers that are roots of a degree two polynomial in $\Z[x]$, $\Rquad = \Cquad \cap \R$, and $\Hquad = \Cquad \cap \HH$. 
\item If $R$ is a commutative Noetherian domain with fraction field $F$, and $\aa, \bb$ are fractional $R$-ideals (that is, finitely generated $R$-submodules of $F$), then the \textit{quotient ideal} is the fractional ideal $\colonideal{\aa}{\bb} = \{\gamma \in F : \gamma\bb \subseteq \aa\}$.
\item If $R$ is a commutative ring and $\aa, \bb$ are $R$-ideals, then $\aa$ and $\bb$ are \textit{coprime}, or $\aa$ is \textit{coprime to $\bb$}, if $\aa + \bb = R$. This terminology will be used even in rings that are not Dedekind domains, particularly non-maximal orders of number fields. Additionally, if $R$ is a commutative Noetherian domain, then a fractional ideal $\cc$ is \textit{coprime} to an integral ideal $\mm$ if $\cc = \aa\bb^{-1}$ for an integral ideal $\aa$ and an invertible integral ideal $\bb$ such that $\aa$ and $\bb$ are both coprime to $\mm$. 
\item If $\r = \smcoltwo{r_1}{r_2} \in \R^2$, then $\{\r\} := \smcoltwo{r_1-\floor{r_1}-1}{r_2-\floor{r_2}}$.
\item The character $\psi$ on $\Mp_2(\Z)$ is a variant of the Rademacher/Meyer invariant and is defined in \Cref{thm:psiformula}.
\item The character $\chi_\r$ on $\Gamma_\r$ is defined in \Cref{thm:thetamod}; see \Cref{lem:charsimp} for a simplified formula.
\item The theta function $\vartheta_\r(z,\tau)$ is defined in \Cref{defn:thetachar}, and $\theta_\r(\tau) = \vartheta_\r(0,\tau)$.
\item The \textit{ray class group of an order} $\Cl_{\mm,\rS}(\OO)$ and the associated \textit{ray class field of an order} $H_{\mm,\rS}^\OO$ are defined in \Cref{sec:rayclassorder}. We set $H_{\mm,\rS} = H_{\mm,\rS}^{\OO_F}$.
\item The \textit{\rcmia} $\Clt_{\mm,\rS}(\OO)$ is defined in \Cref{sec:monoid}, together with its \textit{submonoid of zero classes} $\ZClt_{\mm,\rS}(\OO)$.
\item For $A \in \SL_2(\Z)$, the complex domain $\DD_{\!A}$ is defined in \eqref{eq:DD}, and the complex domain $\tDD_{\!A}$ is defined in \eqref{eq:tDD}.
\item The function $\sfj{\m,A}{z}{\tau}$ is defined in \Cref{defn:sfjacobimaster}.
\item The function $\sf{\r}{\!A}{\tau}$ is defined in \Cref{defn:sfmodular}.
\item The ``stable value''/``RM value'' notation $w[\beta]$ 
(e.g., with $w=\shin^\r$) is defined in \Cref{sec:stable} for modular cocycles and \Cref{sec:stablejacobi} for Jacobi cocycles.
\item \textit{Hirzebruch--Jung (HJ) continued fractions} and \textit{Hirzebruch--Jung (HJ) cycle data} are defined in \Cref{sec:continuedfractions}, with attached notation.
\end{itemize}

\section{Preliminaries on $q$-Pochhammer symbols and half-integral weight modular forms}\label{sec:prelim}

In this section, we will provide some necessary foundational definitions and results about $q$-Pochhammer symbols as well as certain (fractional weight) modular forms on congruence subgroups of $\SL_2(\Z)$ and its metaplectic cover $\Mp_2(\Z)$.

\subsection{$\SL_2$ and the standard symplectic form}

For a commutative ring $R$ with unity, the group $\SL_2(R)$ is the same as the symplectic group $\Sp_2(R)$ (a special case of the general symplectic group $\Sp_{2n}(R)$). In particular, the $2 \times 2$ symplectic group is defined to be
\begin{equation}
\Sp_2(R) = \{\smmattwo{a}{b}{c}{d} \in \Mat_{2 \times 2}(R) : \smmattwo{a}{b}{c}{d}^\top\smmattwo{0}{1}{-1}{0}\smmattwo{a}{b}{c}{d} = \smmattwo{0}{1}{-1}{0}\},
\end{equation}
those matrices preserving the \textit{standard symplectic form} $\symp{\u}{\v} := \u^\top \smmattwo{0}{1}{-1}{0} \v$. Examining this condition, we note
$\smmattwo{a}{b}{c}{d}^\top\smmattwo{0}{1}{-1}{0}\smmattwo{a}{b}{c}{d} = \smmattwo{0}{ad-bc}{-(ad-bc)}{0}$,
so the condition is equivalent to $ad-bc=1$.

We also introduce the notation $\sympt{\u}{\tau} := \symp{\u}{\smcoltwo{\tau}{1}} = u_2\tau - u_1$ for $\u = \smcoltwo{u_1}{u_2}$.

\subsection{Fractional linear transformations}

The special linear group $\SL_2(\R)$ acts on $\C \cup \{\infty\}$ by the \textit{fractional linear transformation action} $\smmattwo{a}{b}{c}{d} \cdot \tau = \frac{a\tau+b}{c\tau+d}$ (with $\smmattwo{a}{b}{c}{d} \cdot \infty = \frac{a}{c}$), and this action restricts to an action on the upper half plane $\HH$. If $\tau = \smmattwo{a}{b}{c}{d} \in \SL_2(\R)$ and $\tau \in \C$, we define
\begin{equation}
j_{\!A}(\tau) := c\tau+d
\end{equation}
and note that $j$ satisfies the cocycle relation $j_{\!A_1A_2}(\tau) = j_{\!A_1}(A_2\cdot\tau)j_{\!A_2}(\tau)$.

If $A \in \SL_2(\R)$ is \textit{hyperbolic}, meaning that $\Tr(A)>2$, the action of $A$ on $\C \cup \{\infty\}$ has two fixed points, both in $\R \cup \{\infty\}$. Following Katok \cite{katok}, the \textit{attracting fixed point} $\beta$ of $A$ is the fixed point satisfying $\beta = \lim_{n \to \infty} A^n \cdot \tau$ for all (equivalently, any) $\tau \in \HH$. Equivalently, $\smcoltwo{\beta}{1}$ is an eigenvector of $A$ with eigenvalue greater than 1.

A real number generating a quadratic extension of $\Q$ will be called a \textit{real quadratic number}, and the set of all real quadratic numbers will be denoted $\Rquad$. If $\beta$ is a real quadratic number, we denote by $\beta'$ its unique nontrivial Galois conjugate. If $A \in \SL_2(\Z)$ is hyperbolic, then its fixed points are two Galois conjugate real quadratic numbers. The set of all complex numbers generating quadratic extensions of $\Q$ will be denoted $\Cquad$, and those in the upper half-plane by $\Hquad = \Cquad \cap \HH$.

The \textit{real Jacobi group} is the semidirect product $\R^2 \semidirect \SL_2(\R)$ with group operation
\begin{equation}
(\m,A)(\n,B) = (\m+A\n,AB).
\end{equation}
The Jacobi group acts on $\C \times \HH$ by the \textit{Jacobi action}
\begin{equation}\label{eq:jacobiaction}
(\m,A)\cdot(z,\tau) = \left(\frac{z}{j_{\!A}(\tau)} + \sympt{\m}{A\cdot\tau},A\cdot\tau\right).
\end{equation}
When the right-hand side of \eqref{eq:jacobiaction} is well-defined, we will use the same notation $(\m,A)\cdot(z,\tau)$, even when $(z,\tau) \nin \C \times \HH$. In particular, the Jacobi action gives an action of the \textit{rational Jacobi group} $\Q^2 \semidirect \SL_2(\Q)$ on $\C \times (\C \setminus \Q)$.

Of particular import is the \textit{integer Jacobi group} $\Z^2 \semidirect \SL_2(\Z)$, which is used to describe the transformation laws of theta functions. The subgroups $\Z^2 \semidirect \{I\}$ and $\{\mathbf{0}\} \semidirect \SL_2(\Z)$ describe \textit{elliptic transformations} and \textit{modular transformations}, respectively.

For $N \in \N$ and $\r \in \Q^2$, we define several congruence subgroups of $\SL_2(\Z)$:
\begin{align}
\Gamma(N) &:= \{A \in \SL_2(\Z) \mid A \con I \Mod{N}\}, \\
\Gamma_1(N) &:= \{A \in \SL_2(\Z) \mid A \con \smmattwo{1}{\ast}{0}{1} \Mod{N}\}, \mbox{ and} \\
\Gamma_{\r} &:= \{A \in \SL_2(\Z) \mid A\r \con \r \Mod{1}\}.
\end{align}
Note that $\Gamma_{\smcoltwo{1/N}{0}} = \Gamma_1(N)$, and $\Gamma_{\r} \subset \Gamma(N)$ for $\r \in \frac{1}{N}\Z^2$.

\subsection{The covering groups $\widetilde{\SL_2(\R)}$ and $\Mp_2(\R)$}\label{sec:covering}

In this section, we review the definition of the universal covering group $\widetilde{\SL_2(\R)}$ of $\SL_2(\R)$ and the metaplectic group $\Mp_2(\R)$.

The group $\widetilde{\SL_2(\R)}$ may be defined as follows:
\begin{equation}
\widetilde{\SL_2(\R)} := 
\left\{
(A,\lambda) : 
\begin{array}{ll}
A \in \SL_2(\Z), \\ 
\lambda : \HH \to \C \mbox{ continuous with } \exp(\lambda(\tau))=c\tau+d
\end{array}
\right\}.
\end{equation}
Its group law is defined by $(A_1,\lambda_1)(A_2,\lambda_2) = (A_1A_2,\lambda_3)$ where $\lambda_3(\tau) = \lambda_1(A_2\cdot\tau) + \lambda_2(\tau)$. This group fits into a short exact sequence of the form
\begin{equation}\label{eq:sltildeseq}
1 \to \Z \to \widetilde{\SL_2(\R)} \to \SL_2(\R) \to 1,
\end{equation}
where the left map is given by $n \mapsto (I, 2\pi i n)$, and the right map is given by $(A,\lambda) \mapsto A$. 
This construction may be described more abstractly: Topologically, $\SL_2(\R)$ is homotopy equivalent to a circle (by Iwasawa decomposition), so $\pi_1(\SL_2(\R)) \isom \Z$, and $\widetilde{\SL_2(\R)}$ is the topological universal cover of $\SL_2(\R)$ endowed with a group structure.
(Algebraically, $\widetilde{\SL_2(\R)}$ is also the universal perfect central extension of $\SL_2(\R)$.)
For convenience, we also define the group $\widetilde{\SL_2(\Z)} = \{(A,\lambda) \in \widetilde{\SL_2(\R)} : A \in \SL_2(\Z)\}$. (This is an abuse of notation, because $\widetilde{\SL_2(\Z)}$ it is not a canonical covering group of $\SL_2(\Z)$ itself as an abstract group.)

The metaplectic group $\Mp_2(\R)$ (a special case of the more general $\Mp_{2n}(\R)$) is defined as a double cover of $\SL_2(\R) = \Sp_2(\R)$. In particular,
\begin{equation}
\Mp_2(\R) 
:= 
\left\{(A,\ep) : 
\begin{array}{ll}
A \in \SL_2(\Z), \\ 
\ep : \HH \to \C \mbox{ continuous with } \ep(\tau)^2=c\tau+d
\end{array}
\right\},
\end{equation}
with multiplication law given by $(A_1,\ep_1)(A_2,\ep_2) = (A_1 A_2, \ep_3)$ with $\ep_3(\tau) = \ep_1(A_2\cdot\tau)\ep_2(\tau)$. A surjective map $\widetilde{\SL_2(\R)} \to \Mp_2(\R)$ may be defined by sending $(A,\lambda) \mapsto (A,\ep)$ with $\ep(\tau) = \exp(\foh\lambda(\tau))$; the kernel of this map is identified with $2\Z \subset \Z$ in the exact sequence \eqref{eq:sltildeseq}. The metaplectic group is the two-fold central covering group of $\SL_2(\R)$, fitting into the short exact sequence
\begin{equation}
1 \to \Z/2\Z \to \Mp_2(\R) \to \SL_2(\R) \to 1.
\end{equation}
We also define the integer metaplectic group $\Mp_2(\Z) := \{(A,\ep) \in \Mp_2(\R) : A \in \SL_2(\Z)\}$.
More generally, for any discrete subgroup $\Gamma \leq \SL_2(\R)$, we will define the metaplectic cover
$\MG := \{(A,\ep) \in \Mp_2(\R) : A \in \Gamma\}$.
In particular, we will use the groups $\MG(N)$, $\MG_1(N)$, and $\MG_\r$, for $N \in \N$ and $\r \in \Q$.

\subsection{The $q$-Pochhammer symbol and variants}

We will need to use several versions of the $q$-Pochhammer symbol.
\begin{defn}
The \textit{finite $q$-Pochhammer symbol} is
\begin{equation}
(w,q)_n = 
\begin{cases}
\ds\prod_{k=0}^{n-1} (1-wq^k) & \mbox{for $n \geq 0$, and} \\
\ds\prod_{k=1}^{-n} (1-wq^{-k})^{-1} & \mbox{for $n<0$}.
\end{cases}
\end{equation}
This definition is equivalent to the following recursive definition, extended to negative integers so as to preserve the recursion:
\begin{align}
(w,q)_0 &= 1, \mbox{ and} \\
(w,q)_{n+1} &= (w,q)_n(1-wq^n).
\end{align}
\end{defn}

\begin{defn}
The \textit{infinite $q$-Pochhammer symbol} is defined for $w, q \in \C$ with $\abs{q} < 1$ by
\begin{equation}
(w,q)_\infty = \lim_{n \to \infty} (w,q)_n = \prod_{k=0}^\infty \left(1-wq^k\right).
\end{equation}
The following alternative notations are also used. The infinite $q$-Pochhammer symbol in ``Jacobi form notation'' for $z \in \C$ and $\tau \in \HH$ is
\begin{equation}
\varpi(z, \tau) := (\ee{z},\ee{\tau})_\infty.
\end{equation}
The infinite $q$-Pochhammer symbol in ``characteristics notation'' for $\r = \smcoltwo{r_1}{r_2} \in \R^2$ and $\tau \in \HH$ is
\begin{equation}
\varpi_\r(\tau) := \varpi\!\left(\sympt{\r}{\tau},\tau\right) = (\ee{r_2\tau - r_1},\ee{\tau})_\infty.
\end{equation}
\end{defn}

\begin{lem}\label{lem:ell}
If $z \in \C$, $\tau \in \HH$, and $m,n \in \Z$, then
\begin{equation}
\varpi(z+m\tau+n,\tau) = \left(\ee{z},\ee{\tau}\right)_m^{-1} \varpi(z,\tau).
\end{equation}
\end{lem}
\begin{proof}
Follows from inspection of the product forms of both sides.
\end{proof}

\subsection{The Dedekind eta function and its logarithm}\label{sec:eta}

In this section, we review the transformation theory of the Dedekind eta function. Proofs of assertions that are not proven here may be found in \cite{atiyah, rademacher}.

The Dedekind eta function
\begin{equation}
\eta(\tau) = \ee{\frac{\tau}{24}}\varpi(\tau,\tau) = \ee{\frac{\tau}{24}}\prod_{k=1}^\infty (1-\ee{k\tau}) \mbox{ for } \tau \in \HH
\end{equation}
has a continuous, well-defined logarithm defined by
\begin{equation}
(\log\eta)(\tau) = \frac{2\pi i\tau}{24} + \sum_{k=1}^\infty \log(1-\ee{k\tau}).
\end{equation}

The function $(\log\eta)(\tau)$ has a modularity property for the group $\widetilde{\SL_2(\Z)}$. Specifically, for $(A, \lambda) \in \widetilde{\SL_2(\Z)}$,
\begin{equation}\label{eq:psitrans}
(\log\eta)(A\cdot\tau) = \frac{2\pi i}{24}\Psi(A,\lambda) + \foh \lambda(\tau) + (\log\eta)(\tau),
\end{equation} 
where $\Psi : \widetilde{\SL_2(\Z)} \to \Z$ is an integer-valued group homomorphism.

Rademacher defines a function (not a homomorphism) $\Phi : \SL_2(\Z) \to \Z$ satisfying the relation
\begin{equation}\label{eq:phitrans}
(\log\eta)(A\cdot\tau) = \frac{2\pi i}{24}\Phi(A) + \foh \one_{c \neq 0}\log\!\left(-i\sgn(c)(c\tau+d)\right) + (\log\eta)(\tau),
\end{equation}
where the middle term in interpreted as zero if $c=0$, taking $\one_{c \neq 0} = 1$ and using the principal branch of the logarithm otherwise. This function is defined as
\begin{equation}\label{eq:Phi}
\Phi(A) =
\begin{cases}
\frac{b}{d}, &\mbox{if } c=0, \\
\frac{a+d}{c} - 12 \sgn(c) s(d,\abs{c}), &\mbox{if } c \neq 0,
\end{cases}
\end{equation}
where $s(d,\abs{c})$ is the \textit{Dedekind sum}
\begin{equation}
s(h,k) = \sum_{j=1}^{\abs{k}-1} \left(\frac{j}{k} - \floor{\frac{j}{k}} - \frac{1}{2}\right)\left(\frac{hj}{k} - \floor{\frac{hj}{k}} - \frac{1}{2}\right).
\end{equation}
(See (71.1), (71.2), and (71.22) on p.\ 150--151 and (68.3) on p.\ 146 of \cite{rademacher}.)
Comparing \eqref{eq:psitrans} and \eqref{eq:phitrans}, we see that $\Psi$ and $\Phi$ are related by the identity
\begin{equation}\label{eq:Psi}
\Psi(A,\lambda) = \Phi(A) + \frac{6}{\pi i}\left(\one_{c \neq 0}\log\!\left(-i\sgn(c)(c\tau+d)\right) - \lambda(\tau)\right).
\end{equation}

The modular transformation law for $\eta(\tau)$ follows from that for $(\log\eta)(\tau)$. If $(A,\ep) \in \Mp_2(\Z)$, then
\begin{equation}\label{eq:etatrans}
\eta(A\cdot\tau) = \psi(A,\ep)\ep(\tau)\eta(\tau),
\end{equation}
where $\psi(A,\ep) = \exp\!\left(\frac{2\pi i}{24}\Psi(A,\lambda)\right)$ for any $\lambda$ with $\ep(\tau) = \exp\!\left(\foh\lambda(\tau)\right)$. We also introduce the shorthand $\psi^2(A) := (\psi(A,\ep))^2$, which does not depend on the choice of $\ep$. Note that $\psi^2$ is a character of $\SL_2(\Z)$.

For the purposes of explicit calculation, it is important to note that the function $\psi(A,\ep)$ can be expressed in terms of Jacobi symbols without needing to use Dedekind sums.
\begin{thm}\label{thm:psiformula}
Let $A = \smmattwo{a}{b}{c}{d} \in \SL_2(\Z)$ and $\tau \in \HH$.
If $c=0$, then
\begin{equation}
\psi(A,\ep)\ep(\tau) = \ee{\frac{\sgn(d)b}{24}}.
\end{equation}
If $c > 0$, then
\begin{equation}
\psi(A,\ep)\frac{\ep(\tau)}{\sqrt{-i (c\tau+d)}} =
\begin{cases}
\left(\frac{d}{c}\right) \ee{\frac{1-c}{8}}\ee{\frac{bd(1-c^2)+c(a+d)}{24}}, & \mbox{ if } 2\nmid c,\\
\left(\frac{c}{\abs{d}}\right) \ee{\frac{d}{8}}\ee{\frac{ac(1-d^2)+d(b-c)}{24}}, & \mbox{ if } 2\nmid d.\\
\end{cases}
\end{equation}
Here, $\left(\frac{d}{c}\right)$ and $\left(\frac{c}{\abs{d}}\right)$ are Jacobi symbols.
If $c < 0$, then $\psi(A,\ep) = i\psi(-A,i\ep)$. 
\end{thm}
\begin{proof}
See \cite[p.\ 163]{rademacher}.
\end{proof}

\subsection{Jacobi theta functions with characteristics}

In this section, we review the transformation theory of Jacobi theta functions with arbitrary real characteristics. Our primary reference is Rademacher \cite{rademacher}. The author was unable to find a suitable source for the general transformation laws of theta functions with characteristics, so those have been proven here.

Jacobi defined four theta functions, $\vartheta_j$ for $j \in \{1,2,3,4\}$; we will treat $\vartheta_1$ as the ``basic'' theta function and define other theta functions in terms of $\vartheta_1$. The primary advantage of this approach is that the full $\SL_2(\Z)$-action takes $\vartheta_1$-values to other $\vartheta_1$-values, whereas the other $\vartheta_j$ are permuted with each other by the action.
\begin{defn}
For $z \in \C$ and $\tau \in \HH$, the \textit{first Jacobi theta function} is
\begin{equation}
\vartheta_1(z,\tau) = - \sum_{n=-\infty}^\infty \ee{\foh \left(n+\foh\right)^2\tau + \left(n+\foh\right)\left(z+\foh\right)}.
\end{equation}
\end{defn}

The first Jacobi theta function satisfies elliptic and modular transformation laws. (These may also be interpreted as a single transformation law under the action of $\Z^2 \semidirect \SL_2(\Z)$.)
\begin{thm}\label{thm:thetaell}
If $z \in \C$, $\tau \in \HH$, and $k,\ell \in \Z$, then
\begin{equation}
\vartheta_1(z+k\tau+\ell,\tau) = (-1)^{k+\ell}\ee{-\foh k^2\tau - kz}\vartheta_1(z,\tau).
\end{equation}
\end{thm}
\begin{proof}
See \cite[p.\ 177, (80.31)]{rademacher}. There is a misprint in that formula, corrected here.
\end{proof}
\begin{thm}\label{thm:thetatrans}
If $z \in \C$, $\tau \in \HH$, and $(A,\ep) \in \Mp_2(\Z)$ with $A = \smmattwo{a}{b}{c}{d}$, then
\begin{equation}
\vartheta_1(A\cdot(z,\tau)) = \psi(A,\ep)^3\ee{\frac{cz^2}{2(c\tau+d)}}\ep(\tau)\vartheta_1(z,\tau).
\end{equation}
\end{thm}
\begin{proof}
See \cite[p.~180]{rademacher}.
\end{proof}

Theta functions with characteristics involve an additional pair of real parameters $\r \in \R^2$, often taken to be rational in applications. In the special case $\r \in \left\{\smcoltwo{0}{0}, \smcoltwo{0}{1/2}, \smcoltwo{1/2}{0}, \smcoltwo{1/2}{1/2}\right\}$, they are essentially the four classical Jacobi theta functions usually called $\vartheta_j$ for $j \in \{1,2,3,4\}$. 
There is no widely accepted notational convention for theta functions with characteristics, and the following definition reflects the author's preferred notation. The use of a shift by $r_2\tau-r_1$ (the symplectic pairing of $\smcoltwo{r_1}{r_2}$ with $\smcoltwo{\tau}{1}$) instead of the more commonly used $r_1\tau+r_2$ makes the modular transformation law nicer.
\begin{defn}\label{defn:thetachar}
For $\r = \smcoltwo{r_1}{r_2} \in \R^2$, $z \in \C$, and $\tau \in \HH$, the \textit{Jacobi theta function with characteristics} is
\begin{align}
\vartheta_\r(z,\tau) 
&= \sum_{n=-\infty}^\infty \ee{\foh \left(n+r_2+\foh\right)^2\tau + \left(n+r_2+\foh\right)\left(z-r_1+\foh\right)} \\
&= -\ee{\foh r_2^2\tau + r_2\left(z-r_1+\foh\right)}\vartheta_1(z+r_2\tau-r_1,\tau).
\end{align}
In particular, $\vartheta_1(z,\tau) = -\vartheta_{\smcoltwo{0}{0}}(z,\tau)$.
\end{defn}

Shifting the characteristics by a pair of integers simply multiplies the theta function by a constant of norm $1$.
\begin{prop}\label{prop:thetacharshift}
Let $\r = \smcoltwo{r_1}{r_2} \in \R^2$, $\m = \smcoltwo{m_1}{m_2} \in \Z^2$, $z \in \C$, and $\tau \in \HH$. Then
\begin{equation}
\vartheta_{\r+\m}(z,\tau) = \ee{-m_1\left(r_2+\foh\right)} \vartheta_{\r}(z,\tau).
\end{equation}
\end{prop}
\begin{proof}
By \Cref{defn:thetachar}, we have
\begin{align}
\vartheta_{\r+\m}(z,\tau)
&= \sum_{n \in \Z} \ee{\left(n+m_2+r_2+\foh\right)^2\tau + \left(n+m_2+r_2+\foh\right)\left(z-m_1-r_1+\foh\right)} \\
&= \sum_{n \in \Z} \ee{\left(n+r_2+\foh\right)^2\tau + \left(n+r_2+\foh\right)\left(z-m_1-r_1+\foh\right)} \\
&= \ee{-m_1\left(r_2+\foh\right)} \vartheta_{\r}(z,\tau),
\end{align}
where in the second-to-last line we substituted $n \mapsto n-m_2$.
\end{proof}

The modular transformation law for the theta function with characteristics also follows by direct calculation from the modular transformation law for $\theta_1$; however, it is more complicated, so we give a detailed proof.
\begin{thm}\label{thm:thetachartrans}
Let $\r = \smcoltwo{r_1}{r_2} \in \R^2$, $z \in \C$, $\tau \in \HH$, and $(A,\ep) \in \Mp_2(\Z)$ with $A = \smmattwo{a}{b}{c}{d}$. Then
\begin{equation}\label{eq:thetachartrans}
\vartheta_{\!A\r}(A\cdot(z,\tau)) = \psi(A,\ep)^3\kappa(A,\r)\ee{\tfrac{cz^2}{2(c\tau+d)}}\ep(\tau)\vartheta_\r(z,\tau),
\end{equation}
where
\begin{equation}\label{eq:kappadef}
\kappa(A,\r) = \ee{\foh\left(cr_1 + (d-1)r_2 - acr_1^2 - 2bcr_1r_2 - bdr_2\right)}.
\end{equation}
\end{thm}
\begin{proof}
By \Cref{defn:thetachar},
\begin{align}
\vartheta_{\!A\r}(A\cdot(z,\tau))
&= -\ee{\foh(cr_1+dr_2)^2\tfrac{a\tau+b}{c\tau+d}+(cr_1+dr_2)\left(\tfrac{z}{c\tau+d}-ar_1-br_2+\foh\right)} \\
&\ \ \ \ \ \times \vartheta_1\!\left(\tfrac{z}{c\tau+d} + (cr_1+dr_2)\tfrac{a\tau+b}{c\tau+d} - (ar_1+br_2), \tfrac{a\tau+b}{c\tau+d}\right) \\
&= -\ee{\tfrac{cr_1+dr_2}{2(c\tau+d)}\left((cr_1+dr_2)(a\tau+b) + 2z - (2ar_1+2br_2-1)(c\tau+d)\right)} \\
&\ \ \ \ \ \times \vartheta_1\!\left(\tfrac{z+(cr_1+dr_2)(a\tau+b)-(ar_1+br_2)(c\tau+d)}{c\tau+d},\tfrac{a\tau+b}{c\tau+d}\right), \\
&= -\ee{\tfrac{cr_1+dr_2}{2(c\tau+d)}\left(r_2\tau-r_1 + 2z - (ar_1+br_2-1)(c\tau+d)\right)} \\
&\ \ \ \ \ \times 
\vartheta_1\!\left(\tfrac{z+r_2\tau-r_1}{c\tau+d},\tfrac{a\tau+b}{c\tau+d}\right), \label{eq:rsub}
\end{align}
using the relation 
\begin{equation}
(cr_1+dr_2)(a\tau+b)-(ar_1+br_2)(c\tau+d) = r_2\tau-r_1
\end{equation}
in \eqref{eq:rsub}.
Using the transformation law for $\vartheta_1$ from \Cref{thm:thetatrans}, we obtain
\begin{align}
\vartheta_{\!A\r}(A\cdot(z,\tau))
&= -\ee{\tfrac{cr_1+dr_2}{2(c\tau+d)}\left(r_2\tau-r_1 + 2z - (ar_1+br_2-1)(c\tau+d)\right)} \\
&\ \ \ \ \ \times \psi(A,\ep)^3\ee{\tfrac{c(z+r_2\tau-r_1)^2}{2(c\tau+d)}}\ep(\tau)\vartheta_1(z+r_2\tau-r_1,\tau) \\
&= -\psi(A,\ep)^3\ee{-\foh(ar_1+br_2-1)(cr_1+dr_2)} \\
&\ \ \ \ \ \times \ee{\tfrac{c(z+r_2\tau-r_1)^2+(cr_1+dr_2)(r_2\tau-r_1 + 2z)}{2(c\tau+d)}}\ep(\tau)\vartheta_1(z+r_2\tau-r_1,\tau) \\
&= -\psi(A,\ep)^3\ee{-\foh(ar_1+br_2-1)(cr_1+dr_2)} \\
&\ \ \ \ \ \times \ee{\tfrac{c z^2}{2(c\tau+d)} + \foh r_2(r_2\tau-r_1+2z)}\ep(\tau)\vartheta_1(z+r_2\tau-r_1,\tau) \\
&= -\psi(A,\ep)^3\ee{\foh(cr_1+dr_2-acr_1^2+(ad+bc)r_1r_2+bdr_2^2)} \\
&\ \ \ \ \ \times \ee{\tfrac{c z^2}{2(c\tau+d)} + \tfrac{r_2^2\tau}{2} + r_2\left(z-r_1+\foh\right) + \tfrac{r_1r_2}{2} - \tfrac{r_2}{2}}\ep(\tau)\vartheta_1(z+r_2\tau-r_1,\tau) \\
&= -\psi(A,\ep)^3\ee{\foh(cr_1+(d-1)r_2-acr_1^2-(ad+bc-1)r_1r_2-bdr_2^2)} \\
&\ \ \ \ \ \times \ee{\tfrac{c z^2}{2(c\tau+d)}}\ee{\tfrac{r_2^2\tau}{2} + r_2\left(z-r_1+\foh\right)}\ep(\tau)\vartheta_1(z+r_2\tau-r_1,\tau) \\
&= \psi(A,\ep)^3\ee{\tfrac{cr_1+(d-1)r_2-acr_1^2-2bcr_1r_2-bdr_2^2}{2}} \ee{\tfrac{c z^2}{2(c\tau+d)}}\ep(\tau)\vartheta_\r(z,\tau) \\
&= \psi(A,\ep)^3\kappa(A,\r)\ee{\tfrac{c z^2}{2(c\tau+d)}}\ep(\tau)\vartheta_\r(z,\tau). \tag*{\qedhere}
\end{align}
\end{proof}

\begin{lem}\label{lem:kappacocycle}
Let $\r \in \Q^2$. The function $\kappa$ satisfies the following cocycle condition: For any $A,B \in \Gamma_\r$,
\begin{equation}
\kappa(AB,\r) = \kappa(A,B\r)\kappa(B,\r).
\end{equation}
\end{lem}
\begin{proof}
This follows from \Cref{thm:thetachartrans} by applying \eqref{eq:thetachartrans} to each of $\vartheta_\r((A_1A_2)\cdot(z,\tau))$ with $A=A_1A_2$, $\vartheta_\r(A_1\cdot(A_2\cdot(z,\tau)))$ with $A=A_1$, and $\vartheta_\r(A_2\cdot(z,\tau))$ with $A=A_2$, and comparing the results.
\end{proof}

\begin{lem}\label{lem:kappaAinv}
Let $\r = \smcoltwo{r_1}{r_2} \in \R^2$ and $A = \smmattwo{a}{b}{c}{d} \in \SL_2(\Z)$.
With $\kappa(A,\r)$ defined as in \Cref{thm:thetachartrans},
\begin{align}
\kappa(A,A^{-1}\r) 
&= \kappa(A^{-1},\r)^{-1} \\
&= \ee{\foh\!\left(cr_1 + (-a+1)r_2 - cdr_1^2 + 2bcr_1r_2 - abr_2^2\right)}. \label{eq:kappaAinv}
\end{align}
\end{lem}
\begin{proof}
By \Cref{lem:kappacocycle}, $1 = \kappa(I,\r) = \kappa(AA^{-1},\r) = \kappa(A,A^{-1}\r)\kappa(A^{-1},\r)$, so $\kappa(A,A^{-1}\r) = \kappa(A^{-1},\r)^{-1}$. The second equality of \eqref{eq:kappaAinv} follows by plugging in $A^{-1} = \smmattwo{d}{-b}{-c}{a}$ to \eqref{eq:kappadef}.
\end{proof}

\subsection{Characters of modular theta functions}

We will now consider modular forms defined by theta functions. We define a character $\chi_\r$ on the group $\Gamma_\r$ that will play an important role in the theory of the Shintani--Faddeev modular cocycle. 

\begin{defn}
Let $\r \in \R^2$, $\tau \in \HH$.
The \textit{theta null with characteristics} is
\begin{equation}
\theta_\r(\tau) = \vartheta_\r(0,\tau).
\end{equation}
\end{defn}

\begin{thm}\label{thm:thetamod}
Let $\r \in \Q^2$.
The function $\theta_\r(\tau)$ is a weight $\foh$ modular form with character for the group
\begin{equation}
\MG_\r = \{(A,\ep) \in \Mp_2(\Z) : A\cdot\r \con \r \Mod{1}\}.
\end{equation}
Specifically, for $(A,\ep) \in \MG_\r$ and $\tau \in \HH$,
\begin{equation}\label{eq:thetamod}
\theta_\r(A\cdot\tau) = \psi(A,\ep)^3\chi_\r(A)\ep(\tau)\theta_\r(\tau)
\end{equation}
where $\chi_\r$ is a character on $\Gamma_\r$ defined by the formula
\begin{equation}
\chi_\r(A) = \ee{\tfrac{(c-d+1)r_1 + (-a+b+1)r_2 - cdr_1^2 + 2(a-1)dr_1r_2 - (a-2)b r_2^2}{2}}.
\end{equation}
\end{thm}
\begin{proof}
We first apply the modular transformation formula given in \Cref{thm:thetachartrans}.
\begin{align}
\theta_\r(A\cdot\tau) 
&= \vartheta_{\!A(A^{-1}\r)}(A\cdot(0,\tau)) \\
&= \psi(A,\ep)^3\kappa(A,A^{-1}\r)\ep(\tau)\vartheta_{\!A^{-1}\r}(z,\tau).
\end{align}
Since $A \in \Gamma_\r$, so is $A^{-1}$, and $A^{-1}\r = \r + \m$ for some $\m = \smcoltwo{m_1}{m_2}\in \Z^2$.
Applying \Cref{prop:thetacharshift} and regrouping factors, we have
\begin{equation}\label{eq:thetamod2}
\theta_\r(A\cdot\tau) = \psi(A,\ep)^3\chi_\r(A)\ep(\tau)\theta_\r(\tau),
\end{equation}
where $\chi_\r(A) := \kappa(A,A^{-1}\r)\ee{-m_1\left(r_2+\foh\right)}$.
We have
\begin{equation}
\m = (A^{-1}-I)\r = \coltwo{(d-1)r_1-br_2	}{-cr_1+(a-1)r_2}.
\end{equation}
Applying \Cref{lem:kappaAinv}, we may write $\chi_\r$ as
\begin{align}
\chi_\r(A)
&= \ee{\tfrac{cr_1 + (-a+1)r_2 - cdr_1^2 + 2bcr_1r_2 - abr_2^2}{2}}\ee{-\left((d-1)r_1-br_2\right)\left(r_2+\foh\right)} \\
&= \ee{\tfrac{(c-d+1)r_1 + (-a+b+1)r_2 - cdr_1^2 + 2(bc-d+1)r_1r_2 - (ab-2b)r_2^2}{2}} \\
&= \ee{\tfrac{(c-d+1)r_1 + (-a+b+1)r_2 - cdr_1^2 + 2(a-1)dr_1r_2 - (a-2)b r_2^2}{2}}.
\end{align}
The fact that $\chi_\r$ is a character on $\Gamma_\r$ follows by applying \eqref{eq:thetamod2} to each of $\theta_\r((A_1A_2)\cdot\tau)$ with $A=A_1A_2$, $\theta_\r(A_1\cdot(A_2\cdot\tau))$ with $A=A_1$, and $\theta_\r(A_2\cdot\tau)$ with $A=A_2$, and comparing the results.
\end{proof}

The character $\chi_\r$ can be written in a somewhat nicer way, as the following lemma shows.
\begin{lem}\label{lem:charsimp}
Let $\r \in \frac{1}{N}\Z^2$ for $N \in \N$. The character $\chi_\r$ on $\Gamma_\r$ has the following formula.
\begin{align}
\chi_\r(A) &= -(-1)^{\delta_2(A\r-\r)}\ee{\frac{1}{2}\symp{A\r}{\r}}, \mbox{ where} \\
\delta_2(\q) &:= 
\begin{cases}
1, &\mbox{if } \q \in 2\Z^2, \\
0, &\mbox{if } \q \nin 2\Z^2.
\end{cases}
\end{align}
Moreover, $\chi_\r$ takes values in the group $\mu_{N'}(\C)$ of $N'$-th roots of unity, where
\begin{equation}
N' = 
\begin{cases}
N & \mbox{if $N$ is odd}, \\
2N & \mbox{if $N$ is even}.
\end{cases}
\end{equation}
\end{lem}
\begin{proof}
Write $A = \smmattwo{a}{b}{c}{d}$ and $\r = \smcoltwo{r_1}{r_2}$; we have $A^{-1} = \smmattwo{d}{-b}{-c}{a}$. By \Cref{thm:thetamod},
\begin{align}
\chi_\r(A) 
&= \chi_\r(A^{-1})^{-1} \\
&= \ee{\tfrac{(-c-a+1)r_1 + (-d-b+1)r_2 + acr_1^2 + 2a(d-1)r_1r_2 + b(d-2) r_2^2}{2}}^{-1} \\
&= \ee{\tfrac{(a+c-1)r_1 + (b+d-1)r_2 - acr_1^2 - 2a(d-1)r_1r_2 - b(d-2) r_2^2}{2}} \\
&= \ee{\tfrac{(ar_1+br_2-r_1) + (cr_1+dr_2-r_1) - ar_1(cr_1+dr_2) - (d-2)r_2(ar_1+br_2)}{2}}.
\end{align}
Set $\r' = \smcoltwo{r_1'}{r_2'} := A\r = \smcoltwo{ar_1+br_2}{cr_1+dr_2}$. Then,
\begin{align}
\chi_\r(A) 
&= \ee{\tfrac{(r_1'-r_1) + (r_2'-r_2) - ar_1r_2' - (d-2)r_2r_1'}{2}} \\
&= \ee{\tfrac{(r_1'-r_1) + (r_2'-r_2) + 2r_2r_1' - (ar_1r_2' + dr_2r_1')}{2}}.
\end{align}
Moreover, $ar_1r_2' + dr_2r_1' = ar_1r_2' + (r_2'-cr_1)r_1' = r_1(ar_2'-cr_1')+r_1'r_2' = r_1r_2+r_1'r_2'$.
Thus,
\begin{align}
\chi_\r(A) 
&= \ee{\tfrac{(r_1'-r_1) + (r_2'-r_2) + 2r_2r_1' - (r_1r_2+r_1'r_2')}{2}} \\
&= \ee{\tfrac{(r_1'-r_1) + (r_2'-r_2) + (r_1'-r_1)(r_2'-r_2) + (r_2r_1'-r_1r_2'))}{2}}.
\end{align}
Since $\r' \con \r \Mod{1}$, we see that $\ee{\tfrac{(r_1'-r_1) + (r_2'-r_2) + (r_1'-r_1)(r_2'-r_2)}{2}} \in \{\pm 1\}$, and moreover that
\begin{equation}
\ee{\tfrac{(r_1'-r_1) + (r_2'-r_2) + (r_1'-r_1)(r_2'-r_2)}{2}} = -(-1)^{\delta_2(A\r-\r)}.
\end{equation}
Also, $r_2r_1'-r_1r_2' = \symp{\r'}{\r} = \symp{A\r}{\r}$, so 
\begin{equation}\label{eq:chidelta2}
\chi_\r(A) = -(-1)^{\delta_2(A\r-\r)}\ee{\frac{1}{2}\symp{A\r}{\r}}.
\end{equation}
Clearly the right-hand side \eqref{eq:chidelta2} is a $(2N)$-th root of unity. Moreover, when $N$ is odd, we check $\delta_2(A\r-\r) \con 1 \con 1+\symp{A\r}{\r} \Mod{2}$ when $A\r \con \r \Mod{2}$, and $\delta_2(A\r-\r) \con 0 \con 1+\symp{A\r}{\r} \Mod{2}$ when $A\r \not\con \r \Mod{2}$, so in fact $\chi_\r(A)$ is an $N$-th root of unity.
\end{proof}

\begin{lem}
Let $\r \in \Q^2$. For any $\m \in \Z^2$ and any $A \in \Gamma_\r$,
\begin{equation}
\chi_{\r+\m}(A) = \chi_{\r}(A).
\end{equation}
\end{lem}
\begin{proof}
For all $\tau \in \HH$, \Cref{thm:thetamod} gives
\begin{align}
\theta_{\r+\m}(A\cdot\tau) &= \psi(A,\ep)^3\chi_{\r+\m}(A)\ep(\tau)\theta_{\r+\m}(\tau); \\
\theta_\r(A\cdot\tau) &= \psi(A,\ep)^3\chi_\r(A)\ep(\tau)\theta_\r(\tau).
\end{align}
Dividing, we have
\begin{equation}
\frac{\theta_{\r+\m}(A\cdot\tau)}{\theta_\r(A\cdot\tau)} = \frac{\chi_{\r+\m}(A)}{\chi_\r(A)} \cdot \frac{\theta_{\r+\m}(\tau)}{\theta_\r(\tau)}.
\end{equation}
By \Cref{prop:thetacharshift}, $\frac{\theta_{\r+\m}(A\cdot\tau)}{\theta_\r(A\cdot\tau)} = \ee{-m_1\left(r_2+\foh\right)} = \frac{\theta_{\r+\m}(\tau)}{\theta_\r(\tau)}$. Therefore $\chi_{\r+\m}(A) = \chi_{\r}(A)$.
\end{proof}

\begin{rmk}
While the metaplectic transformation law for the modular theta nulls with rational characters is sufficient for proving the properties of the (multiplicative) Shintani--Faddeev modular cocycle, a similar study of additive cocycles would require working out the transformation law for $\log \theta_\r$ under a corresponding congruence subgroup $\widetilde\Gamma_\r$ of $\widetilde{\SL_2(\Z)}$.
\end{rmk}

\subsection{The Jacobi triple product formula}

The Jacobi triple product formula relates a certain $q$-series to a certain infinite product and can be seen as an expression for a Jacobi theta function in terms of infinite $q$-Pochhammer symbols. Beyond its importance in the theory of modular forms, this ubiquitous formula has Lie theoretic, probabilistic, and physical interpretations, among others \cite{macdonald,bb,bfj}.  It is commonly stated as follows.
\begin{thm}\label{thm:jacprodoriginal}
If $w,q \in \C$ with $\abs{q} < 1$, then
\begin{equation}\label{eq:jacprod1}
\sum_{n=-\infty}^\infty w^{n}q^{n^2} = \prod_{k=1}^{\infty} \left(1-q^{2k}\right)\left(1+wq^{2k-1}\right)\left(1+w^{-1}q^{2k-1}\right).
\end{equation}
\end{thm}
\begin{proof}
See \cite[p.\ 231, (100.1)]{rademacher}.
\end{proof}
After a change of variables, the Jacobi triple product formula is equivalent to the following.
\begin{prop}\label{prop:jacprod2}
If $z \in \C$ and $\tau \in \HH$, then
\begin{equation}\label{eq:jacprod2}
\varpi(z,\tau)\varpi(-z,\tau) =
-i \ee{-\tfrac{\tau}{12}}\left(\ee{\tfrac{z}{2}}-\ee{-\tfrac{z}{2}}\right) \frac{\vartheta_1(z,\tau)}{\eta(\tau)}.
\end{equation}
\end{prop}
\begin{proof}
In \Cref{thm:jacprodoriginal}, take $w = \ee{z - \tfrac{\tau}{2} + \foh} = -\ee{z - \tfrac{\tau}{2}}$ and $q = \ee{\tfrac{\tau}{2}}$. Then, the left-hand side of \eqref{eq:jacprod1} becomes
\begin{align}
\sum_{n=-\infty}^\infty w^{n}q^{n^2}
&= \sum_{n=-\infty}^\infty \ee{nz-\tfrac{n}{2}\tau+\tfrac{n}{2}+\tfrac{n^2}{2}\tau} \\
&= \ee{-\tfrac{\tau}{8}+\foh\left(z+\foh\right)} \sum_{n=-\infty}^\infty \ee{\foh\left(n-\foh\right)^2\tau + \left(n-\foh\right)\left(z+\foh\right)} \\
&= -i\ee{-\tfrac{\tau}{8}+\tfrac{z}{2}} \vartheta_1(z,\tau). \label{eq:jacprodpartone}
\end{align}
The right-hand side of \eqref{eq:jacprod1} becomes
\begin{align}
&\prod_{k=1}^{\infty} \left(1-q^{2k}\right)\left(1+wq^{2k-1}\right)\left(1+w^{-1}q^{2k-1}\right) \\
&= \prod_{k=1}^{\infty} \left(1-\ee{k\tau}\right) \left(1-\ee{z-\tfrac{\tau}{2}+(2k-1)\tfrac{\tau}{2})}\right) \left(1-\ee{-z+\tfrac{\tau}{2}+(2k-1)\tfrac{\tau}{2})}\right) \\
&= 
\prod_{k=1}^{\infty} \left(1-\ee{k\tau}\right) \times
\prod_{k=1}^{\infty} \left(1-\ee{z+(k-1)\tau)}\right) \times 
\prod_{k=1}^{\infty} \left(1-\ee{-z+(k+1)\tau)}\right) \\
&= 
\left(\frac{\eta(\tau)}{\ee{\tfrac{\tau}{24}}}\right)
\left(\varpi(z,\tau)\right)
\left(\frac{\varpi(-z,\tau)}{1-\ee{-z}}\right) \\
&= \frac{\ee{-\tfrac{\tau}{24} + \tfrac{z}{2}}}{\ee{\tfrac{z}{2}}-\ee{-\tfrac{z}{2}}}\eta(\tau)\varpi(z,\tau)\varpi(-z,\tau). \label{eq:jacprodparttwo}
\end{align}
Equating \eqref{eq:jacprodparttwo} and \eqref{eq:jacprodpartone} and multiplying/dividing/rearranging factors, we get \eqref{eq:jacprod2}.
\end{proof}
The following version of the triple product formula incorporates ``characteristics'' $\r$ and will be used in \Cref{sec:further} to prove important identities for the Shintani--Faddeev cocycles.
\begin{prop}\label{prop:jacprod3}
If $\r \in \R^2$, $z \in \C$, and $\tau \in \HH$, let
\begin{equation}
\varpi_\r(z,\tau) = \varpi(z + \sympt{\r}{\tau},\tau) = \prod_{k=0}^\infty \left(1-\ee{z+(k+r_2)\tau-r_1}\right). 
\end{equation}
Then
\begin{align}
\varpi_\r(z,\tau)\varpi_{-\r}(-z,\tau)
&= i \ee{-\left(\tfrac{r_2^2}{2}+\tfrac{1}{12}\right)\tau-r_2\left(z-r_1+\foh\right)} \\
& \ \ \ \times \left(\ee{\tfrac{z+r_2\tau-r_1}{2}}-\ee{\tfrac{-z-r_2\tau+r_1}{2}}\right) \frac{\vartheta_\r(z,\tau)}{\eta(\tau)}.
\label{eq:jacprod3}
\end{align}
\end{prop}
\begin{proof}
Make the substitution $z \mapsto z+r_2\tau-r_1$ in \Cref{prop:jacprod2}. We have
\begin{equation}
\varpi_\r(z,\tau)\varpi_{-\r}(-z,\tau) =
-i \ee{-\tfrac{\tau}{12}}\left(\ee{\tfrac{z+\sympt{\r}{\tau}}{2}}-\ee{-\tfrac{z+\sympt{\r}{\tau}}{2}}\right) \frac{\vartheta_1(z+\sympt{\r}{\tau},\tau)}{\eta(\tau)}.
\end{equation}
Using the relation $\vartheta_\r(z,\tau) = - \ee{\foh r_2^2\tau+r_2(z-r_1+\foh)}\vartheta_1(z+\sympt{\r}{\tau},\tau)$, 
\begin{align}
\varpi_\r(z,\tau)\varpi_{-\r}(-z,\tau)
&= i \ee{-\tfrac{\tau}{12}}\left(\ee{\tfrac{z+r_2\tau-r_1}{2}}-\ee{\tfrac{-z-r_2\tau+r_1}{2}}\right) \\
& \ \ \ \times \ee{-\foh r_2^2\tau-r_2(z-r_1+\foh)} \frac{\vartheta_\r(z,\tau)}{\eta(\tau)}, 
\end{align}
which can be algebraically simplified to \eqref{eq:jacprod3}.
\end{proof}

\subsection{Modular specializations of the $q$-Pochhamer symbol}\label{sec:modspecial}

The following formulas can be proven directly by manipulation of the infinite products. Taking $q=\ee{\tau}$:
\begin{align}
\varpi_{\tiny\smcoltwo{0}{1}}(\tau) &= (q,q)_\infty &&= q^{-1/24} \eta(\tau); \\
\varpi_{\tiny\smcoltwo{0}{1/2}}(\tau) &= \frac{(q^{1/2},q^{1/2})_\infty}{(q,q)_\infty} &&= q^{1/48} \frac{\eta(\tau/2)}{\eta(\tau)}; \\
\varpi_{\tiny\smcoltwo{1/2}{1}}(\tau) &= \frac{(q^{2},q^{2})_\infty}{(q,q)_\infty} &&= q^{-1/24} \frac{\eta(2\tau)}{\eta(\tau)}; \\
\varpi_{\tiny\smcoltwo{1/2}{1/2}}(\tau) &= \frac{(-q^{1/2},-q^{1/2})_\infty}{(q,q)_\infty} &&= \zeta_{48}q^{1/48} \frac{\eta((\tau+1)/2)}{\eta(\tau)}. \label{eq:modspecial}
\end{align}
Thus, we see that $\varpi_\r(\tau)$ is a weak modular function on $\Gamma(2)$ with character when $\r \in \foh\Z^2$. Conversely, when $\r \nin \foh\Z^2$, $\varpi_\r(\tau)$ is not modular (although we haven't proven non-modularity rigorously). Nonetheless, $\varpi_\r(\tau)$ satisfies a more complicated modular-like property, which will be examined in \Cref{sec:wannabe}.

\section{On moduli spaces of ray class data}\label{sec:moduli}

It will be helpful to have an interpretation of the ``ray class data'' determining a Stark unit as a point on some continuous ``moduli space.'' We use the phrase ``moduli space'' in a loose sense. In particular, our moduli space will be described as the quotient of $(\R/\Z)^2 \times (\R \setminus \Q)$ by an action of discrete group that is not properly discontinuous. Ray class data will correspond to the dense set of points $(\Q/\Z)^2 \times \Rquad$.

\subsection{Ray class groups and ray class fields of orders}\label{sec:rayclassorder}

We briefly review some definitions and results from \cite{kopplagarias} that generalize standard results about ray class groups and ray class fields to non-maximal orders. Let $\OO$ be an order in a number field $F$, let $\mm$ be an ideal in $\OO$, and let $\rS$ be a subset of the real embeddings of $F$.
\begin{defn}
The \textit{ray class group of the order $\OO$ modulo $(\mm, \rS)$} is
\begin{equation}
\Cl_{\mm,\rS}(\OO) = \frac{\rJ_{\mm}^\ast(\OO)}{\rP_{\mm,\rS}(\OO)},
\end{equation}
where
\begin{align}
\rJ_{\mm}^\ast(\OO) &= \{\mbox{invertible fractional ideals of $\OO$ coprime to $\mm$}\}, \mbox{ and} \\
\rP_{\mm,\rS}(\OO) &= \{\alpha\OO \mbox{ such that } \alpha \con 1 \Mod{\mm} \mbox{ and } \rho(\alpha)>0 \mbox{ for } \rho \in \rS\}.
\end{align}
\end{defn}

The study of the structure of ray class groups naturally leads to the study of certain groups of units satisfying congruence conditions.
\begin{defn}\label{defn:ugroups}
For a commutative ring with unity $R$ and an ideal $I$ of $R$, define the group
\begin{equation}
\U_I(R) := \{\alpha \in R^\times : \alpha \equiv 1 \bmod{I}\} = (1+I) \cap R^\times.
\end{equation}

If $R$ has real embeddings and $\rS$ is a subset of the real embeddings of $R$, define
\begin{equation}
\U_{I,\rS}(R) := \{\alpha \in R^\times : \alpha \equiv 1 \bmod{I} \mbox{ and } \rho(\alpha)>0 \mbox{ for } \rho \in \rS\}.
\end{equation}
\end{defn}

A key theorem of \cite{kopplagarias} relates different class groups of orders to each other via a surjective map whose quotient is described using the $\U$-groups.
\begin{thm}\label{thm:exseq}
Let $F$ be a number field and $\OO \subseteq \OO' \subseteq \OO_F$ be orders of $F$.
Let $\mm$ be an ideal of $\OO$, $\mm'$ an ideal of $\OO'$ such that $ \mm\OO' \subseteq \mm'$, and $\rS' \subseteq \rS \subseteq \{\mbox{embeddings } F \inj \R\}$. 
Let $\dd$ be any $\OO'$-ideal such that $\dd \subseteq \colonideal{\mm}{\OO'}$.
We have the following exact sequence.
\begin{equation}\label{eq:exseq}
1 \to \frac{\U_{\mm',\rS'}\!\left(\OO'\right)}{\U_{\mm,\rS}\!\left(\OO\right)} \to \frac{\U_{\mm'}\!\left(\OO'/\dd\right)}{\U_{\mm}\!\left(\OO/\dd\right)} \times \{\pm 1\}^{|\rS \setminus \rS'|} \to \Cl_{\mm,\rS}(\OO) \to \Cl_{\mm',\rS'}(\OO') \to 1.
\end{equation}
\end{thm}
\begin{proof}
See \cite[Thm.\ 6.5]{kopplagarias}.
\end{proof}

To the ray class group of an order, there is associated a \textit{ray class field of the order}, denoted $H_{\mm,\rS}^\OO$. Some of its important properties are summarized by the following theorem.
\begin{thm}\label{thm:main1} 
Let $F$ be a number field, $\OO$ an order of $F$, $\mm$ an ideal of $\OO$, and $\rS$ a subset 
of the set of real embeddings of $F$. 
Then there exists a unique abelian Galois extension 
$H_{\mm,\rS}^{\OO}/F$ with the property that a prime ideal $\pp$ of $\OO_F$ that is 
coprime to the quotient ideal $\colonideal{\mm}{\OO_F}$ 
splits completely in $H_{\mm,\rS}^{\OO}/F$ if and only if $\pp \cap \OO = \pi\OO$, a principal prime $\OO$-ideal having $\pi \in \OO$ with $\pi \equiv 1 \Mod{\mm}$ and $\rho(\pi)>0$ for $\rho \in \rS$.

Additionally, these fields have the following properties:
\begin{itemize}
\item $H_{\mm\OO_F,\rS}^{\OO_F} \subseteq H_{\mm,\rS}^{\OO} \subseteq H_{\colonideal{\mm}{\OO_F},\rS}^{\OO_F}$.
\item There is a canonical isomorphism $\Art_{\OO} : \Cl_{\mm, \rS}(\OO) \to \Gal\!\left(H_{\mm,\rS}^{\OO}/F\right)$.
\end{itemize}
\end{thm}
\begin{proof}
See \cite[Thm.\ 1.1, Thm.\ 1.2, Thm.\ 1.3]{kopplagarias}.
\end{proof}

We provide one new proposition about ray class groups of orders that will be relevant to our moduli interpretation.
\begin{prop}\label{prop:oopiso}
Let $\OO \subseteq \OO'$ be an orders in the same number field $F$. Let $\mm$ be an $\OO'$-ideal (thus also an $\OO$-ideal) and $\rS$ a subset of the real embeddings of $F$. 
Then the map $\ol\ext = \ol\ext_{(\OO;\mm,\rS)}^{(\OO';\mm,\rS)}: \Cl_{\mm,\rS}(\OO) \to \Cl_{\mm,\rS}(\OO')$ induced by ideal extension $\ext(\aa) = \aa\OO'$ is an isomorphism.
\end{prop}
\begin{proof}
By \Cref{thm:exseq}
with $\dd = \colonideal{\mm}{\OO'}=\mm$ and $\mm' = \mm\OO'=\mm$, there is an exact sequence
\begin{equation}
1
\to
\frac{\U_{\mm,\rS}(\OO')}{\U_{\mm,\rS}(\OO)}
\to 
\frac{\U_{\mm}(\OO'/\mm)}{\U_{\mm}(\OO/\mm)}
\to 
\Cl_{\mm,\rS}(\OO)
\to
\Cl_{\mm,\rS}(\OO')
\to 
1.
\end{equation}
The group $\U_{\mm}(\OO'/\mm) = \{\alpha \in (\OO'/\mm)^\times : \alpha \con 1 \Mod{\mm}\}$ is the trivial group, so $\frac{\U_{\mm}(\OO'/\mm)}{\U_{\mm}(\OO/\mm)}$ is the trivial group, and thus the map from $\Cl_{\mm,\rS}(\OO)$ to $\Cl_{\mm,\rS}(\OO')$ is an isomorphism.
\end{proof}

\subsection{The flat imprimitive ray class monoid}\label{sec:monoid}

In \cite{kopplagariasmonoids}, the present author and Lagarias describe several different \textit{ray class monoids} that extend the usual definition of the ray class group to the structure of a larger (but still finite) monoid (semigroup with identity). The present paper requires one of these constructions in particular, the \textit{\rcmia}, which extends the ray class group to include classes of ideals not coprime to the modulus.

We extend the ray class group to a monoid by extending $\rJ^\ast_\mm(\OO)$ to a large monoid $\rJf_\mm(\OO)$ still consisting of $\OO$-invertible ideals, but relaxing the condition of coprimality to $\mm$ to a condition of \textit{semilocal integrality at $\mm$}. We use the term \textit{flat} (and the corresponding musical symbol) because nonzero $\OO$-ideals are invertible if and only if they are flat as $\OO$-modules, and to avoid the ambiguity of the term ``invertible'' (as ideals in $\rJf_\mm(\OO)$ not coprime to $\mm$ are invertible as $\OO$-ideals but not invertible in the monoid $\rJf_\mm(\OO)$).
\begin{defn}
We define the following submonoid of the group of invertible ideals:
\begin{align}
\rJf_\mm(\OO) &= \{\aa \in \rJ^\ast(\OO) : \aa\OO[S_\mm^{-1}] \subseteq \OO[S_\mm^{-1}]\}.
\end{align}
The condition that $\aa\OO[S_\mm^{-1}] \subseteq \OO[S_\mm^{-1}]$ is equivalent to the condition that $\aa\OO_\pp \subseteq \OO_\pp$ for all nonzero prime ideals $\pp \subseteq \mm$; we call this condition \textit{semilocal integrality at $\mm$}.
Consider the equivalence relation $\sim_{\mm,\rS}$ on $\rJf_\mm(\OO)$ defined by
\begin{equation}
\aa \sim_{\mm,\rS} \bb \iff \begin{array}{c}\exists \cc \in \rJf_{\mm}(\OO) \mbox{ and } \alpha, \beta \in \OO[S_\mm^{-1}] \mbox{ such that } \aa = \alpha\cc, \bb = \beta\cc, \\ \alpha \equiv \beta \Mod{\mm}, \sgn(\rho(\alpha)) = \sgn(\rho(\beta)) \mbox{ for all } \rho \in \rS.\end{array}
\end{equation}
The \textit{\rcmia} is
\begin{equation}
\Clt_{\mm,\rS}(\OO) = \frac{\rJf_\mm(\OO)}{\sim_{\mm,\rS}}.
\end{equation}
Classes in the image of the map $\Cl_{\mm,\rS}(\OO) \inj \Clt_{\mm,\rS}(\OO)$ are called \textit{primitive}, and other classes are called \textit{imprimitive}.
The \textit{submonoid of zero classes} is
\begin{equation}
\ZClt_{\mm,\rS}(\OO) = \{[\dd] \in \Clt_{\mm,\rS}(\OO) : \dd \subseteq \mm\}.
\end{equation}
\end{defn}

We now show that the ray class group embeds into the flat imprimitive ray class monoid in the expected manner.
\begin{prop}
The inclusion map $\rJ_\mm^\ast(\OO) \inj \rJf_\mm(\OO)$ induces an injection of monoids $\Cl_{\mm,\rS}(\OO) \inj \Clt_{\mm,\rS}(\OO)$.
\end{prop}
\begin{proof}
Consider $\aa, \bb \in \rJ_\mm(\OO)$. We wish to show that $\aa$ is equivalent to $\bb$ in $\Cl_{\mm,\rS}(\OO)$ if and only if $\aa \sim_{\mm,\rS} \bb$.

If $\aa$ is equivalent to $\bb$ in $\Cl_{\mm,\rS}(\OO)$, then there is some $\gamma\OO \in \rP_{\mm, \rS}(\OO)$ such that $\aa = \gamma\bb$, $\gamma \equiv 1 \Mod{\mm}$, and $\rho(\gamma)>0$ for all $\rho \in \rS$. It follows that $\aa \sim_{\mm,\rS} \bb$ by taking $\cc = \bb$ and $(\alpha,\beta) = (1,\gamma)$.

Conversely, suppose $\aa \sim_{\mm,\rS} \bb$. Then there are some $\cc \in \rJf_{\mm}(\OO)$ and $\alpha, \beta \in \OO[S_\mm^{-1}]$ such that $\aa = \alpha\cc$, $\bb = \beta\cc$, $\alpha \equiv \beta \Mod{\mm}$, and $\sgn(\rho(\alpha)) = \sgn(\rho(\beta))$ for all $\rho \in \rS$. Since $\aa, \bb$ are coprime to $\mm$, it follows (from the equations $\aa=\alpha\cc$ and $\bb=\beta\cc$ and the semilocal integrality of $\alpha\OO$, $\beta\OO$, and $\cc$) that $\alpha\OO$, $\beta\OO$, and $\cc$ are also coprime to $\mm$. Thus, $\alpha\beta^{-1} \equiv 1 \Mod{\mm}$. Also, $\aa = (\alpha\beta^{-1})\bb$, and $\rho(\alpha\beta^{-1})$ for all $\rho \in \rS$, so $\aa$ and $\bb$ are equivalent in $\Cl_{\mm,\rS}(\OO)$.
\end{proof}

To facilitate describing the structure of the monoid $\Clt_{\mm,\rS}(\OO)$, we define a suitable notion of an exact sequence of commutative monoids. 
As in an exact sequence of abelian groups, we want the fibers of the latter map to be cosets of the image of the former; to guarantee this property, we impose it directly, because it is not sufficient to say that the image of the former map is the kernel of the latter.

\begin{defn}
A sequence
\begin{equation}
\cdots \to A \xrightarrow{\alpha} B \xrightarrow{\beta} C \to \cdots
\end{equation}
of commutative monoids with homomorphisms between them is \textit{exact at $B$} 
if every nonempty preimage of an element $C$ under $\beta$ is a coset of an image of $\alpha$; that is, if for all $c \in C$, either $\beta^{-1}(c) = \emptyset$, or there exists $b \in B$ such that
\begin{equation}
b\alpha(A) = \beta^{-1}(c).
\end{equation}
A sequence that is exact at all objects with an in-arrow and out-arrow is simply called \textit{exact}.
\end{defn}

We prove a proposition ``resolving'' the map from $\Clt_{\mm,\rS}(\OO)$ to $\Cl(\OO)$ in order to understand the structure of $\Clt_{\mm,\rS}(\OO)$.

\begin{prop}\label{prop:exmonoid}
Let $\phi : \Clt_{\mm,\rS}(\OO) \to \Cl(\OO)$ be the map given by $\phi([\bb]) = [\bb]$. Then, there is an exact sequence of monoids
\begin{equation}
\left(\OO/\mm,\times\right) \times \{\pm 1\}^\rS \xrightarrow{\psi} \Clt_{\mm,\rS}(\OO) \xrightarrow{\phi} \Cl(\OO) \to 1.
\end{equation}
\end{prop}
\begin{proof}
Exactness at $\Cl(\OO)$ is equivalent to the surjectivity of $\phi$. By \cite[Lem.~5.12]{kopplagarias} (taking $\dd=\mm$), every class in $\Cl(\OO)$ is represented by some $\bb \in \rJ_{\mm}^\ast(\OO)$, and $\rJ_{\mm}^\ast(\OO) \subseteq \rJf_{\mm}(\OO)$, so $\phi$ is surjective.

Define the map $\psi : \left(\OO/\mm,\times\right) \times \{\pm 1\}^\rS \to \Clt_{\mm,\rS}(\OO)$ by
\begin{equation}
\psi(\ol{\alpha},\epsilon) = [\alpha\OO] \mbox{ where } \alpha \equiv \ol{\alpha} \Mod{\mm} \mbox{ and } \sgn(\rho(\alpha)) = \epsilon_\rho.
\end{equation}
This map is well-defined because:
\begin{itemize}
\item[(i)] For any pair $(\ol{\alpha},\epsilon)$, the set $(\ol{\alpha} + \mm) \cap \{\alpha \in \OO : \sgn(\rho(\alpha)) = \epsilon_\rho\} \neq \emptyset$.
\item[(ii)] If $\alpha_1, \alpha_2 \in \OO \setminus \{0\}$, $\alpha_1 \equiv \alpha_2 \Mod{\mm}$, and $\sgn(\rho(\alpha_1))=\sgn(\rho(\alpha_2))$, then $[\alpha_1\OO] = [\alpha_2\OO]$ in $\Clt_{\mm,\rS}(\OO)$.
\end{itemize}

To prove exactness at $\Clt_{\mm,\rS}(\OO)$, consider a class $\BB \in \Cl(\OO)$. By \cite[Lem.~5.12]{kopplagarias}, we may write $\BB = [\bb]$ for some $\bb \in \rJ^\ast_\mm(\OO)$. Clearly $\phi([\alpha\bb]) = \BB$ for any $\alpha \in \OO[S_\mm^{-1}] \setminus \{0\}$, so $\{[\alpha\bb] : \alpha \in \OO[S_\mm^{-1}] \setminus \{0\}\} \subseteq \phi^{-1}(\BB)$. On the other hand, suppose $\aa \in \rJf_{\mm}(\OO)$ such that $\phi([\aa]) = \BB$. Then $\aa$ is equivalent to $\bb$ in $\Cl(\OO)$, so $\aa = \alpha\bb$ for some $\alpha \in F^\times$. Moreover, $\alpha\OO = \aa\bb^{-1} \in \rJf_{\mm}(\OO)$, that is, $\alpha\OO$ is semilocally integral at $\mm$, so $\alpha \in \OO[S_\mm^{-1}] \setminus \{0\}$. Therefore, $\{[\alpha\bb] : \alpha \in \OO[S_\mm^{-1}] \setminus \{0\}\} = \phi^{-1}(\BB)$. The left-hand set is the same as $[\bb]\phi\!\left(\left(\OO/\mm,\times\right) \times \{\pm 1\}^\rS\right)$, so we have proven that the sequence is exact at $\Clt_{\mm,\rS}(\OO)$.
\end{proof}

The exact sequence in \Cref{prop:exmonoid} is related to the ray class group by the following commutative diagram, where in both rows the image of the first map consists of the classes of principal ideals.
\begin{equation}
\begin{tikzcd}[column sep=small] 
& \left(\OO/\mm\right)^\times \times \{\pm 1\}^\rS \arrow[r] \arrow[d, hook] 
& \Cl_{\mm,\rS}(\OO) \arrow[r] \arrow[d, hook] 
& \Cl(\OO) \arrow[r] \arrow[d, equal] 
& 1
\\ 
& \left(\OO/\mm,\times\right) \times \{\pm 1\}^\rS \arrow[r, "\psi"] 
& \Clt_{\mm,\rS}(\OO) \arrow[r, "\phi"]
& \Cl(\OO) \arrow[r]
& 1
\\
\end{tikzcd}
\end{equation}
The monoid of zero classes has the properties that $\image(\psi) \cap \ZClt_{\mm,\rS}(\OO) = \{(0,\ep)\}$ (which does not depend on the choice of $\ep$), and $\phi$ restricts to an isomorphism $\ZClt_{\mm,\rS}(\OO) \isom \Cl(\OO)$.

\begin{egz}
This example shows that, in contrast to the case of the ray class group seen in \Cref{prop:oopiso}, the surjective monoid homomorphism $\ol\ext = \ol\ext_{(\OO;\mm,\rS)}^{(\OO';\mm,\rS)}: \Clt_{\mm,\rS}(\OO) \to \Clt_{\mm,\rS}(\OO')$ induced by extension of ideals $\ext(\aa) = \aa\OO'$ need not be an isomorphism when $\mm$ is an $\OO'$-ideal. Let $\OO = \Z[3\sqrt{3}]$ and $\OO' = \Z[\sqrt{3}]$, and consider the unit $\e = 2+\sqrt{3}$. Let $\aa = 3\OO$ and $\bb = 3\e\OO$. Then, $\ext(\aa)=\ext(\bb)=3\OO'$.

We show by contradiction that $\aa \not\sim_{9\OO',\emptyset} \bb$. If $\aa \sim_{9\OO',\emptyset} \bb$, then we would have $\beta\aa = \alpha\bb$ for $\alpha \equiv \beta \Mod{9\OO'}$. We obtain $3\beta\OO = 3\e\alpha\OO$, so $\e\alpha\beta^{-1} \in \OO^\times = \langle -1, \e^3 \rangle$, and thus $\alpha = \pm \e^{3n+1}\beta$ for some $n \in \Z$, so $\pm \e^{3n+1}\beta \equiv \beta \Mod{9\OO'}$. We must have $3\e\OO = \bb = \beta\cc$ for $\cc \in \rJf_{9\OO'}(\OO)$, so $3\OO'=\beta\cc\OO'$, and we must have $\beta \div 3$ in $\OO' = \Z[\sqrt{3}]$, that is, $\beta\gamma = 3$ for some $\gamma \in \OO'$. Multiplying both sides of the congruence by $\gamma$, we obtain $\pm 3\e^{3n+1} \equiv 3 \Mod{9\OO'}$, so $\pm \e^{3n+1} \equiv 1 \Mod{3\OO'}$. Note that $\e^3 = 26+15\sqrt{3} \equiv 1 \Mod{3\OO'}$, so the congruence simplifies to $\pm \e \equiv 1 \Mod{3\OO'}$. But, since $\e = 2+\sqrt{3}$, we obtain a contradiction.
\end{egz}

\subsection{Key properties of orders}

We now recall a few basic results and definitions regarding orders of number fields.

\begin{prop}\label{prop:locallyprincipal}
If $\OO$ is an order in a number field and $\aa$ is a fractional $\OO$-ideal, the following are equivalent:
\begin{itemize}
\item[(1)] $\aa$ is invertible as an $\OO$-ideal.
\item[(2)] For every nonzero prime $\pp$ of $\OO$, the localization $\aa_\pp := \aa\OO_\pp$ is a principal $\OO_\pp$-ideal.
\end{itemize}
\end{prop}
\begin{proof}
See \cite[Prop.~3.8]{kopplagarias} or \cite[Cor.~2.1.7]{dtz}.
\end{proof}

\begin{defn}
If $\OO$ is an order in a number field $F$ and $\aa$ is a fractional $\OO$-ideal, the \textit{multiplier ring} (or \textit{multiplier order}) of $\aa$ is
\begin{equation}
\ord(\aa) := \colonideal{\aa}{\aa} = \{x \in F : x\aa \subseteq \aa\}.
\end{equation}
\end{defn}

\begin{prop}\label{prop:ordinvertible}
If $\OO$ is an order in a quadratic field, $\aa$ is a fractional $\OO$-ideal, and $\OO' = \ord(\aa)$, then $\aa$ is an invertible fractional $\OO'$-ideal.
\end{prop}
\begin{proof}
See \cite[p.~557]{jt}.
\end{proof}

\subsection{The main correspondence}

We will now describe classes in ray class groups of real quadratic fields as corresponding to special ``real multiplication'' points 
\begin{equation}
\Q^2/\Z^2 \times (F \setminus \Q) \xhookrightarrow{\rho_1} \Q^2/\Z^2 \times \Rquad \subset \R^2/\Z^2 \times \R
\end{equation}
modulo an action of $\SL_2(\Z)$ or $\GL_2(\Z)$.
We will use this correspondence is to relate Stark units (more specifically, Stark--Tangedal--Yamamoto invariants), attached to objects on the left-hand side of \eqref{eq:correspondence}, to RM values of Shintani--Faddeev modular cocycles, attached to objects on the right-hand side of \eqref{eq:correspondence}.
A similar correspondence can be given in the complex case, replacing $\Rquad$ by the set $\Hquad$ of quadratic numbers in the upper half-plane; a notable difference in the real case is that the action of $\SL_2(\Z)$ on $\Q^2/\Z^2 \times \Rquad$ is not totally discontinuous.
These correspondences are also related to Gauss composition for quadratic forms, as will be explored further in \cite{bk}.

\begin{thm}\label{thm:correspondence}
Let $\OO$ be an order in a real quadratic field $F$, and let $\Fquad = F \setminus \Q$.
Let $\mm$ be a nonzero $\OO$-ideal and $\OO' = \ord(\mm)$.
There are explicit compatible functions
\begin{equation}\label{eq:correspondence}
\begin{tikzcd}
\Clt_{\mm\infty_1\infty_2}(\OO) \ar[r, "\tilde\Upsilon_\mm"] \ar[d] & \SL_2(\Z) \backslash \!\left(\Q^2/\Z^2 \times \Fquad\right) \ar[d] \\
\Clt_{\mm\infty_2}(\OO) \ar[r, "\Upsilon_\mm"] & \GL_2(\Z) \backslash \!\left(\Q^2/\Z^2 \times \Fquad\right)
\end{tikzcd}
\end{equation}
where the action of $\GL_2(\Z)$ on $\Q^2/\Z^2 \times \Fquad$ is $A\cdot(\r,\beta) = (s_{\!A}(\beta)A\r,A\cdot\beta)$, the quantity $s_{\!A}(\beta) = \sgn(\rho_1(j_{A}(\beta))$, the notation $\GL_2(\Z) \backslash \!\left(\Q^2/\Z^2 \times F\right)$ denotes the set of orbits of this right action (and similarly for $\SL_2(\Z)$), and the downward maps are the obvious quotient maps.
If $\OO' = \OO$, then $\tilde\Upsilon_\mm$ and $\Upsilon_\mm$ are injective.
Generally, the image of $\tilde\Upsilon_\mm$ is described as
\begin{align}
\image(\tilde\Upsilon_{\mm}) &= \SL_2(\Z) \backslash \M_{\OO',\mm}, \mbox{ where} \\
\M_{\OO',\mm} &= \{(\r,\beta) \in \Q/\Z \times F_{\OO'} : \sympt{\r}{\beta}\mm \subseteq \beta\Z+\Z\}, \mbox{ and} \\
F_{\OO'} &= \{\beta \in F : \ord(\beta\Z+\Z) = \OO'\}.
\end{align}
In particular, for $m \in \N$,
\begin{equation}
\image(\tilde\Upsilon_{m\OO}) = \SL_2(\Z) \backslash \!\left(\tfrac{1}{m}\Z^2/\Z^2 \times F_{\OO}\right).
\end{equation}
These functions factor through the monoid maps induced by extension of ideals: $\tilde\Upsilon_\m(\A) = \tilde\Upsilon_\m\!\left(\ol\ext_{(\OO;\mm\infty_1\infty_2)}^{(\OO';\mm\infty_1\infty_2)}(\A)\right)$; $\Upsilon_\m(\A) = \Upsilon_\m\!\left(\ol\ext_{(\OO;\mm\infty_2)}^{(\OO';\mm\infty_2)}(\A)\right)$. Additionally, the zero classes are
$\ZClt_{\mm\infty_1\infty_2}(\OO) = \tilde\Upsilon_\m^{-1}(\{\mathbf{0}\} \times \Fquad)$.

The function $\tilde\Upsilon_{\mm}$ is described as follows: 
Given $\A \in \Clt_{\mm\infty_1\infty_2}(\OO)$, let $\A_0$ be the class of $\A$ in the narrow class group $\Cl_{\infty_1\infty_2}(\OO)$. Choose an integral ideal $\bb \in \A_0^{-1}$ that is coprime to $\mm$.
Express $\bb\mm = \alpha(\beta\Z+\Z)$ with $\rho_1(\alpha), \rho_2(\alpha) > 0$ and $\rho_1(\beta)>\rho_2(\beta)$.
Choose a representative $\gamma\OO$ of $\bb\A$ such that $\gamma \in \bb$ and $\rho_1(\gamma),\rho_2(\gamma)>0$, and write $\gamma = \alpha\sympt{\r}{\beta}$ for some $\r \in \Q^2$.
Set
\begin{equation}
\tilde\Upsilon_\m(\A) = \SL_2(\Z) \cdot (\r, \beta).
\end{equation}
The function $\Upsilon_\m(\A)$ is then defined by $\Upsilon_\m(\A) = \GL_2(\Z) \cdot \tilde\Upsilon_\m(\tilde\A)$ for any choice of lift of $\A$ to $\tilde\A \in \Clt_{\mm\infty_1\infty_2}(\OO)$.
\end{thm}
\begin{proof}
We will first show that the map $\tilde\Upsilon_\mm$ is well-defined, that is, that it does not depend on the choices of $\bb$, $\alpha$, $\beta$, or $\gamma$. Consider two such tuples of choices $(\bb_1, \alpha_1, \beta_1, \gamma_1)$ and $(\bb_2, \alpha_2, \beta_2, \gamma_2)$, as well as corresponding $\r_1$ and $\r_2$ such that $\gamma_j = \alpha_j\sympt{\r_j}{\beta_j}$. 
We have $\bb_j\mm = \alpha_j(\beta_j\Z+\Z)$. 
Moreover, $\bb_1\mm$ and $\bb_2\mm$ are both in the ideal class $\mm\A_0^{-1}$, so there exists some $\delta \in \OO[S_\mm^{-1}]^\times$ such that $\alpha_1(\beta_1\Z+\Z) = \delta\alpha_2(\beta_2\Z+\Z)$ 
and $\rho_1(\delta), \rho_2(\delta) > 0$.
Thus, there are integers $a, b, c, d, a', b', c', d'$ such that
\begin{align}
\smcoltwo{\alpha_1\beta_1}{\alpha_1} = \smmattwo{a}{b}{c}{d}\smcoltwo{\delta\alpha_2\beta_2}{\delta\alpha_2}
& & \mbox{ and } & &
\smcoltwo{\delta\alpha_2\beta_2}{\delta\alpha_2} = \smmattwo{a'}{b'}{c'}{d'}\smcoltwo{\alpha_1\beta_1}{\alpha_1}.
\end{align}
Thus, the matrices $\smmattwo{a}{b}{c}{d}$ and $\smmattwo{a'}{b'}{c'}{d'}$ are inverses of each other in $\GL_2(\Z)$, and $\beta_1 = \smmattwo{a}{b}{c}{d} \cdot \beta_2$. Set $A = \smmattwo{a}{b}{c}{d}$. We have
\begin{align}
\rho_1(\beta_1)- \rho_2(\beta_1)
&= A\cdot\rho_1(\beta_2) - A\cdot\rho_2(\beta_2) \\
&= \frac{\det(A)\left(\rho_1(\beta_2)-\rho_2(\beta_2)\right)}{\Nm(c\beta+d)} \\
&= \frac{\det(A)\left(\rho_1(\beta_2)-\rho_2(\beta_2)\right)}{\Nm(\delta^{-1}\alpha_1\alpha_2^{-1})},
\end{align}
and $\rho_1(\beta_1)- \rho_2(\beta_1)$, $\rho_1(\beta_2)-\rho_2(\beta_2)$, and $\Nm(\delta^{-1}\alpha_1\alpha_2^{-1})$ are all positive, so $\det(A) = 1$, and $A \in \SL_2(\Z)$.

Since $\bb_1\mm = \delta\bb_2\mm$ and every ideal of a quadratic order is invertible in its multiplier ring, it follows that $\bb_1\OO' = \delta\bb_2\OO'$. Moreover, if $\ff = \colonideal{\OO}{\OO'}$ is the relative conductor, then $\mm \supseteq \ff$; thus, \cite[Prop.\ 4.8]{kopplagarias} says that the extension map $\ext(\aa) = \aa\OO'$ defines an isomorphism $\rJ_{\mm}(\OO) \to \rJ_{\mm}(\OO')$ on fractional ideals coprime to $\mm$. Since $\bb_1$ and $\delta\bb_2$ are coprime to $\mm$, it follows that we can ``cancel'' the factor of $\OO'$ and obtain $\bb_1 = \delta\bb_2$. Equivalently, $\delta^{-1}\bb_1 = \bb_2$.

We are given that $\gamma_1\OO \in \bb_1\A$, so $\delta^{-1}\gamma_1\OO \in \delta^{-1}\bb_1\A = \bb_2\A$; we are also given that $\gamma_2\OO \in \bb_2\A$. So $\delta^{-1}\gamma_1\OO$ and $\gamma_2\OO$ belong to the same class in $\Clt_{\mm\infty_1\infty_2}(\OO)$; that is, there is some global unit $\e \in \OO$ such that $\e\delta^{-1}\gamma_1 - \gamma_2 \in \mm\OO[S_\mm^{-1}]$ and $\sgn(\rho_i(\e\delta^{-1}\gamma_1))=\sgn(\rho_i(\gamma_2))$ for $i \in \{1,2\}$. As $\delta, \gamma_1, \gamma_2$ are positive at both real places, it follows that $\rho_i(\e) > 0$ for $i \in \{1,2\}$.
We may write
\begin{equation}
\e\smcoltwo{\alpha_2\beta_2}{\alpha_2} = E\smcoltwo{\alpha_2\beta_2}{\alpha_2} \mbox{ for some } E \in \SL_2(\Z).
\end{equation}
We then have $\e\delta^{-1}\smcoltwo{\alpha_1\beta_1}{\alpha_1} = \e A \smcoltwo{\alpha_2\beta_2}{\alpha_2} = A E \smcoltwo{\alpha_2\beta_2}{\alpha_2}$. Write $AE = \smmattwo{e}{f}{g}{h}$.
Thus,
\begin{align}
\e\delta^{-1}\gamma_1 - \gamma_2
&= \e\delta^{-1}(r_{12}\alpha_1\beta_1 - r_{11}\alpha_1) - (r_{22}\alpha_2\beta_2 - r_{21}\alpha_2) \\
&= \left(r_{12}(e\alpha_2\beta_2+f\alpha_2) - r_{11}(g\alpha_2\beta_2+h\alpha_2)\right) - (r_{22}\alpha_2\beta_2 - r_{21}\alpha_2) \\
&= \left(-gr_{11}+er_{12}-r_{22}\right)\alpha_2\beta_2-\left(hr_{11}-fr_{12}-r_{21}\right)\alpha_2.
\end{align}
Moreover, $\gamma_2 \in \bb_2$, and $\e\delta^{-1}\gamma_1 = \delta^{-1}\bb_1 = \bb_2$, so
\begin{equation}
\e\delta^{-1}\gamma_1 - \gamma_2 \in \bb_2 \cap \mm\OO[S_\mm^{-1}] = \bb_2\mm,
\end{equation}
because $\bb_2$ is coprime to $\mm$. Thus, $-gr_{11}+er_{12}-r_{22}$ and $hr_{11}-fr_{12}-r_{21}$ are integers, so
\begin{equation}
(AE)^{-1}\r_1 = \smmattwo{h}{-f}{-g}{e}\r_1 \equiv \r_2 \Mod{\Z^2}.
\end{equation}
We also know that $\alpha_2 j_{\!AE}(\beta_2) = \e\delta^{-1}\alpha_1$, and $\alpha_1,\alpha_2,\delta,\e$ are positive at both real embeddings, so $s_{\!AE}(\beta) = 1$. Thus, $s_{\!AE}(\beta)AE\r_2 \equiv \r_1 \Mod{\Z^2}$ and $AE \cdot \beta_2 = A \cdot \beta_2 = \beta_1$. We have now established that $\tilde\Upsilon_\mm$ is well-defined.

We now observe that $\tilde\Upsilon_\mm$ factors through the induced extension map
\begin{equation}
\ol\ext_{\OO;\mm\infty_1\infty_2}^{\OO';\mm\infty_1\infty_2} : \Clt_{\mm\infty_1\infty_2}(\OO) \to \Clt_{\mm\infty_1\infty_2}(\OO').
\end{equation}
This is seen by observing that $\bb\mm = \bb\OO'\mm$, and $\gamma\OO \in \bb\A \implies \gamma\OO' \in \bb\OO'\ol\ext_{\OO;\mm\infty_1\infty_2}^{\OO';\mm\infty_1\infty_2}(\A)$, so the definition of $\tilde\Upsilon_\mm(\A)$ remains unchanged under replacing $\bb$ by $\bb\OO'$. It follows (because $\M_{\OO',\mm}$ depends only on $\OO'$, not on $\OO$) that one need only prove the claims about the image of $\tilde\Upsilon_\mm$ in the case when $\OO' = \OO$.

Now suppose that $\OO' = \OO$; it follows by \Cref{prop:ordinvertible} that $\mm$ is $\OO$-invertible. To prove that $\tilde\Upsilon_\mm$ is injective and that it has image specified in the theorem statement, we will construct a function
\begin{equation}
\tilde\Omega_\mm : \SL_2(\Z) \backslash \M_{\OO,\mm} \to \Clt_{\mm\infty_1\infty_2}(\OO)
\end{equation}
and show that $\tilde\Omega_\mm$ defines an inverse to $\tilde\Upsilon_\mm$. Consider $(\r,\beta) \in \M_{\OO,\mm}$ such that $\rho_1(\beta)>\rho_2(\beta)$; note that every $\SL_2(\Z)$-orbit in $\Fquad$ contains such a $\beta$. 
Represent $\r \in \Q^2/\Z^2$ by an element $\r \in \Q^2$ such that $\sympt{\r}{\beta}$ is totally positive. 
Since $\ord(\beta\Z+\Z) = \OO$, it follows from \Cref{prop:ordinvertible} that $\beta\Z+\Z$ is $\OO$-invertible, so $\colonideal{\mm}{\beta\Z+\Z}$ is also $\OO$-inverible. 
By \Cref{prop:locallyprincipal}, for every nonzero prime $\pp$ of $\OO$, we have $\colonideal{\mm}{\beta\Z+\Z} = \alpha_p\OO_\pp$ for some $\alpha_\pp$. 
We may choose some $\alpha \in F^\times$ such that $\alpha$ is totally positive, $\alpha\OO_\pp = \alpha_\pp\OO_\pp$ whenever $\pp \supseteq \mm$, and $\alpha(\beta\Z+\Z) \subseteq \mm$. Setting $\bb = \alpha\colonideal{\beta\Z+\Z}{\mm}$, it follows that $\alpha(\beta\Z+\Z) = \bb\mm$, $\bb$ is an integral $\OO$-ideal, and $\bb+\mm=\OO$. 
Define $\gamma = \alpha\sympt{\r}{\beta}$ and
\begin{equation}
\tilde\Omega_\mm(\r,\beta) = [\gamma\bb^{-1}] \in \Clt_{\mm\infty_1\infty_2}(\OO).
\end{equation}

To show that $\tilde\Omega_\mm$ is well-defined, consider $(\r_1,\beta_1), (\r_2,\beta_2) \in \M_{\OO,\mm}$ such that
$\sympt{\r_1}{\beta}, \sympt{\r_2}{\beta}$ are totally positive and
\begin{equation}
(\r_2,\beta_2) = A \cdot (\r_1 + \n,\beta_1) = \left(s_{\!A}(\beta_1)A(\r_1 + \n), A\cdot\beta_1\right)
\end{equation}
for some $A = \smmattwo{a}{b}{c}{d} \in \SL_2(\Z)$ and $\n \in \Z^2$. Consider two choices of $\alpha_1, \alpha_2$ as above. These in turn determine $\bb_j = \alpha_j\colonideal{\beta_j\Z+\Z}{\mm}$ and $\gamma_j = \alpha_j\sympt{\r_j}{\beta_j}$ for $j \in \{1,2\}$.
We have
\begin{align}
\gamma_2 
&= \alpha_2\sympt{s_{\!A}(\beta_1)A(\r_1+\n)}{A\cdot\beta_1}
= \alpha_2s_{\!A}(\beta_1)\sympt{A(\r_1+\n)}{A\cdot\beta_1} \\
&= \frac{\alpha_2 s_{\!A}(\beta_1)}{j_{\!A}(\beta_1)}\sympt{\r_1+\n}{\beta_1}
= \frac{\alpha_2 s_{\!A}(\beta_1)}{\alpha_1 j_{\!A}(\beta_1)}\left(\gamma_1 + \alpha_1\sympt{\n}{\beta_1}\right). \label{eq:g12rel}
\end{align}
We also have
\begin{align}
\bb_2
&= \alpha_2\colonideal{\beta_2\Z+\Z}{\mm}
= \alpha_2\colonideal{\frac{a\beta_1+b}{c\beta_1+d}\Z+\Z}{\mm} 
= \frac{\alpha_2}{j_{\! A}(\beta_1)} \colonideal{\beta_1\Z+\Z}{\mm}
= \frac{\alpha_2}{\alpha_1 j_{\! A}(\beta_1)}\bb_1.
\end{align}
Thus,
\begin{equation}
\gamma_2\bb_2^{-1}
= s_{\!A}(\beta_1)\left(\gamma_1 + \alpha_1\sympt{\n}{\beta_1}\right)\bb_1^{-1}
= \left(\gamma_1 + \alpha_1\sympt{\n}{\beta_1}\right)\bb_1^{-1}.
\end{equation}
Since $\bb_1 = \alpha_1\colonideal{\beta_1\Z+\Z}{\mm}$ (and $\mm$ is $\OO$-invertible), we have $\bb_1\mm = \alpha_1(\beta_1\Z+\Z)$. Thus, $\alpha_1\sympt{\n}{\beta_1} \in \bb_1\mm \subseteq \mm$, so
\begin{equation}
\gamma_1 + \alpha_1\sympt{\n}{\beta_1}
\equiv \gamma_1 \Mod{\mm}.
\end{equation}
Therefore,
\begin{equation}\label{eq:simpart}
\gamma_2\bb_2^{-1} \sim_{\mm} \gamma_1\bb_1^{-1}.
\end{equation}
Moreover, the $\alpha_j$ and $\gamma_j$ are totally positive, and $s_{\!A}(\beta_1) = \sgn(j_{\!A}(\beta_1))$, so by \eqref{eq:g12rel} we have $\sgn(\rho_i(\gamma_1+\alpha_1\sympt{n,\beta})) = \sgn(\rho_1(j_A(\beta_1)))\sgn(\rho_i(j_A(\beta_1)))$. 
That is,
\begin{align}
\sgn(\rho_1(\gamma_1+\alpha_1\sympt{n}{\beta}))
&= \sgn(\rho_1(j_A(\beta_1)))^2 = 1, \mbox{ and} \\
\sgn(\rho_1(\gamma_1+\alpha_1\sympt{n}{\beta}))
&= \sgn(\Nm(j_A(\beta_1)))
= \sgn\!\left(\frac{\rho_1(\beta_1)-\rho_2(\beta_1)}{\rho_1(\beta_2)-\rho_2(\beta_2)}\right)
= 1,
\end{align}
using the conditions that $\rho_1(\beta_j) > \rho_2(\beta_j)$ in the last step. Hence \eqref{eq:simpart} can be improved to
\begin{equation}\label{eq:simall}
\gamma_2\bb_2^{-1} \sim_{\mm\infty_1\infty_2} \gamma_1\bb_1^{-1},
\end{equation}
and therefore the function $\tilde\Omega_\mm$ is well-defined.

It remains to check that $\tilde\Omega_\mm \circ \tilde\Upsilon_\mm$ is the identity on $\Clt_{\mm\infty_1\infty_2}$ and $\tilde\Upsilon_\mm \circ \tilde\Omega_\mm$ is the identity on $\SL_2(\Z) \backslash (\Q^2/\Z^2 \times \Fquad)$.
But these claims follow directly from the definitions of the two functions.
The formula $\ZClt_{\mm\infty_1\infty_2}(\OO) = \tilde\Omega_\mm(\{\mathbf{0}\} \times \Fquad) = \tilde\Upsilon_\m^{-1}(\{\mathbf{0}\} \times \Fquad)$ also follows directly from the construction of $\tilde\Omega_\mm$.

In the special case $\mm = m\OO$ for $m \in \N$, we have
\begin{align}
\M_{\OO,m\OO} 
&= \left\{(\r,\beta) \in \Q/\Z \times F_\OO : (r_2\beta-r_1)m\OO \subseteq \beta\Z+\Z\right\} \\
&= \left\{(\r,\beta) \in \Q/\Z \times F_\OO : r_2\beta-r_1 \in \tfrac{1}{m}(\beta\Z+\Z)\right\} \\
&= \tfrac{1}{m}\Z/\Z \times F_{\OO}.
\end{align}
Thus, $\image(\tilde\Upsilon_{m\OO}) = \SL_2(\Z)\backslash \left(\frac{1}{m}\Z/\Z \times F_{\OO}\right)$.

Finally, we must show that quotienting by $\sim_{\mm\infty}$ on the left-hand side of \eqref{eq:correspondence} corresponds under $\tilde\Upsilon_\mm$ to quotienting on the right-hand side by $\GL_2(\Z)$. (The statements about $\Upsilon_\mm$ will then follow.) 
Consider the ideal class $\mathfrak{R}_{-+} \in \Cl_{\mm\infty_1\infty_2}(\OO)$ given by
\begin{equation}
\mathfrak{R}_{-+} = \{\lambda\OO : \lambda \equiv 1 \Mod{\mm},\, \rho_1(\lambda) < 0 < \rho_2(\lambda)\};
\end{equation} 
then $\Clt_{\mm\infty_2}(\OO)$ is $\Clt_{\mm\infty_1\infty_2}(\OO)$ modulo the action of $\{\mathfrak{I}, \mathfrak{R}_{-+}\}$ (where $\mathfrak{I}$ is the identity class). 
Consider any $\A \in \Clt_{\mm\infty_1\infty_2}(\OO)$. Choose $\bb_1 \in \A_0^{-1}$, and write $\bb_1\mm = \alpha_1(\beta_1\Z+\Z)$, $\gamma_1\OO \in \bb_1\A$ with $\gamma_1 \in \bb_1$, and $\gamma_1 = \alpha_1\sympt{\r_1}{\beta_1}$ such that $\alpha_1, \gamma_1$ are totally positive and $\rho_1(\beta_1)>\rho_2(\beta_1)$, so that $\tilde\Upsilon_\mm(\A) = \SL_2(\Z) \cdot (\r_1,\beta_1)$. Choose some $A \in \GL_2(\Z)$ such that $\det(A)=-1$ and $\rho_2(j_{\!A}(\beta_1)) < 0 < \rho_1(j_{\!A}(\beta_1))$. Choose some $\n \in \Z^2$ such that $\delta := \sympt{\r_1+\n}{\beta_1}$ has $\rho_1(\delta)<0<\rho_2(\delta)$. 
Let $\bb_2 = j_{\!A}(\beta_1)\bb_1$, $\alpha_2 = \alpha_1 j_{\!A}(\beta_1)^2$, $\beta_2 = A\cdot\beta_1$, $\gamma_2 = -j_{\!A}(\beta_1)\delta$, and $\r_2 = A(\r_2+\n) = s_{\!A}(\beta)A(\r_2+\n)$. We may then check that $\bb_2 \in j_A(\beta_1)\A_0^{-1} = (\sR_{-+}\A)_0^{-1}$,
\begin{equation}
\bb_2\mm 
= j_{\!A}(\beta_1)\alpha_1(\beta\Z+\Z) 
= \alpha_1 j_{\!A}(\beta_1)^2((A\cdot\beta)\Z+\Z),
\end{equation}
$\gamma_2\OO = j_{\!A}(\beta_1)\delta\OO \in j_{\!A}(\beta_1)\bb_1\sR_{-+}\A = \bb_2\sR_{-+}\A$, $\gamma_2 \in \bb_2$, and 
\begin{equation}
\gamma_2
= \alpha_1 j_{\!A}(\beta_1)^2 \sympt{A(\r_1+\n)}{A\cdot\beta_1} = \alpha_2\sympt{\r_2}{\beta_2}.
\end{equation}
Moreover, $\alpha_2, \gamma_2$ are totally positive, and $\Nm(j_{\!A}(\beta_1)) = \det(A) \frac{\rho_1(\beta_1)-\rho_2(\beta_1)}{\rho_1(\beta_2)-\rho_2(\beta_2)}$, and taking the sign of each factor shows that $\rho_1(\beta_2)>\rho_2(\beta_2)$. Therefore,
\begin{align}
\tilde\Upsilon_{\mm}(\mathfrak{R}_{-+}\A) 
&= \SL_2(\Z) \cdot (\r_2, \beta_2) 
= \SL_2(\Z) \cdot A \cdot (\r_1,\beta_1) 
= A \cdot \SL_2(\Z) \cdot (\r_1,\beta_1)
= A \cdot \tilde\Upsilon_{\mm}(\A).
\end{align}
This proves that $\Upsilon_\mm$ defines a function from $\Clt_{\mm\infty_2}(\OO) \to \GL_2(\Z)\backslash(\Q^2/\Z^2 \times \Fquad)$ making the diagram \eqref{eq:correspondence} commute. (It then follows from the corresponding statements for $\tilde\Upsilon_\mm$ that $\Upsilon_\mm$ factors through the induced extension map $\ol\ext_{(\OO;\mm\infty_2)}^{(\OO';\mm\infty_2)}$, that its image is $\GL_2(\Z)\backslash \M_{\OO',\mm}$, and that is is injective whenever $\OO'=\OO$.)
\end{proof}

The maps $\tilde\Upsilon_\mm$ and $\Upsilon_\mm$ have some unintuitive behavior that should be pointed out. 
Firstly, by varying $\mm$, every $(\r, \beta)$ lies in the image of infinitely many $\tilde\Upsilon_\mm$. 
Indeed, if $(\r,\beta) \in \M_{\OO,\mm}$, then $(\r,\beta) \in \M_{\OO,\nn}$ for every nonzero $\OO$-ideal $\nn \subseteq \mm$. 

If $(\r,\beta) = \tilde\Upsilon_\mm(\A)$ for some primitive class $\A$ in a ray class group $\Cl_{\mm\infty_1\infty_2}(\OO)$, then $(\r,\beta) \in \M_{\OO,\nn}$ if and only if $\nn \subseteq \mm$ and $\nn$ is $\OO$-invertible; in such cases, we may consider $\A$ to be the canonical preimage of $(\r,\beta)$ and $\mm$ to be the ``level'' of $(\r,\beta)$. One might hope that every $(\r,\beta)$ is in the image of a ray class group (rather than only a ray class monoid); however, that is not always true.
\begin{egz}\label{eg:notraygroupimage}
Let $(\r,\beta) = \left(\smcoltwo{0}{1/3},3\sqrt{3}\right) = \left(\smcoltwo{-2}{1/3},3\sqrt{3}\right)$, where the latter representative is chosen so that $\sympt{\r}{\beta} = \sqrt{3}+2$ is totally positive. In order for $(\r,\beta) \in \SL_2(\Z)\backslash \M_{\mm,\OO}$, we must have $\OO = \ord\!\left(3\sqrt{3}\Z+\Z\right) = 3\sqrt{3}\Z+\Z$ and $(\sqrt{3}+2)\mm \subseteq 3\sqrt{3}\Z+\Z$. The latter condition and the integrality of $\mm$ implies that 
\begin{equation}
\mm 
\subseteq \frac{1}{\sqrt{3}+2}\!\left(3\sqrt{3}\Z+\Z\right) \cap \left(3\sqrt{3}\Z + \Z\right) 
= 3\sqrt{3}\Z+3\Z.
\end{equation}
However, $3\sqrt{3}\Z+3\Z=3\OO_{\Q(\sqrt{3})}$ is not $\OO$-invertible. In the partial order on ideals, there are four largest invertible $\OO$-ideals contained in $3\sqrt{3}\Z+3\Z$, giving the four possibilities
\begin{align}
\mm &\subseteq 9\sqrt{3}\Z+3\Z, \\
\mm &\subseteq 3\sqrt{3}\Z+9\Z, \\
\mm &\subseteq (3\sqrt{3}+3)\Z+9\Z, \text{ or} \\
\mm &\subseteq (3\sqrt{3}+6)\Z+9\Z.
\end{align}
Since there is no unique maximum $\OO$-invertible value of $\mm$ for which $(\r,\beta) \in \M_{\mm,\OO}$ (i.e., no well-defined ``level''), it follows that $(\r,\beta)$ is not in the image of a ray class group.

We give more details in the case $\mm = 9\sqrt{3}\Z + 3\Z$. We compute $\tilde\Omega_\mm(\r,\beta)$ by taking $\alpha = 3$, $\bb = \alpha\colonideal{\beta\Z+\Z}{\mm} = 3\colonideal{\OO}{3\OO} = \OO$, and $\gamma = \alpha\sympt{\r}{\beta} = 3\sqrt{3}+6$. Thus,
\begin{equation}
\tilde\Omega_\mm(\r,\beta) = [\gamma\bb^{-1}] = [(3\sqrt{3}+6)\OO] = [3\sqrt{3}\OO] \in \Clt_{\mm\infty_1\infty_2}(\OO).
\end{equation}
Thus, $(\r,\beta) = \tilde\Upsilon_\mm([3\sqrt{3}\OO])$, and the class $[3\sqrt{3}\OO]$ is imprimitive. (It can similarly be checked that $\tilde\Omega_\mm(\r,\beta)$ is imprimitive in the cases $\mm = 3\sqrt{3}\Z+9\Z$, $\mm = (3\sqrt{3}+3)\Z+9\Z$, and $\mm = (3\sqrt{3}+6)\Z+9\Z$. It follows that the same is true for subideals of these.)
\end{egz}
This example shows the necessity of working with the pathologies of the ray class monoids to describe all RM points in terms of ideal-theoretic data defining zeta functions.

Finally, we note some consequences of the reduction theory for binary quadratic forms for our correspondence. These are important for intermediate steps in \Cref{sec:partialzero} that use continued fraction expansions corresponding to reduced representatives.

\begin{defn}\label{defn:reduced}
Let $\A \in \Clt_{\mm\infty_1\infty_2}(\OO)$ (resp.~$\Clt_{\mm\infty_2}(\OO)$), and write
\begin{align}
\tilde\Upsilon_\mm(\A) &= \SL_2(\Z) \cdot (\r,\beta) &
\mbox{(resp.~} \Upsilon_\mm(\A) &= \GL_2(\Z) \cdot (\r,\beta) \mbox{)}.
\end{align}
We say that $(\r, \beta)$ is a \textit{reduced} representative of $\tilde\Upsilon_\mm(\A)$ (resp.~$\Upsilon_\mm(\A)$) if $-1 \leq r_1 < 0$, $0 \leq r_2 < 1$, and $0 < \rho_2(\beta) < 1 < \rho_1(\beta)$.
\end{defn}

\begin{prop}\label{prop:reduced}
Every $\tilde\Upsilon_\mm(\A)$ (resp.~$\Upsilon_\mm(\A)$) has at least one, and at most finitely many, reduced representatives.
\end{prop}
\begin{proof}
Follows from \cite[Thm.\ 1.3 and Thm.\ 1.4]{katoknotes}.
\end{proof}

\section{Modular properties of the $q$-Pochhammer symbol}\label{sec:wannabe}

In this section, we describe how the $q$-Pochhammer symbol transforms under modular transformations and give a framework for understanding the transformation factor as a \textit{modular 1-cocycle} or \textit{Jacobi 1-cocycle} and evaluating its \textit{real multiplication values} (or more generally, its \textit{stable values}), which are cohomological invariants. We define a notion of \textit{$w$-modular form} 
for a \textit{modular cocycle} $A \mapsto w_{\!A}$; this notion is a multiplicative analogue of Zagier's concept of a \textit{holomorphic quantum modular form}. The $q$-Pochhammer symbol with characteristics is a $w$-modular form for $w=\shin^\r$, the Shintani--Faddeev modular cocycle. We endeavor to keep the theory as simple as possible for now so as not to obscure what's going on, postponing more sophisticated cohomological interpretations to \Cref{sec:cohomology}.

\subsection{A working definition for first cohomology}\label{sec:working}

Let $\FF$ be any sheaf of multiplicative groups of $\C$-valued functions on a topological space $X$, and let $X^\circ$ be an open subset of $X$. Let $\Gamma$ be a group with a continuous action $\Gamma \times X \to X$. For $A \in \Gamma$ and $f \in \FF(\UU)$, we write $f^{A} \in \FF(A^{-1}\cdot \UU)$ for the function defined by $f^{A}(u) = f(A\cdot u)$.

We are concerned primarily with sheaves of analytic or meromorphic functions, usually on connected open sets, for which restriction maps are injective. To lessen the notational overload, we will not write the restriction maps unless they are needed. In other words, if $\UU_1, \UU_2 \subseteq X$ and $f_j \in \FF(\UU_j)$, we write
\begin{align}
f_1 = f_2 &\mbox{ to mean } f_1|_{\UU_1 \cap \UU_2} = f_2|_{\UU_1 \cap \UU_2} \mbox{ and} \\
f_1f_2 &\mbox{ to mean } (f_1|_{\UU_1 \cap \UU_2})(f_2|_{\UU_1 \cap \UU_2}).
\end{align}
Additionally, for arbitrary subsets $S \subseteq X$, we will write $\FF(S) := \bigcup_{\UU \supseteq S} \FF(\UU)$, where the union denotes a direct limit over open sets $\UU$ taken with respect to the restriction maps. Moreover, for arbitrary $S_1, S_2 \subseteq X$, we set $\FF_{S_1}\!(S_2) := \FF(S_1 \cap S_2)$.

\begin{defn}\label{defn:domainscohom}
A \textit{system of domains} $\UU = (\UU_{\!A})_{\!A \in \Gamma}$ is a $\Gamma$-tuple of open subsets of $X^\circ$ that is also an open cover of $X^\circ$, that is, a map $(A \mapsto \UU_{\!A}) : \Gamma \to \{\mbox{open subsets of $X^\circ$}\}$ such that $X^\circ = \bigcup_{A \in \Gamma} \UU_{\!A}$.
A \textit{$1$-cochain for $\UU$} is an element $w = (w_{\!A})_{\!A\in\Gamma}$ of the multiplicative group
\begin{equation}
C_{\UU}^1(\Gamma,\FF) = \prod_{A \in \Gamma} \FF(\UU_{\!A}).
\end{equation}
A \textit{$1$-cocycle for $\UU$} is an element of the subgroup
\begin{align}
Z_{\UU}^1(\Gamma,\FF) = \{w \in C_{\UU}^1(\Gamma,\FF) : w_{A_1A_2} = w_{\!A_1}^{A_2} w_{\!A_2}\}.
\end{align}
A \textit{$1$-coboundary for $\UU$} is an element of the subgroup 
\begin{align}
B_{\UU}^1(\Gamma,\FF) = \{w \in C_{\UU}^1(\Gamma,\FF) : w_{\!A} = f^{A}f^{-1} \mbox{ for some } f \in \FF(X^\circ)\}.
\end{align}
A \textit{first cohomology class for $\UU$} is an element of the quotient group
\begin{align}
H_{\UU}^1(\Gamma,\FF) &= \frac{Z_{\UU}^1(\Gamma,\FF)}{B_{\UU}^1(\Gamma,\FF)}.
\end{align}
A cohomology class is typically denoted as $[w]$ for $w \in Z_{\UU}^1(\Gamma,\FF)$.
\end{defn}

\subsection{Modular and Jacobi cocycles}

We now restrict to the group actions and sheaves of interest for this paper. On any complex manifold, we will denote by $\AA$ and $\MM$ the sheaves of rings of analytic functions and meromorphic functions, respectively, so that $\AA^\times$ and $\MM^\times$ are the sheaves of multiplicative groups of nowhere vanishing analytic functions and nonzero meromorphic functions, respectively.

\begin{defn}
Suppose $X = \C \cup \{\infty\}$, $X^\circ = \C$, and $\Gamma$ is a discrete subgroup of $\SL_2(\R)$ acting by fractional linear transformations $\smmattwo{a}{b}{c}{d}\cdot\tau = \frac{a\tau+b}{c\tau+d}$. Suppose $\HH \subseteq \UU_{\!A} \subseteq \C \cup \{\infty\}$ for $A \in \Gamma$, and $\{\UU_{\!A}\}_{\!A \in \Gamma}$ is an open cover of $\C$.  If $\FF=\AA^\times$ or $\FF=\MM^\times$, the elements of the groups $C_{\UU}^1(\Gamma,\FF)$, $Z_{\UU}^1(\Gamma,\FF)$, $B_{\UU}^1(\Gamma,\FF)$, and $H_{\UU}^1(\Gamma,\FF)$ are called \textit{modular $1$-cochains}, \textit{modular $1$-cocycles}, \textit{modular $1$-coboundaries}, and \textit{modular first cohomology classes}, respectively.
\end{defn}

The first example of a weight cocycle is the standard modular cocycle.
\begin{egz}
For $A = \smmattwo{a}{b}{c}{d} \in \SL_2(\R)$ and $\tau \in \C$, define
\begin{equation}
j_{\!A}(\tau) = c\tau+d.
\end{equation}
For any $\Gamma \leq \SL_2(\R)$, $j_{\!A}$ is an analytic modular cocycle for the constant system of domains $\UU_{\!A} = \C$.
\end{egz}

\begin{defn}
Suppose $X = X^\circ = \C \times (\C \setminus \Q)$ and $\Gamma$ is a discrete subgroup of $\R^2 \semidirect \SL_2(\R)$ acting by the Jacobi action 
\begin{equation}
(\m,A)\cdot(z,\tau) = \left(\frac{z}{j_{\!A}(\tau)}+\sympt{\m}{A\cdot\tau},A\cdot\tau\right)
\end{equation}
(where $A \cdot \tau$ denotes the fractional linear transformation action).
Suppose $\C \times \HH \subseteq \UU_{\!A} \subseteq X$ for $A \in \Gamma$, and $\{\UU_{\!A}\}_{\!A \in \Gamma}$ is an open cover of $X^\circ$. 
If $\FF=\AA^\times$ or $\FF=\MM^\times$, the elements of the groups $C_{\UU}^1(\Gamma,\FF)$, $Z_{\UU}^1(\Gamma,\FF)$, $B_{\UU}^1(\Gamma,\FF)$, and $H_{\UU}^1(\Gamma,\FF)$ are called \textit{Jacobi $1$-cochains}, \textit{Jacobi $1$-cocycles}, \textit{Jacobi $1$-coboundaries}, and \textit{Jacobi first cohomology classes}, respectively.
\end{defn}

\subsection{Cocycles as generalized modular weights}

Recall that a meromorphic modular form of weight $k$ is a meromorphic function $f : \HH \to \C$ whose coboundary is $j_{\!A}^k$; that is, such that $f(A\cdot\tau) = j_{\!A}(\tau)^k f(\tau)$ This definition may be generalized so as to replace $j_{\!A}^k$ by an arbitrary modular cocycle $w$.
\begin{defn}
Let $w \in Z_{\UU}^1(\Gamma,\MM_\C^\times)$ be a modular cocycle for some system of domains $\UU$.
A meromorphic complex-valued function $f : \HH \to \C$ is a 
\textit{$w$-modular form}\footnote{The author has used the term ``wannabe modular form'' in several talks on the subject but is now aware that Zagier calls the additive analogues of such functions---that is, $f$ such that $f(A\cdot\tau)-f(\tau)$ or more generally $j_{\!A}(\tau)^{-k} f(A\cdot\tau)-f(\tau)$ has a larger domain of analyticity---``holomorphic quantum modular forms.'' The latter terminology has appeared in print in work of Bringmann, Ono, and Wagner \cite{bow}. We primarily use the term ``$w$-modular form'' (where the cocycle $w$ is specified) in this work, but when discussing these objects in general, we refer to them informally as ``multiplicative holomorphic quantum modular forms.''} 
if
\begin{equation}
f(A\cdot\tau) = w_{\!A}(\tau)f(\tau)
\end{equation}
for all $\tau \in \HH$  (except where both sides have a pole).
\end{defn}
We also define a compatible generalization of meromorphic Jacobi forms.
\begin{defn}
Let $u \in Z_{\UU}^1(\Gamma,\MM_{\C \times (\C\setminus\Q)}^\times)$ be a Jacobi cocycle for some system of domains $\UU$.
A meromorphic complex-valued function of two variables $g(z,\tau)$ for $z \in \C$ and $\tau \in \HH$ is 
a \textit{$u$-Jacobi form} if 
\begin{equation}
g((\m,A)\cdot(z,\tau)) = u_{(\m,A)}(z,\tau)g(z,\tau)
\end{equation}
for all $(z,\tau) \in \C \times \HH$ (except where both sides have a pole).
\end{defn}

\subsection{Stable values and real multiplication values of modular cocycles}\label{sec:stable}

In this section, we show how a modular $1$-cocycle---and indeed, a first cohomology class---can sometimes be evaluated at a point to produce a numerical value. Doing so requires choosing ``canonical'' generators for certain stabilizers. We restrict to the case when $\Gamma$ is a finite-index subgroup of $\SL_2(\Z)$.

\begin{defn}\label{defn:generator}
Let $\Gamma$ be a finite-index subgroup of $\SL_2(\Z)$, and let $\beta \in \C \cup \infty$. Define $A_\beta^+$ to be the unique element of $\Gamma$ with the following properties:
\begin{itemize}
\item[(1)] The stabilizer $\stab_\Gamma(\beta) = \langle A_\beta^+ \rangle$ (if $-I \nin \Gamma$) or $\stab_\Gamma(\beta) = \langle \pm I, A_\beta^+ \rangle$ (if $-I \in \Gamma$).
\item[(2)] The following condition holds in the appropriate case:
\begin{itemize}
\item If $\beta \in \Rquad$, then $A_\beta^+ \smcoltwo{\beta}{1} = \lambda\smcoltwo{\beta}{1}$ for some $\lambda>1$.
\item If $\beta \in \SL_2(\Z) \cdot \infty = \Q \cup \{\infty\}$, then $A_\beta^+ = P\smmattwo{1}{b}{0}{1}P^{-1}$ for some $b>0$ and some $P \in \SL_2(\Z)$.
\item If $\beta \in \SL_2(\Z) \cdot \frac{1+\sqrt{-3}}{2}$, then $A_\beta^+ = P\smmattwo{0}{-1}{1}{-1}P^{-1}$ for some $P \in \SL_2(\Z)$.
\item If $\beta \in \SL_2(\Z) \cdot \sqrt{-1}$, then $A_\beta^+ = P\smmattwo{0}{-1}{1}{0}P^{-1}$ for some $P \in \SL_2(\Z)$.
\item Otherwise, $A_\beta^+ = I$.
\end{itemize}
\end{itemize}
\end{defn}

The positivity conditions for choosing a generator of the stabilizer may be understood geometrically in terms of the fractional linear transformation action. For $\beta \in \Rquad$, the eigenvalue condition $\lambda>1$ means that $A_\beta^+$ acts on $\C \cup \{\infty\}$ with $\beta$ as an attracting fixed point, dynamically speaking. Additionally, $A_\beta^+$ preserves the modular geodesic between $\beta$ and $\beta'$ setwise, and it moves points on this geodesic a distance of $2\log\lambda$ toward $\beta$ in the hyperbolic metric. For $\beta \in \Q \cup \{\infty\}$, the matrix $A_\beta^+$ acts by shifting points along horocycles centered on the point $\tau = \beta$, and the condition $b>0$ specifies the direction of this movement.

\begin{prop}\label{prop:stablewd}
Let $\Gamma$ be a finite-index subgroup of $\SL_2(\Z)$, $\UU$ a system of domains for $\Gamma$, and $w$ a meromorphic modular $1$-cocycle for $\UU$.
Let $\beta \in \C \cup \{\infty\}$, and suppose $\beta \in \UU_{\!A^+_\beta}$. If defined, the value
$w[\beta] := w_{\!A_\beta^+}(\beta) \in \C^\times$
depends only on the
class of $w$ in $\frac{Z_{\UU}^1(\Gamma,\MM_\C^\times)}{B_{\UU}^1(\Gamma,\MM_{\{\beta\},\C}^\times)}$, where $\MM_{\{\beta\},\C}$ denotes the sheaf of meromorphic functions that are analytic at $\beta$. In particular, $[w] \in H_{\UU}^1(\Gamma,\AA_\C^\times)$ defines a canonical value $[w][\beta] := w[\beta] \in \C^\times$. For $[w] \in H_{\UU}^1(\Gamma,\MM_\C^\times)$, we instead obtain $[w][\beta] \in \C^\times/\lambda^{2\Z}$ where $\lambda = j_{\!A^+_\beta}(\beta)$.
\end{prop}
\begin{proof}
Let $A = A_\beta^+$.
Consider $w, \tilde{w} \in Z_{\UU}^1(\Gamma,\MM_\C^\times)$ such that $[\tilde{w}] = [w] \in H_{\UU}^1(\Gamma,\MM_\C^\times)$. Then, we can write
\begin{equation}
\tilde{w}_{\!A}(\tau) = \frac{f(A\cdot\tau)}{f(\tau)}w_{\!A}(\tau)
\end{equation}
for some $f \in \MM^\times(\C)$. Near $\tau = \beta$, we have a Laurent series expansion
\begin{equation}
f(\tau) = \sum_{k=n}^\infty c_k(\tau-\beta)^k
\end{equation}
for some $n \in \Z$ (possibly negative) and $c_n \neq 0$. Since $\beta = A\cdot\beta$, writing $A = \frac{a\tau+b}{c\tau+d}$, we have
\begin{align}
f(A\cdot\tau) 
&= \sum_{k=n}^\infty c_k\left(\frac{a\tau+b}{c\tau+d}-\frac{a\beta+b}{c\beta+d}\right)^k \\
&= \sum_{k=n}^\infty c_k\left(\frac{(ad-bc)(\tau-\beta)}{(c\tau+d)(c\beta+d)}\right)^k \\
&= \sum_{k=n}^\infty c_k j_{\!A}(\beta)^{-k}j_{\!A}(\tau)^{-k}(\tau-\beta)^k.
\end{align}
Thus, as $\tau \to \beta$, we have
\begin{equation}
\lim_{\tau \to \beta} \frac{f(A\cdot\tau)}{f(\tau)}
= \lim_{\tau \to \beta} \frac{c_n j_{\!A}(\beta)^{-n}j_{\!A}(\tau)^{-n}}{c_n}
= j_{\!A}(\beta)^{-2n}.
\end{equation}
Hence $\tilde{w}_{\!A}(\tau) = \lambda^{-2n}w_{\!A}(\tau)$ where $\lambda = j_{\!A}(\beta)$.
Therefore, $[w] \in H_{\UU}^1(\Gamma,\MM_\C^\times)$ defines a canonical element $[w][\beta] \in \C^\times/\lambda^{2\Z}$.

Moreover, if $f \in \MM_{\{\beta\},\C}^\times(\C)$, then $n=0$, so $\tilde{w}_{\!A}(\tau) = w_{\!A}(\tau)$. Therefore, the class of $w$ in $\frac{Z_{\UU}^1(\Gamma,\MM_\C^\times)}{B_{\UU}^1(\Gamma,\MM_{\{\beta\},\C}^\times)}$ defines a canonical value $[w][\beta] \in \C^\times$.
\end{proof}

\begin{defn}
We call $w[\beta] := w_{\!A_\beta^+}(\beta)$ the \textit{stable value} of $w$ at $\beta$; we may also write this quantity as $[w][\beta]$ to indicate that it only depends on the class $[w]$ (although we typically use the former for notational simplicity).
If $\beta \in \Rquad$, we also call $w[\beta]$ (or $[w][\beta]$) the \textit{real multiplication value (RM value)} of $w$ (or $[w]$) at $\beta$.
\end{defn}

We now describe the RM values of $j_{\!A}(\tau)$ and show that the two uses of $\lambda$ earlier in this section are consistent with each other.
\begin{lem}\label{lem:jeval}
If $\beta \in \C \cup \{\infty\}$ and $A_\beta^+ \smcoltwo{\beta}{1} = \lambda \smcoltwo{\beta}{1}$, then
\begin{equation}
    j[\beta] =  \lambda.
\end{equation}
\end{lem}
\begin{proof}
Let $A = A_\beta^+$.
For any $R \in \GL_2(\C)$, we have
\begin{align}
    j_{\!RAR^{-1}}(R\cdot\beta) 
    &= j_{\!RA}(\beta)j_{\!R^{-1}}(R\cdot\beta) \\
    &= j_{\!R}(A\cdot\beta)j_{\!A}(\beta)j_{\!R^{-1}}(R\cdot\beta) \\
    &= j_{\!R}(\beta)j_{\!A}(\beta)j_{\!R^{-1}}(R\cdot\beta).
\end{align}
But also
$1 = j_{\!I}(\beta) = j_{\!R^{-1}R}(\beta) = j_{\!R^{-1}}(R\cdot\beta)j_{\!R}(\beta)$.
Thus, $j_{\!RAR^{-1}}(R\cdot\beta) = j_{\!A}(\beta)$. 
Choose $R$ such that $RAR^{-1}$ is in Jordan form:
\begin{equation}
    RAR^{-1} = \mattwo{\lambda^{-1}}{\delta}{0}{\lambda},
\end{equation}
where $\delta = 0$ if $\lambda \neq 1$.
Thus, $j_{\!A}(\beta) = j_{\!RAR^{-1}}(R\cdot\beta) = \lambda$.
\end{proof}

More precisely, we can directly compute in the hyperbolic, parabolic, and elliptic cases, respectively, that:
\begin{itemize}
\item If $\beta \in \Rquad$, then $j[\beta]$ is a unit in the real quadratic field $\Q(\beta)$.
\item If $\beta \in \SL_2(\Z) \cdot \infty = \Q \cup \{\infty\}$, then $j[\beta]=1$.
\item If $\beta \in \SL_2(\Z) \cdot \frac{1+\sqrt{-3}}{2}$, then $j[\beta]=\frac{-1+\sqrt{-3}}{2}$.
\item If $\beta \in \SL_2(\Z) \cdot \sqrt{-1}$, then $j[\beta]=\sqrt{-1}$.
\item Otherwise, $j[\beta]=1$.
\end{itemize}
Moreover, in the hyperbolic case, we have the following characterization of $j[\beta]$ as the fundamental totally positive unit of a certain real quadratic order.
\begin{prop}\label{prop:jevalreal}
Let $\beta \in \Rquad$ and $\lambda = j[\beta]$. If $\OO = \colonideal{\beta\Z + \Z}{\beta\Z + \Z}$, then $\lambda$ is a generator of the group $\OO^\times_+$ of totally positive units of $\OO$.
\end{prop}
\begin{proof}
Write $\OO^\times_+ = \langle \e \rangle$.
Clearly $\lambda$ is a totally positive unit in $\OO$ (as it is in $\OO$, is greater than $1$, and is the eigenvalue of an integral matrix), so $\lambda$ is some positive integral power of $\e$. On the other hand, $\e(\beta\Z+\Z) = \beta\Z+\Z$, so there exists $a,b,c,d \in \Z$ such that
\begin{align}
\e\beta &= a\beta + b, \\
\e &= c\beta + d.
\end{align}
That is, $\e\smcoltwo{\beta}{1} = \smmattwo{a}{b}{c}{d}\smcoltwo{\beta}{1}$. Since $\e$ is an eigenvalue of $\smmattwo{a}{b}{c}{d}$, it follows that $\e' = \e^{-1}$ is also an eigenvalue of $\smmattwo{a}{b}{c}{d}$, so $\det \smmattwo{a}{b}{c}{d} = 1$. Therefore, $\smmattwo{a}{b}{c}{d}$ is some integral power of $A_\beta^+$, so $\e$ is an integral power of $\lambda$. As we've already shown that $\lambda$ is a positive integral power of $\e$, we conclude that $\lambda = \e$.
\end{proof}

\subsection{Stable values and real multiplication values of Jacobi cocycles}\label{sec:stablejacobi}

On can define stable values of Jacobi cocycles is a similar way to stable values of modular cocycles. We will restrict to irrational $\beta$ because the Jacobi group action is not well-defined on $(\C \cup \{\infty\}) \times (\C \cup \{\infty\})$ but is well-defined on $\C \times (\C \setminus \Q)$. Stable values will be trivial except when $\beta \in \Rquad$ (real multiplication points) and $\beta \in \left(\SL_2(\Z) \cdot \sqrt{-1}\right) \cup \left(\SL_2(\Z) \cdot \frac{-1+\sqrt{-3}}{2}\right)$ (elliptic points). It will turn out that (ignoring rational points) the stable values of Jacobi cocycles are identical to stable values of related modular cocycles.

\begin{defn}
Let $\Gamma$ be a finite-index subgroup of $\Z^2 \semidirect \SL_2(\Z)$.
Let $\beta \in \C \setminus \Q$ and $\ww \in \beta\Q+\Q$.
Define $\m_{\ww,\beta}^+$ to be the unique element of $\Z^2$ such that
\begin{equation}
\stab_{\Gamma}(z,\beta) = \langle(\m_{\ww,\beta}^+,B_{\ww,\beta}^+)\rangle \mbox{ or } \langle(\mathbf{0},-I), (\m_{\ww,\beta}^+,B_{\ww,\beta}^+)\rangle,
\end{equation}
where $B_{\ww,\beta}^+ = (A_\beta^+)^k$ for some $k \in \N$ and $A_\beta^+$ is defined with respect to the subgroup $\Gamma \cap \SL_2(\Z)$ as in \Cref{defn:generator}.
\end{defn}

\begin{defn}
Let $\Gamma$ be a finite-index subgroup of $\Z^2 \semidirect \SL_2(\Z)$.
Let $\beta \in \C \setminus \Q$ and $\ww \in \beta\Q+\Q$.
Define the \textit{stable value} of $u$ at $(\ww,\beta)$ to be
\begin{equation}
u[\ww,\beta] := u_{\m_{\ww,\beta}^+,B_{\ww,\beta}^+}(\ww,\beta).
\end{equation}
If $\beta \in \Rquad$, then $u[\ww,\beta]$ is also called the \textit{real multiplication (RM) value} of $u$ at $(\ww,\beta)$.
\end{defn}

\begin{lem}\label{lem:mintermsofB}
Let $\Gamma$ be a finite-index subgroup of $\Z^2 \semidirect \SL_2(\Z)$, and let $\beta \in \C \setminus \Q$.
If $\r \in \Q^2$ and $\ww=\sympt{\r}{\beta}$, then
$\m_{\ww,\beta}^+ = (I-B_{\ww,\beta}^+)\r$.
\end{lem}
\begin{proof}
Write $\m_{\ww,\beta}^+ = \m = \smcoltwo{m_1}{m_2}$ and $B_{\ww,\beta}^+ = B = \smmattwo{a}{b}{c}{d}$.
The stability condition $B \cdot \beta = \beta$ means that $\frac{a\beta+b}{c\beta+d} = \beta$, that is, $c\beta^2 = (a-d)\beta + b$. Thus,
\begin{align}
(-c\beta+a)(c\beta+d) 
&= -c^2\beta^2 + (ac-cd)\beta + ad \\
&= -c((a-d)\beta+b) + (ac-cd)\beta + ad
= -bc+ad = 1.
\end{align}
That is, $(c\beta+d)^{-1} = -c\beta+a$, and it follows that
\begin{align}
\frac{\ww}{j_{\!B}(\beta)} + \sympt{\m}{B\cdot\beta}
&= \sympt{\r}{\beta}(c\beta+d)^{-1} + \sympt{\m}{\beta} \\
&= (r_2\beta-r_1)(-c\beta+a) + (m_2\beta - m_1) \\
&= -r_2(c\beta^2) + (cr_1 + ar_2 + m_2\beta)\beta - (ar_1+m_1) \\
&= -r_2((a-d)\beta+b) + (cr_1 + ar_2 + m_2)\beta - (ar_1+m_1) \\
&= (cr_1 + dr_2 + m_2)\beta - (ar_1 + br_2 + m_1).
\end{align}
But also $\frac{\ww}{j_{\!B}(\beta)} + \sympt{\m}{B\cdot\beta} = \ww = r_2\beta-r_1$, so 
\begin{equation}\label{eq:cdab}
(cr_1 + dr_2 + m_2)\beta - (ar_1 + br_2 + m_1)
= r_2\beta - r_2.
\end{equation}
Since $\beta \nin \Q$, we can equate coefficients of $\beta$ and $1$ in \eqref{eq:cdab}, obtaining $m_2 = -cr_1 + (1-d)r_2$ and $m_1 = (1-a)r_1 - br_2$, that is,
$\m = (I-B)\r$.
\end{proof}

\begin{prop}\label{prop:jacobistable}
Let $u_{\m,A}(z,\tau)$ be a Jacobi cocycle for a finite-index subgroup $\Gamma \leq \Z^2 \cap \SL_2(\Z)$ with the property that $\Gamma \cap (\Z^2 \semidirect \{I\}) = \Z^2$. For each $\r \in \Q^2$, define an associated modular cocycle
\begin{equation}
w^\r_{\!A}(\tau) = u_{(I-A)\r,A}(\sympt{\r}{\tau},\tau).
\end{equation}
Then $w^\r$ is a modular cocycle for the group $\Gamma \cap \Gamma_\r$. If $\beta \in \C \setminus \Q$ and $w^\r[\beta]$ is defined, then
\begin{equation}
u[\sympt{\r}{\beta},\beta] = w^\r[\beta].
\end{equation}
\end{prop}
\begin{proof}
First, we check the modular cocycle condition. 
In the Jacobi group, we have the identity
\begin{align}
((I-A_1)\r,A_1)((I-A_2)\r,A_2)
&= ((I-A_1)\r + (A_1-A_1A_2)\r, A_1A_2) \\
&= ((I-A_1A_2)\r, A_1A_2).
\end{align}
The Jacobi action and the modular action are also related by the identity
\begin{align}
\left((I-A)\r,A\right)\cdot\left(\sympt{\r}{\tau},\tau\right)
&= \left(\frac{\sympt{\r}{\tau}}{j_{\!A}(\tau)} + \sympt{(I-A)\r}{A\cdot\tau} A\cdot\tau\right) \\
&= \left(\sympt{\r}{A\cdot\tau} + \sympt{\r-A\r}{A\cdot\tau}, A\cdot\tau\right) \\
&= \left(\sympt{A\r}{A\cdot\tau},A\cdot\tau\right).
\end{align}
Thus, using the Jacobi cocycle condition and the two identities just shown, we have
\begin{align}
w^\r_{\!A_1A_2}(\tau)
&= u_{(I-A_1A_2)\r,A_1A_2}(\sympt{\r}{\tau},\tau) \\
&= u_{(I-A_1)\r,A_1}\!\left(((I-A_2)\r,A_2)\cdot(\sympt{\r}{\tau},\tau)\right)u_{(I-A_2)\r,A_2}(\sympt{\r}{\tau},\tau) \\
&= u_{(I-A_1)\r,A_1}\!\left(\sympt{A_1\r}{A_1\cdot\tau},A_1\cdot\tau\right)u_{(I-A_2)\r,A_2}(\sympt{\r}{\tau},\tau) \\
&= w^\r_{\!A_1}(A_2\cdot\tau)w^\r_{\!A_2}(\tau).
\end{align}

Now, we check the equality of stable values. Let $\m = \m_\beta^+$ and $B = B_\beta^+$. Then $B$ is also the positive generator for the stabilizer of $\beta$ in $\Gamma \cap \Gamma_\r$ (because it is the smallest power of $A_\beta^+$ fixing $\m$ modulo $\beta\Z+\Z$, or equivalently, fixing $\r$ modulo $\Z^2$). Therefore, using \Cref{lem:mintermsofB},
\begin{align}
u[\sympt{\r}{\beta},\beta] 
= u_{(I-B)\m,B}(\sympt{\r}{\beta},\beta)
= w^\r[\beta]. \tag*{\qedhere}
\end{align}
\end{proof}

\subsection{The Shintani--Faddeev cocycles}\label{sec:sfcocycles}

We now define the main transcendental functions of interest in this paper, which give nontrivial examples of a Jacobi cocycle and a modular cocycle. 
For now, we define them as coboundaries on $\MM_\HH^\times$ and $\MM_{\C \times \HH}^\times$, which we will later extend to cocyles on larger systems of domains.

We will define both the ``modular'' and ``Jacobi'' versions of the cocycle in terms of special cases of the following function.
\begin{defn}\label{defn:sfjacobimaster}
For $\m \in \R^2$, $A \in \SL_2(\R)$, $z \in \C$, and $\tau \in \HH$, define the following meromorphic function of $z$ and $\tau$:
\begin{equation}
\sfj{\m,A}{z}{\tau} = \frac{\varpi\!\left(\frac{z}{j_{\!A}(\tau)}+\sympt{\m}{A\cdot\tau},A \cdot \tau\right)}{\varpi(z,\tau)}.
\end{equation}
\end{defn}

\begin{defn}\label{defn:sfjacobi}
The \textit{Shintani--Faddeev Jacobi cocycle} is the $(\Z^2 \semidirect \SL_2(\Z))$-tuple 
\begin{equation}
(\sfj{\m,A}{z}{\tau})_{(\m,A) \in \Z^2 \semidirect \SL_2(\Z)}.
\end{equation}
Here, we meromorphically continue $\sfj{\m,A}{z}{\tau}$ to $(z,\tau) \in \C \times \DD_{\!A}$, where
\begin{equation}\label{eq:DD}
\DD_{\!A} = 
\begin{cases}
\C \setminus (-\infty,-d/c] & \mbox{if } c > 0,\\
\C & \mbox{if } c=0 \mbox{ and } d>0,\\
\HH & \mbox{if } c=0 \mbox{ and } d<0,\\
\C \setminus [-d/c,\infty) & \mbox{if } c < 0.
\end{cases}
\end{equation}
The existence of the meromorphic continuation will be shown in \Cref{thm:wannabej}. In the case when $\m=\mathbf{0}$, we will sometimes drop $\m$ and write
\begin{equation}
\sfj{A}{z}{\tau} := \sfj{\mathbf{0},A}{z}{\tau}.
\end{equation}
\end{defn}

\begin{defn}\label{defn:sfmodular}
Let $\r \in \Q^2$. The \textit{Shintani--Faddeev modular cocycle (with rational characteristics $\r$)} of $A \in \Gamma_\r$ is the $\Gamma_\r$-tuple
\begin{equation}
(\sf{\r}{A}{\tau})_{\!A \in \Gamma_\r},
\end{equation}
where $\sf{\r}{A}{\tau} = \sfj{\r,A}{0}{\tau}$.
Here, we meromorphically continue $\sf{\r}{A}{\tau}$
to $\tau \in \tDD_{\!A}$, where
\begin{equation}\label{eq:tDD}
\tDD_{\!A} = \left\{\begin{array}{ll}
\C \setminus (-\infty,-d/c] & \mbox{ if } c > 0,\\
\C & \mbox{ if } c=0,\\
\C \setminus [-d/c,\infty) & \mbox{ if } c < 0.
\end{array}\right.
\end{equation}
The existence of the meromorphic continuation will be shown in \Cref{thm:wannabem}.
\end{defn}
For $\tau \in \HH$, except when $\r \in \Z^2$ and $r_2 \leq 0$, we have
\begin{equation}
\sf{\r}{A}{\tau} = \frac{\varpi_\r(A \cdot \tau)}{\varpi_\r(\tau)} 
= \prod_{k=0}^\infty \frac{1-\ee{(k+r_2)(A\cdot\tau)-r_1}}{1-\ee{(k+r_2)\tau-r_1}}.
\end{equation}

The Shintani--Faddeev modular cocycle is related to the Shintani--Faddeev Jacobi cocycle by the following proposition.
\begin{prop}\label{prop:goose}
The following relations of meromorphic functions hold for $\r \in \Q^2$, $A \in \Gamma_\r$, and $\tau \in \DD_{\!A}$.
\begin{align}
\sf{\r}{A}{\tau} 
&= \sfj{(I-A)\r,A}{\sympt{\r}{\tau}}{\tau} \\
&= \left(\ee{\frac{\sympt{\r}{\tau}}{j(A,\tau)}},\ee{A\cdot\tau}\right)_{\symp{(I-A)\r}{\smcoltwo{1}{0}}}^{-1} \sfj{A}{\sympt{\r}{\tau}}{\tau}; \label{eq:prel1} \\
\sf{\r}{A}{\tau} 
&= \left(\ee{\sympt{A^{-1}\r}{\tau}},\ee{\tau}\right)_{\symp{\r}{(I-A)\smcoltwo{1}{0}}} \sfj{A}{\sympt{A^{-1}\r}{\tau}}{\tau}. \label{eq:prel2}
\end{align}
\end{prop}
\begin{proof}
For $\tau \in \HH$, write
\begin{equation}
\sf{\r}{A}{\tau} \label{eq:shinfracj}
= \frac{\varpi_\r(A\cdot\tau)}{\varpi_\r(\tau)}
= \frac{\varpi\!\left(\sympt{\r}{A\cdot\tau},A\cdot\tau\right)}{\varpi\!\left(\sympt{\r}{\tau},\tau\right)}.
\end{equation}
To prove \eqref{eq:prel1}, express
\begin{align}
\sympt{\r}{A\cdot\tau}
&= \sympt{A\r}{A\cdot\tau} + \sympt{(I-A)\r}{A\cdot\tau} \\
&= \frac{\sympt{\r}{\tau}}{j_{\!A}(\tau)} + \symp{(I-A)\r}{\smcoltwo{A\cdot\tau}{1}} \\
&= \frac{\sympt{\r}{\tau}}{j_{\!A}(\tau)}  + \symp{(I-A)\r}{\smcoltwo{1}{0}}(A\cdot\tau) + \symp{(I-A)\r}{\smcoltwo{0}{1}}.
\end{align}
Note that $\symp{(I-A)\r}{\smcoltwo{1}{0}}$ and $\symp{(I-A)\r}{\smcoltwo{0}{1}}$ are integers because $A\r \con \r \Mod{\Z^2}$. Apply \Cref{lem:ell} to the numerator of \eqref{eq:shinfracj} to prove \eqref{eq:prel1}. 

Similarly, to prove \eqref{eq:prel2}, express
\begin{equation}
\sympt{\r}{\tau} = \sympt{A^{-1}\r}{\tau} + \symp{(I-A^{-1})\r}{\smcoltwo{1}{0}}\tau + \symp{(I-A^{-1})\r}{\smcoltwo{0}{1}},
\end{equation}
and again apply \Cref{lem:ell}, this time to the demoninator of \eqref{eq:shinfracj}.

Both \eqref{eq:prel1} and \eqref{eq:prel2} then hold on $\DD_{\!A}$ by meromorphic continuation.
\end{proof}

\begin{rmk}
The zeros of the analytic function $\varpi(z,\tau)$ for $(z,\tau) \in \C \times \HH$ occur exactly when $z \in \tau\N_0 + \Z$, where $\N_0$ is the set of nonnegative integers. 
It follows that, for $(\m,A) \in \Z^2 \semidirect \SL_2(\Z)$, the meromorphic function $\sigma_{\m,A}(z,\tau)$ has its poles and zeros at $z$-values at lattice points in $\tau\Z + \Z$, with the poles occurring in a cone and the zeros occurring in the opposite cone. It may be shown using the double gamma product formula \eqref{eq:dgprod} that the locations of the poles and zeros remain the same for $(z,\tau) \in \C \times \DD_{\!A}$. It may further be shown that the poles and zeros of $\sf{\r}{\!A}{\tau}$ occur at discrete sets of rational numbers.
\end{rmk}

\begin{rmk}
For $(\m,A) \in \Z^2 \semidirect \SL_2(\Z)$, the function $\sigma_{\m,A}(z,\tau)$ can be written as a ratio of $q$-Pochhammer symbols not only on the upper half-plane $\HH$, but also on the lower half-plane $-\HH$. This is done in \cite{dimofte}. The function $\sigma_{\m,A}(z,\tau)$ can thus be sensibly defined on $\C \setminus \{j_{\!A}(\tau) \leq 0\}$ in all cases; this is discussed further in \cite{afk}.
\end{rmk}

\subsection{The Shintani--Faddeev modular cocycle with half-integral characteristics}\label{sec:half}

As an aside, we deal with the special cases when $\r \in \foh\Z^2$, in which we obtain simple expressions for $\sf{\r}{A}{\tau}$.
If $\r \in \foh\Z$, then $-I \in \Gamma_\r$, and the identity $\sf{\r}{A}{\tau} = \sf{\r}{-A}{\tau}$ implies that $\sf{\r}{A}{\tau}$ is meromorphic on $\DD_{\!A} \cup \DD_{-A} = \C \setminus \{-d/c\}$. 
Moreover, from \eqref{eq:modspecial}, we have the relations
\begin{align}
\varpi_{\tiny\smcoltwo{0}{1}}(\tau) 
&= q^{-1/24} \eta(\tau); \\
\varpi_{\tiny\smcoltwo{0}{1/2}}(\tau) 
&= q^{1/48} \frac{\eta(\tau/2)}{\eta(\tau)}; \\
\varpi_{\tiny\smcoltwo{1/2}{1}}(\tau) 
&= q^{-1/24} \frac{\eta(2\tau)}{\eta(\tau)}; \\
\varpi_{\tiny\smcoltwo{1/2}{1/2}}(\tau) 
&= \zeta_{48}q^{1/48} \frac{\eta((\tau+1)/2)}{\eta(\tau)}. 
\end{align}
For all $A \in \SL_2(\Z)$, let $\hep_{\!A}(\tau) = \ep(\tau)\psi(A,\ep)$ for any choice of square root function $\ep(\tau)^2 = j_{\!A}(\tau)$; the relation $\hep_{\!A}(\tau) = \frac{\eta(A\cdot\tau)}{\eta(\tau)}$ shows that there is no dependence on the choice of square root. Let $P = \smmattwo1002$, $Q = \smmattwo2001$, and $R=\smmattwo1102$. We obtain the following relations.
\begin{align}
\mbox{For $A \in \SL_2(\Z)$: }
\sf{\tiny\smcoltwo{0}{1}}{A}{\tau} 
&= \frac{\ee{-\tfrac{A\cdot\tau}{24}} \eta(A\cdot\tau)}{\ee{-\tfrac{\tau}{24}} \eta(\tau)} \\
&= \ee{\tfrac{\tau - A\cdot\tau}{24}}\hep_{\!A}(\tau). \label{eq:shin01} \\
\mbox{For $A \in \Gamma_{\tiny\smcoltwo{0}{1/2}}$: }
\sf{\tiny\smcoltwo{0}{1/2}}{A}{\tau}
&= \frac{\ee{\tfrac{A\cdot\tau}{48}} \eta(\tfrac{A\cdot\tau}{2}) / \eta(A\cdot\tau)}{\ee{\frac{\tau}{48}} \eta(\tfrac{\tau}{2}) / \eta(\tau)} \\
&= \ee{\tfrac{A\cdot\tau - \tau}{48}}\frac{\eta(PAP^{-1}\cdot(P\cdot\tau))/\eta(P\cdot\tau)}{\eta(A\cdot\tau)/\eta(\tau)} \\
&= \ee{\tfrac{A\cdot\tau - \tau}{48}}\frac{\hep_{\!PAP^{-1}}(P\cdot\tau)}{\hep_{\!A}(\tau)}. \\
\mbox{For $A \in \Gamma_{\tiny\smcoltwo{1/2}{1}}$: }
\sf{\tiny\smcoltwo{1/2}{1}}{A}{\tau}
&= \frac{\ee{-\tfrac{A\cdot\tau}{24}} \eta(2(A\cdot\tau)) / \eta(A\cdot\tau)}{\ee{-\tfrac{\tau}{24}} \eta(2\tau) / \eta(\tau)} \\
&= \ee{\tfrac{\tau - A\cdot\tau}{24}}\frac{\eta(QAQ^{-1}\cdot(Q\cdot\tau))/\eta(Q\cdot\tau)}{\eta(A\cdot\tau)/\eta(\tau)} \\
&= \ee{\tfrac{\tau - A\cdot\tau}{24}}\frac{\hep_{\!QAQ^{-1}}(Q\cdot\tau)}{\hep_{\!A}(\tau)}. \\
\mbox{For $A \in \Gamma_{\tiny\smcoltwo{1/2}{1/2}}$: }
\sf{\tiny\smcoltwo{1/2}{1/2}}{A}{\tau} 
&= \frac{\ee{\tfrac{A\cdot\tau}{48}} \eta(\frac{A\cdot\tau+1}{2}) / \eta(A\cdot\tau)}{\ee{\tfrac{\tau}{48}} \eta(\tfrac{\tau+1}{2}) / \eta(\tau)} \\
&= \ee{\tfrac{A\cdot\tau - \tau}{48}}\frac{\eta(RAR^{-1}\cdot(R\cdot\tau))/\eta(R\cdot\tau)}{\eta(A\cdot\tau)/\eta(\tau)} \\
&= \ee{\tfrac{A\cdot\tau - \tau}{48}}\frac{\hep_{\!RAR^{-1}}(R\cdot\tau)}{\hep_{\!A}(\tau)}.
\end{align}
In the last three cases, it is noteworthy that: 
\begin{align}
\left(\frac{\hep_{\!PAP^{-1}}(P\cdot\tau)}{\hep_{\!A}(\tau)}\right)^2
&= \frac{j_{\!PAP^{-1}}(\tau)\psi^2(PAP^{-1})}{j_{\!A}(\tau)\psi^2(A)}
= \frac{j_{\!PAP^{-1}}(\tau)}{j_{\!A}(\tau)}
= \frac{j_{\!P}(A\cdot\tau)}{j_{\!P}(\tau)}
= \frac{2}{2}
= 1; \\
\left(\frac{\hep_{\!QAQ^{-1}}(Q\cdot\tau)}{\hep_{\!A}(\tau)}\right)^2
&= \frac{j_{\!QAQ^{-1}}(\tau)\psi^2(QAQ^{-1})}{j_{\!A}(\tau)\psi^2(A)}
= \frac{j_{\!QAQ^{-1}}(\tau)}{j_{\!A}(\tau)}
= \frac{j_{\!Q}(A\cdot\tau)}{j_{\!Q}(\tau)}
= \frac{1}{1}
= 1; \\
\left(\frac{\hep_{\!RAR^{-1}}(R\cdot\tau)}{\hep_{\!A}(\tau)}\right)^2
&= \frac{j_{\!RAR^{-1}}(\tau)\psi^2(RAR^{-1})}{j_{\!A}(\tau)\psi^2(A)}
= \frac{j_{\!RAR^{-1}}(\tau)}{j_{\!A}(\tau)}
= \frac{j_{\!R}(A\cdot\tau)}{j_{\!R}(\tau)}
= \frac{2}{2}
= 1.
\end{align}
Therefore, 
$\sf{\tiny\smcoltwo{0}{1/2}}{A}{\tau} = \pm \ee{\tfrac{A\cdot\tau - \tau}{48}}$,
$\sf{\tiny\smcoltwo{1/2}{1}}{A}{\tau} = \pm \ee{\tfrac{\tau - A\cdot\tau}{24}}$, and
$\sf{\tiny\smcoltwo{1/2}{1/2}}{A}{\tau} = \pm \ee{\tfrac{A\cdot\tau - \tau}{48}}$,
where the signs depend (only) on $A$.

\subsection{The double sine function}

The Shintani--Faddeev Jacobi cocycle of $A=S = \smmattwo{0}{-1}{1}{0}$ is closely related to Shintani's \textit{double sine function}. The double sine function is defined in terms of the \textit{double gamma function} originally studied by Barnes \cite{barnes}.
\begin{defn}
For $\omega_1, \omega_2 \in \R_+$ and $z \in \R$ such that $-z \nin \omega_1\N_0 + \omega_2\N_0$, and $s \in \C$ such that $\re(s)>2$, define the \textit{double zeta function} by
\begin{equation}
\zeta_2(s, z; \omega_1,\omega_2) = \sum_{m=0}^\infty \sum_{n=0}^\infty (z+\omega_1 m + \omega_2 n)^{-s}.
\end{equation}
This function is holomorphically continued to $\omega_1, \omega_2 \in \C$ and $z \in \C$ such that 
$\re(\omega_1)>0$, $\re(\omega_2)>0$, and $\re(z)>0$, 
and $s \in \C \setminus \{1,2\}$ by the contour integral formula
\begin{equation}
\zeta_2(s, z; \omega_1,\omega_2) = \frac{\Gamma(1-s)}{2\pi i} \int_C \frac{e^{-zt}}{(1-e^{-\omega_1 t})(1-e^{-\omega_2 t})} (-t)^s \frac{dt}{t},
\end{equation}
where $C$ is a contour following the real line from below staring at $\infty - i \varepsilon$, going clockwise around zero, and following the real line from above to $\infty + i \varepsilon$. See \cite[\S 38]{barnes} for a proof of this formula.

Define \textit{double gamma function} by
\begin{equation}
\Gamma_2(z; \omega_1,\omega_2) := \rho_2(\omega_1,\omega_2) \exp\!\left(\left.\frac{d}{ds} \zeta_2(s, z; \omega_1,\omega_2)\right|_{s=0}\right),
\end{equation}
where $\rho_2(\omega_1,\omega_2)$ is a nonzero constant independent of $z$ whose exact value is irrelevant to our considerations.
The double gamma function has a product formula
\begin{align}
\Gamma_2(z;\omega_1,\omega_2)^{-1}\label{eq:dgprod} 
&= z \exp\!\left(\gamma_{22}(\omega_1,\omega_2)z + \foh\gamma_{21}(\omega_1,\omega_2)z^2\right) \\
& \ \ \ \times \!\prod_{\substack{m,n \geq 0 \\ (m,n) \neq 0}} \left(1+\frac{z}{m\omega_1+n\omega_2}\right)\exp\!\left(-\frac{z}{m\omega_1+n\omega_2}+\frac{z^2}{2(m\omega_1+n\omega_2)^2}\right), 
\end{align}
which extends $\Gamma_2$ to a meromorphic function of $z \in \C$ and $\omega_1, \omega_2 \in \C \setminus \{0\}$ such that $\frac{\omega_2}{\omega_1}$ is not a negative real number. See \cite[\S 19--24]{barnes} for a proof of this formula and the definition of $\gamma_{22}(\omega_1,\omega_2)$ and $\gamma_{21}(\omega_1,\omega_2)$.

Define the \textit{double sine function} (which does not depend on the value of $\rho_2(\omega_1,\omega_2)$) by 
\begin{equation}\label{eq:dsinegamma}
\SS_2(z;\omega_1,\omega_2) = \frac{\Gamma_2(\omega_1+\omega_2-z;\omega_1,\omega_2)}{\Gamma_2(z;\omega_1,\omega_2)}.
\end{equation}
(This function is called $S_2$ by Koyama and Kurokawa \cite{sinevalues} and Tagedal \cite{tangedal}.)
We will often specialize to the case $\omega_2=1$, in which case we set
\begin{align}
\Gamma_2(z, \tau) &:= \Gamma_2(z; \tau,1), \\
\SS_2(z,\tau) &:= \SS_2(z;\tau,1).
\end{align}
\end{defn}

\begin{thm}[Shintani]\label{thm:shin5}
We have the following relation between double sine function and the Shintani--Faddeev Jacobi cocycle of $A = S = \smmattwo{0}{-1}{1}{0}$:
\begin{equation}
\sfj{S}{z}{\tau} = \ee{\frac{\tau-3+\tau^{-1}}{24} + \frac{(\tau-z)(1-z)}{4\tau}} \left(1-\ee{\frac{z}{\tau}}\right) \SS_2(z,\tau)^{-1}.
\end{equation}
\end{thm}
\begin{proof}
This is a rephrasing of \cite[Prop.~5]{shintani}.
\end{proof}

The restriction $\omega_2=1$ is not a serious restriction, by the following identity.
\begin{prop}\label{prop:scaling}
For any $\alpha \in \C^\times$,
\begin{equation}\label{eq:hello1}
\SS_2(\alpha z;\alpha \omega_1,\alpha \omega_2) = \SS_2(z,\omega_1,\omega_2).
\end{equation}
In particular,
\begin{equation}
\SS_2(z;\omega_1,\omega_2) = \SS_2\!\left(\frac{z}{\omega_2};\frac{\omega_1}{\omega_2}\right). \label{eq:hello2}
\end{equation}
\end{prop}
\begin{proof}
Consider the Laurent series expansion
\begin{align}
\tfrac{e^{-zt}}{(1-e^{-\omega_1 t})(1-e^{-\omega_2 t})}
&= \tfrac{1}{\omega_1\omega_2}t^{-2} 
+ \tfrac{\omega_1+\omega_2-2z}{2\omega_1\omega_2}
+ \tfrac{\omega_1^2+3\omega_1\omega_2+\omega_2^2 - 6z(\omega_1+\omega_2-z)}{12\omega_1\omega_2}
+ O(t).
\end{align}
By the residue theorem, we have
\begin{align}
\zeta_2(0,z;\omega_1,\omega_2)
&= \tfrac{1}{2\pi i} \int_C \tfrac{e^{-zt}}{(1-e^{-\omega_1 t})(1-e^{-\omega_2 t})} \frac{dt}{t}
= \tfrac{\omega_1^2+3\omega_1\omega_2+\omega_2^2 - 6z(\omega_1+\omega_2-z)}{12\omega_1\omega_2}.
\end{align}
Moreover, for $\alpha > 0$, we have
\begin{equation}
\zeta_2(s,\alpha z;\alpha\omega_1,\alpha\omega_2) = \alpha^{-s} \zeta_2(s,z;\omega_1,\omega_2).
\end{equation}
Taking the derivative in $s$,
\begin{equation}
\zeta_2'(s,\alpha z;\alpha \omega_1,\alpha \omega_2)
= \alpha^{-s} \zeta_2'(s,z;\omega_1,\omega_2)
-  (\log\alpha) \alpha^{-s} \zeta_2(s,z;\omega_1,\omega_2),
\end{equation}
and thus
\begin{align}
\zeta_2'(0,\alpha z;\alpha \omega_1,\alpha \omega_2)
&= \zeta_2'(0,z;\omega_1,\omega_2) \\
& \ \ \ - (\log\alpha) \tfrac{\omega_1^2+3\omega_1\omega_2+\omega_2^2 - 6z(\omega_1+\omega_2-z)}{12\omega_1\omega_2}; \\
\zeta_2'(0,\alpha (\omega_1+\omega_2-z));\alpha \omega_1,\alpha \omega_2)
&= \zeta_2'(0,\omega_1+\omega_2-z;\omega_1,\omega_2) \\
& \ \ \ - (\log\alpha) \tfrac{\omega_1^2+3\omega_1\omega_2+\omega_2^2 - 6z(\omega_1+\omega_2-z)}{12\omega_1\omega_2}.
\end{align}
Thus,
\begin{align}
\SS_2(\alpha z;\alpha\omega_1,\alpha\omega_2)
&= \exp\left(\zeta_2'(0,\alpha (\omega_1+\omega_2-z));\alpha \omega_1,\alpha \omega_2) - \zeta_2'(0,\alpha z;\alpha \omega_1,\alpha \omega_2)\right) \\
&= \exp\left(\zeta_2'(0,\omega_1+\omega_2-z); \omega_1, \omega_2) - \zeta_2'(0, z; \omega_1, \omega_2)\right) \\
&= \SS_2(z;\omega_1,\omega_2),
\end{align}
proving \eqref{eq:hello1} for $\alpha>0$. The claim follows for all $\alpha \in \C^\times$ because the ratio of the two sides is a meromorphic function in $\alpha$. Equation \eqref{eq:hello2} follows by a specialization of the variables.
\end{proof}

The double sine function satisfies the following ``quasiperiodicity'' properties with respect to the elliptic transformations $z \mapsto z+1$ and $z \mapsto z+\tau$.
\begin{prop}\label{prop:quasiperiodicity}
The double sine function satisfies the identities
\begin{align}
\Sin_2(z+1,\tau) &= \left(2\sin\!\left(\frac{\pi z}{\tau}\right)\right)^{-1} \Sin_2(z,\tau) \mbox{ and} \\
\Sin_2(z+\tau,\tau) &= \left(2\sin(\pi z)\right)^{-1} \Sin_2(z,\tau).
\end{align}
\end{prop}
\begin{proof}
This is a special case of \cite[Thm.\ 2.1]{sine}.
\end{proof}

The reader should be aware of different conventions for related functions appearing in the physics literature, denoted by $S_b(z)$ and $\Phi_b(z)$.
\begin{defn}
The \textit{physicist's double sine function} is
\begin{equation}
S_b(z) := \SS_2(z;b,b^{-1})^{-1}.
\end{equation}
Faddeev's \textit{noncompact quantum dilogarithm} is defined by
\begin{equation}\label{eq:faddefn}
\Phi_b(z) = \exp\left(\frac{1}{4}\int_C \frac{e^{-2izw}}{\sinh(wb)\sinh(wb^{-1})} \frac{dw}{w}\right),
\end{equation}
where $C$ is a contour that follows the real line from $-\infty$ the $\infty$ except near zero, where it goes into the upper half plane to avoid the pole at $w=0$.
\end{defn}
The noncompact quantum dilogarithm may be expressed in terms of the Shintani--Faddeev Jacobi cocycle as follows.
\begin{prop}
For $\im(b^2)>0$, we have
\begin{equation}\label{eq:fadrel1}
\Phi_b(z) 
= \sfj{S}{\frac{1}{2} - \frac{b^2}{2} - izb}{b^2}.
\end{equation}
\end{prop}
\begin{proof}
This formula is given as \cite[eq.~(53)]{resurgence}.
\end{proof}

There is a fast-converging integral formula for the logarithm of the double sine function. While we don't require if for our main results, we have used it to check formulas numerically in Mathematica.
\begin{prop}\label{prop:sinhintegral}
If $0 < \re(z) < \re(\omega_1+\omega_2)$, then
\begin{align}
\SS_2(z;\omega_1,\omega_2) 
&= \exp\left(-\int_0^\infty \left(\frac{\sinh\left(\left(\frac{\omega_1+\omega_2}{2}-z\right)t\right)}{2\sinh\left(\frac{\omega_1 t}{2}\right)\sinh\left(\frac{\omega_2 t}{2}\right)} - \frac{\omega_1+\omega_2-2z}{\omega_1 \omega_2 t}\right)\frac{dt}{t}\right).
\end{align}
\end{prop}
\begin{proof}
This formula is stated in the case of $S_b(z) = \SS_2(z;b,b^{-1})^{-1}$ in \cite[App.~B]{ponsot}. The general case follows by \Cref{prop:scaling}.
\end{proof}

\subsection{Multiplicative quantum modularity of the $q$-Pochhammer symbol}

From the definition of $\sigma_{\!A}$, we have $\varpi\left(\frac{z}{c\tau+d},A\cdot\tau\right) = \sfj{A}{z}{\tau}\varpi(z,\tau)$. In order to show that $\varpi$ is a $\sigma$-Jacobi form in a nontrivial sense, we need to show that $\sfj{A}{z}{\tau}$ are defined on larger domains than $\varpi$ is.
\begin{theorem}\label{thm:wannabej}
The function $\varpi(z,\tau)$ is a meromorphic $\sigma_{\!A}$-Jacobi form for $\Gamma = \SL_2(\Z)$ with the system of domains $\DD_{\!A}$ defined by \eqref{eq:DD}.
Specifically, it satisfies the elliptic relation
\begin{equation}
\varpi(z+m\tau+n,\tau) = \left(\ee{z},\ee{\tau}\right)_m^{-1} \varpi(z,\tau)
\end{equation}
and the modular relation
\begin{equation}
\varpi\!\left(\frac{z}{j(z,\tau)},A \cdot \tau\right) = \sfj{A}{z}{\tau} \varpi(z,\tau),
\end{equation}
and $\sfj{A}{z}{\tau}$ is meromorphic on $\C \times \DD_{\!A}$.
\end{theorem}
\begin{proof}
The elliptic relation is straightforward to check. The modular relation holds by the definition of $\sfj{A}{z}{\tau}$; however, we must prove that $\sfj{A}{z}{\tau}$ meromorphically continues to $(z,\tau) \in \C \times \DD_{\!A}$.  Write $A = \smmattwo{a}{b}{c}{d}$.

First, note that
\begin{equation}
\varpi(z,\tau+b) = \varpi(z,\tau).
\end{equation}
So, when $c=0$ and $d>0$, $\sfj{A}{z}{\tau} = 1$ (and thus is defined for $(z,\tau) \in \C \times \C$). When $c=0$ and $d<0$, there is nothing to prove.

Now suppose $c>0$. Divide $a$ by $c$ with negative remainder to obtain $a=ck-c'$ for some $k \in \Z$ and some $c' \in \Z$ with $0 \leq c' < c$. Set
\begin{equation}
B = \mattwo{a'}{b'}{c'}{d'} = S^{-1}T^{-k}A= \mattwo{c}{d}{ck-a}{dk-b}.
\end{equation} 
We then have $A=T^kSB$, so by the cocycle condition, 
\begin{align}
\sfj{A}{z}{\tau} 
&= \sfj{T^k}{\frac{z}{j_{\!SB}(\tau)}}{SB \cdot \tau}\sfj{S}{\frac{z}{j_{\!B}(\tau)}}{B\cdot\tau}\sfj{B}{z}{\tau} \\
&= \sfj{S}{\frac{z}{j_{\!B}(\tau)}}{B\cdot\tau}\sfj{B}{z}{\tau}.
\end{align}
By \Cref{thm:shin5}, $\sfj{S}{z'}{\tau}$ analytically continues to all $\tau \in \C \setminus (-\infty,0]$. Thus, $\sfj{S}{\frac{z}{j_{\!B}(\tau)}}{B\cdot\tau}$ continues (in $\tau$) to the lower half-plane and the portion of the real line where $B\cdot\tau > 0$. We have $B\cdot\tau = \frac{a'\tau+b'}{c'\tau+d'}= \frac{c\tau+d}{c'\tau+d'}$, and $cd'-dc'=a'd'-b'c'=1$, so $\frac{d'}{c'}-\frac{d}{c} = \frac{1}{cc'} > 0$ and $\frac{d'}{c'} > \frac{d}{c}$. Thus, $\tau > -\frac{d}{c}$ implies that $B \cdot \tau > 0$ and $\sfj{S}{\frac{z}{j_{\!B}(\tau)}}{B\cdot\tau}$ is well-defined. The inequality $c'<c$ allows us to induct on $c'$, and again using the fact the $\left(-\infty,-\frac{d'}{c'}\right) \subseteq \left(-\infty,-\frac{d}{c}\right)$, we have shown that $\sfj{A}{z}{\tau}$ meromorphically continues to $\tau \in \DD_{\!A}$.

If $c < 0$, we use the relation $1 = \sfj{A^{-1}A}{z}{\tau} = \sfj{A^{-1}}{\frac{z}{j_{\!A}(\tau)}}{A\cdot\tau}\sfj{A}{z}{\tau}$. Since $A^{-1} = \smmattwo{d}{-b}{-c}{a}$, the result on positive $c$ tells us that $\s_{\!A}(\tau)$ meromorphically continues the lower half-plane and to real $\tau$ such that $A\cdot\tau > \frac{a}{c}$. It is straightforward to check that this is equivalent to the condition that $\tau < -\frac{d}{c}$.
\end{proof}

\begin{theorem}\label{thm:wannabem}
Let $\r \in \Q^2$. The function $\varpi_\r(\tau)$ is a $\shin^\r$-modular form for $\Gamma_\r$ with the system of domains $\tDD_{\!A}$ as defined in \eqref{eq:tDD}.
Specifically, it satisfies the modular relation
\begin{equation}
\varpi_\r(A \cdot \tau) = \sf{\r}{A}{\tau} \varpi_\r(\tau),
\end{equation}
and $\sf{\r}{A}{\tau}$ is meromorphic on $\tDD_{\!A}$.
\end{theorem}
\begin{proof}
The modular relation follows by the definition of $\sf{\r}{A}{\tau}$. By \Cref{prop:goose}, 
\begin{align}
\sf{\r}{A}{\tau} &= \left(\ee{\sympt{A^{-1}\r}{\tau}},\ee{\tau}\right)_{\symp{\r}{(I-A)\smcoltwo{1}{0}}} \sfj{A}{\sympt{A^{-1}\r}{\tau}}{\tau}. 
\end{align}
Thus, $\sf{\r}{A}{\tau}$ extends to a meromorphic function on $\DD_{\!A}$. The domain $\tDD_{\!A} = \DD_{\!A}$ unless $c=0$ and $d<0$. This final case only arises when $-I \in \Gamma_\r$, that is, when $\r \in \foh\Z^2$, and in this case, we may replace $A$ by $-A$ to show that $\sf{\r}{A}{\tau}$ is meromorphic on $\C$.
\end{proof}

\subsection{Functional equations of the Shintani--Faddeev cocycles}\label{sec:further}

In this section, we obtain functional equations for $\sfj{A}{z}{\tau}$ and $\sf{\r}{A}{\tau}$, using modularity of classical theta functions discussed in \Cref{sec:prelim}.

\begin{thm}
For $A \in \SL_2(\Z)$, $z \in \C$, and $\tau \in \DD_{\!A}$,
\begin{equation}
\sfj{A}{z}{\tau}\sfj{A}{-z}{\tau} = 
\psi^2(A) \frac{\ee{\tfrac{z}{2(c\tau+d)}}-\ee{-\tfrac{z}{2(c\tau+d)}}}{\ee{\tfrac{z}{2}}-\ee{-\tfrac{z}{2}}} \ee{\tfrac{\tau - A\cdot\tau}{12} + \tfrac{cz^2}{2(c\tau+d)}}.
\end{equation}
\end{thm}
\begin{proof}
Define the function $f(z,\tau)$ for $z \in \C$ and $\tau \in \HH$ by
\begin{equation}\label{eq:fdefn}
f(z,\tau) := \frac{\vartheta_1(z,\tau)}{\eta(\tau)}.
\end{equation}
This function is a meromorphic Jacobi form of weight 0. Specifically, by \eqref{eq:etatrans} and \Cref{thm:thetatrans}, it satisfies the following modular transformation law for $A = \smmattwo{a}{b}{c}{d} \in \SL_2(\Z)$.
\begin{equation}\label{eq:ftrans}
f\!\left(\frac{z}{c\tau+d},A\cdot\tau\right) = \psi^2(A) \ee{\frac{cz^2}{2(c\tau+d)}} f(z,\tau).
\end{equation}
By the Jacobi triple product formula, specifically by \eqref{eq:jacprod2},
\begin{align}
\varpi(z,\tau)\varpi(-z,\tau) 
&= -i\ee{-\tfrac{\tau}{12}}\left(\ee{\tfrac{z}{2}}-\ee{-\tfrac{z}{2}}\right)f(z,\tau). \label{eq:jacagain}
\end{align}
By definition, we have
\begin{equation}
\sfj{A}{z}{\tau}\sfj{A}{-z}{\tau} = \frac{\varpi\!\left(\frac{z}{c\tau+d},A\cdot\tau\right)\varpi\!\left(-\frac{z}{c\tau+d},A\cdot\tau\right)}{\varpi(z,\tau)\varpi(-z,\tau)}.
\end{equation}
Thus,
\begin{align}
\sfj{A}{z}{\tau}\sfj{A}{-z}{\tau}
&= \frac{-i\ee{-\tfrac{A\cdot\tau}{12}}\left(\ee{\tfrac{z}{2(c\tau+d)}}-\ee{-\tfrac{z}{2(c\tau+d)}}\right)}{-i\ee{-\tfrac{\tau}{12}}\left(\ee{\tfrac{z}{2}}-\ee{-\tfrac{z}{2}}\right)} \times \frac{f\!\left(\frac{z}{c\tau+d},A\cdot\tau\right)}{f(z,\tau)} \\
&= \frac{\ee{-\tfrac{A\cdot\tau}{12}}\left(\ee{\tfrac{z}{2(c\tau+d)}}-\ee{-\tfrac{z}{2(c\tau+d)}}\right)}{\ee{-\tfrac{\tau}{12}}\left(\ee{\tfrac{z}{2}}-\ee{-\tfrac{z}{2}}\right)} \times \psi^2(A) \ee{\tfrac{cz^2}{2(c\tau+d)}} \\
&= \psi^2(A) \frac{\ee{\tfrac{z}{2(c\tau+d)}}-\ee{-\tfrac{z}{2(c\tau+d)}}}{\ee{\tfrac{z}{2}}-\ee{-\tfrac{z}{2}}} \ee{\tfrac{\tau - A\cdot\tau}{12} + \tfrac{cz^2}{2(c\tau+d)}}. \tag*{\qedhere}
\end{align} 
\end{proof}

We will need a version of the previous theorem for the Shintani--Faddeev cocycle with characteristics, $\sf{\r}{A}{\tau}$.
\begin{thm}\label{thm:funchar}
Let $\r \in \Q^2 \setminus \Z^2$, $A \in \Gamma_\r$, and $\tau \in \tDD_{\!A}$.
We have the identity
\begin{align}\label{eq:funchar}
\sf{\r}{A}{\tau}\sf{-\r}{A}{\tau}
&= \psi^2(A)\chi_\r(A) \ee{\left(\tfrac{r_2^2}{2}+\tfrac{1}{12}\right)(\tau - A\cdot\tau)} \frac{\ee{\tfrac{r_2(A\cdot\tau)-r_1}{2}}-\ee{\tfrac{-r_2(A\cdot\tau)+r_1}{2}}}{\ee{\tfrac{r_2\tau-r_1}{2}}-\ee{\tfrac{-r_2\tau+r_1}{2}}}.
\end{align}
\end{thm}
\begin{proof}
For $\tau \in \HH$, define the function $f_\r(\tau) := \frac{\theta_\r(\tau)}{\eta(\tau)}$.
By \Cref{thm:thetamod} and \eqref{eq:etatrans}, we have the modular transformation law
\begin{equation}\label{eq:frtrans}
f_\r(A\cdot\tau) = \psi^2(A)\chi_\r(A)f_\r(\tau).
\end{equation}
Taking $z=0$ in \Cref{prop:jacprod3},
\begin{align}\label{eq:varpipm}
\varpi_\r(\tau)\varpi_{-\r}(\tau)
&= i \ee{-\left(\tfrac{r_2^2}{2}+\tfrac{1}{12}\right)\tau-r_2\left(-r_1+\foh\right)}\left(\ee{\tfrac{r_2\tau-r_1}{2}}-\ee{\tfrac{-r_2\tau+r_1}{2}}\right) f_\r(\tau).
\end{align}
By definition, 
\begin{equation}\label{eq:shinpmfrac}
\sf{\r}{A}{\tau}\sf{-\r}{A}{\tau} 
= \frac{\varpi_\r(A\cdot\tau)\varpi_{-\r}(A\cdot\tau)}{\varpi_\r(\tau)\varpi_{-\r}(\tau)}.
\end{equation}
Applying \eqref{eq:varpipm} to \eqref{eq:shinpmfrac}, and then using \eqref{eq:frtrans} to simplify the resulting expression, yields \eqref{eq:funchar} for $\tau \in \HH$.
The identity extends to $\tau \in \tDD_{\!A}$ by analytic continuation.
\end{proof}

\subsection{Values at rational $\tau$ and quantum modularity}\label{sec:rational}

In this section, we'll evaluate $\sfj{A}{z}{\tau}$ for $\tau = \frac{m}{n} \in \DD_{\!A} \cap \Q$. 
The values of $\sfj{A}{z}{\tau}$ at rational $\tau$ are not needed for the main results of this paper. We hope that they may be useful in the future---perhaps further study of these values could lead to a definition of a $p$-adic analogue of the Shintani--Faddeev cocycle using $p$-adic interpolation. They also suggest a connection to quantum modularity.

Our formula for $\sfj{A}{z}{\tau}$ will be stated in \Cref{prop:qcdid} with a restriction on the value of $z$. This restriction could potentially be removed by a careful accounting of branch cuts.

The \textit{cyclic quantum dilogarithm} $D_\zeta(w)$ is a finite product appearing in the asymptotic formula describing the behavior of the $q$-Pochhamer symbol as $q$ approaches an $n$-th root of unity $\zeta$. The cyclic quantum dilogarithm plays a role in Garoufalidis and Zagier's work on Nahm's conjecture on the modularity of certain $q$-hypergeometric series \cite{garzag}. It is defined as
\begin{equation}\label{eq:qcd}
D_\zeta(w) := \prod_{k=1}^{n-1} (1-\zeta^k w)^{k}.
\end{equation}

The behavior of the $q$-Pochhammer symbol as $q$ approaches a primitive $n$-th root of unity $\zeta$ is described as follows.
\begin{prop}\label{prop:bazhanov1}
Let $\zeta$ be a primitive $n$-th root of $1$.
When $\abs{w} < 1$ and $q = e^{-t/n^2}\zeta$ for $\re(t)>0$,
\begin{equation}\label{eq:bazhanov}
(w;q)_\infty = R(w^n,t)D_\zeta(w)^{-1/n}(1+O(t)) \mbox{ as } t \to 0,
\end{equation}
where
\begin{align}
R(x,t) &:= (1-x)^{1/2}\exp(-\Li_2(x)/t)
\end{align}
and the root $D_\zeta(w)^{-1/n}$ is defined by taking the $(-n)$-th root of each term in the product \eqref{eq:qcd} using the standard branch of the logarithm.
\end{prop}
\begin{proof}
This is \cite[Prop.~3.2]{bazhanov}.
\end{proof}

We can evaluate $\sfj{A}{z}{\tau}$ at $\tau = \frac{m}{n}$ in terms of $D_\zeta(w)$, giving a ``quantum modularity'' property for a function related to $D_\zeta(w)$.

\begin{prop}\label{prop:qcdid}
Let $A = \smmattwo{a}{b}{c}{d} \in \SL_2(\Z)$, and let $m,n \in \Z$, $\gcd(m,n)=1$, $n>0$ such that $j_{\!A}\!\left(\frac{m}{n}\right) > 0$. 
For $\im(z)>0$,
\begin{equation}\label{eq:qcdid}
\sfj{A}{z}{\frac{m}{n}} = \frac{D_{{\rm e}\left(\frac{am+bn}{cm+dn}\right)}\!\left(\ee{\frac{nz}{cm+dn}}\right)^{-1/(cm+dn)}}{D_{{\rm e}\left(\frac{m}{n}\right)}\!\left(\ee{z}\right)^{-1/n}}.
\end{equation}
\end{prop}
\begin{proof}
Suppose $z \in \C$ and $\tau \in \HH$; let $\tilde{z} = \frac{z}{c\tau+d}$ and $\tilde{\tau} = A \cdot \tau$. Suppose also that $\tau = \frac{m}{n} - \frac{t}{2\pi i n^2}$ for a small complex number $t$ with positive real part, and let $\tilde{m}=am+bn, \tilde{n}=cm+dn$. We have
\begin{equation}
\sfj{A}{z}{\tau} = \frac{\left(\ee{\tilde{z}};\ee{\tilde{\tau}}\right)_\infty}{\left(\ee{z};\ee{\tau}\right)_\infty}.
\end{equation}
Let $\zeta = \ee{\frac{m}{n}}$. If $\tau = \frac{m}{n} - \frac{t}{2\pi i n^2}$ then as $t \to 0$,
\begin{align}
\tilde{\tau} = \frac{a\tau+b}{c\tau+d} 
&= \frac{a\frac{m}{n}+b}{c\frac{m}{n}+d} - \frac{t}{2\pi i n^2\left(c\frac{m}{n}+d\right)^2} + O(t^2) \\
&= \frac{am+bn}{cm+dn} - \frac{t}{2\pi i (cm+dn)^2} + O(t^2) \\
&= \frac{\tilde{m}}{\tilde{n}} - \frac{t}{2\pi i \tilde{n}^2} + O(t^2).
\end{align}
Also, as $t \to 0$,
\begin{align}
\tilde{z} &= \frac{nz}{cm+dn} + O(t).
\end{align}
Let $\tilde{z}' = \frac{nz}{cm+dn}$ and $\tilde{\tau}' = \frac{\tilde{m}}{\tilde{n}} - \frac{t}{2\pi i \tilde{n}^2}$. Then, 
it follows from the asymptotic formula in \eqref{eq:bazhanov} and the continuity in $w$ of the functions on the right-hand side that
\begin{equation}
\lim_{t \to 0} \frac{(\ee{\tilde{z}};\ee{\tilde{\tau}})_\infty}{(\ee{\tilde{z}'};\ee{\tilde{\tau}'})_\infty} = 1. \label{eq:tlim1}
\end{equation}
Let $\zeta = \ee{\frac{m}{n}}$ and $\tilde{\zeta} = \ee{\frac{\tilde{m}}{\tilde{n}}}$. As $t \to 0$, we have by \Cref{prop:bazhanov1} that 
\begin{align}
(\ee{z};\ee{\tau})_\infty &= R(\ee{nz},t)D_\zeta(\ee{z})^{-1/n}(1+O(t)) \mbox{ and }\\
(\ee{\tilde{z}'};\ee{\tilde{\tau}'})_\infty &= R(\ee{\tilde{n}\tilde{z}'},t)D_{\tilde{\zeta}}(\ee{\tilde{z}'})^{-1/\tilde{n}}(1+O(t)). \label{eq:tlim2}
\end{align}
However, $\tilde{n}\tilde{z}' = (cm+dn)\frac{nz}{cm+dn} = nz$. Thus, by \eqref{eq:tlim1} and \eqref{eq:tlim2},
\begin{equation}
\frac{(\ee{\tilde{z}};\ee{\tilde{\tau}})_\infty}{(\ee{z},\ee{\tau})_\infty} = \frac{D_{\tilde{\zeta}}(\ee{\tilde{z}})^{-1/\tilde{n}}}{D_{\zeta}(\ee{z})^{-1/n}}(1+O(t)).
\end{equation}
Sending $t \to 0$ proves the proposition.
\end{proof}

The restrictions $\abs{w} < 1$ in \Cref{prop:bazhanov1} and $\re(z)>0$ in \Cref{prop:qcdid} are there to avoid branch points of the logarithm and dilogarithm functions and could be removed or relaxed by a more careful analysis. In particular, both sides of \eqref{eq:qcdid} analytically continue in $z$ to a larger domain, so the identity would continue to hold on that larger domain.

Formally, \eqref{eq:qcdid} seems to imply that $f\!\left(z,\frac{m}{n}\right):=-\frac{1}{n}\log D_{{\rm e}(\frac{m}{n})}(\ee{z})$ defines a quantum Jacobi form of weight 0 in the sense of Bringmann and Folsom \cite{bringfol}. (See also Zagier \cite{zagierquantum}.) There are a few issues with this: The function $f$ has branch points at $z \in \frac{1}{n}\Z$, and \Cref{prop:qcdid} is proven only for $\im(z)>0$ and $j_{\!A}\!\left(\frac{m}{n}\right) > 0$. These issues could potentially be resolved by handling branch cuts carefully to extend the domain of \eqref{eq:qcdid} and relaxing the definition of a quantum Jacobi form to allow some undefined values.

The evaluation of $\sf{\r}{A}{\tau}$ at rational $\tau$ is more subtle due to the possible zeros of $D_\zeta(w)$. The details are deferred to later work.

\subsection{Special properties of the Shintani--Faddeev cocycle at real multiplication points}

The Shintani--Faddeev cocycle with characteristics has the following pseudolattice invariance property at fixed points.
\begin{prop}\label{prop:invariance}
Suppose that $\r \in \Q^2$, $A \in \Gamma_\r$, and $\beta \in \tDD_{\!A}$ such that $A\cdot\beta = \beta$. Let $\n \in \Z^2$. We have the invariance relation 
\begin{equation}
\sf{\r+\n}{A}{\beta}
=
\begin{cases}
j_{\!A}(\beta)\,\sf{\r}{\!A}{\beta} & \mbox{if $\r \in \Z^2$ and $r_2 \leq 0 < r_2+n_2$}, \\
j_{\!A}(\beta)^{-1}\,\sf{\r}{\!A}{\beta} & \mbox{if $\r \in \Z^2$ and $r_2+n_2 \leq 0 < r_2$}, \\
\sf{\r}{\!A}{\beta} & \mbox{otherwise}.
\end{cases}
\end{equation}
\end{prop}
\begin{proof}
Let $\textbf{e}_1 = \smcoltwo{1}{0}$ and $\textbf{e}_2 = \smcoltwo{0}{1}$. 
For $\tau \in \HH$, we have $\varpi_{\r+\textbf{e}_1}(\tau) = \varpi_\r(\tau)$, and thus $\sf{\r+\textbf{e}_1}{\!A}{\tau} = \sf{\r}{\!A}{\tau}$; the latter identity holds for all $\tau \in \tDD_{\!A}$ by analytic continuation. 

Suppose that $r_1 \nin \Z$ or $r_2 \neq 0$. Then, $\varpi_{\r+\textbf{e}_2}\!(\tau) = (1+\ee{r_2\tau-r_1})^{-1}\varpi_\r(\tau)$ and thus
\begin{equation}
\sf{\r+\textbf{e}_2}{\!A}{\tau} = \frac{1-\ee{r_2\tau-r_1}}{1-\ee{r_2(A\cdot\tau)-r_1}}\,\sf{\r}{\!A}{\tau} 
\end{equation}
for $\tau \in \HH$. Again by analytic continuation, this identity of meromorphic functions holds for $\tau \in \tDD_{\!A}$. Setting $\tau = \beta$, numerator and denominator become equal, and so $\sf{\r+\textbf{e}_1}{\!A}{\tau} = \sf{\r}{\!A}{\tau}$.

In the case when $r_1 \in \Z$ and $r_2 = 0$, we instead must write
\begin{align}
\shin_{\!A}^{\r+\mathbf{e}_2}(\tau) 
&= \lim_{z \to 0} \sigma_{\r+\textbf{e}_2,A}(z,\tau) \\
&= \lim_{z \to 0} \frac{1-\ee{z}}{1-\ee{j_{\!A}(\tau)^{-1}z}}\sigma_{\r,A}(z,\tau) \\
&= j_{\!A}(\tau)\,\ee{(1-j_{\!A}(\tau)^{-1})z}\sf{\r}{\!A}{\tau}.
\end{align}
Setting $\tau = \beta$, we obtain
$\sf{\r+\textbf{e}_2}{\!A}{\beta} = j_{\!A}(\beta)\,\sf{\r}{\!A}{\beta}$.

The proposition follows by writing $\n=n_1\textbf{e}_1+n_2\textbf{e}_2 \in \Z^2$ and making an induction argument.
\end{proof}

\Cref{prop:invariance} \textit{does not} imply that $\sf{\r}{A}{\beta}$ is doubly periodic as a function of real $\r \in \R^2$; it generally isn't.

At fixed points, the Shintani--Faddeev cocycle with characteristics has a simplified functional equation relating values at $\r$ and $-\r$.
\begin{thm}\label{thm:shincharacter}
Suppose that $\r \in \Q^2 \setminus \Z^2$, $A \in \Gamma_\r$, and $\beta \in \tDD_{\!A}$ such that $A\cdot\beta = \beta$. Then,
\begin{align}
\sf{\r}{A}{\beta}\sf{-\r}{A}{\beta} = \psi^2(A)\chi_\r(A).
\end{align}
\end{thm}
\begin{proof}
Follows from \Cref{thm:funchar} by simplifying \eqref{eq:funchar} in the case when $A \cdot \beta = \beta$.
\end{proof}

The following theorem shows that the real multiplication values of the Shintani--Faddeev modular cocycle satisfy a $\GL_2(\Z)$-invariance property.
\begin{thm}\label{thm:shinconj}
Suppose that  $\r \in \Q^2$, $A \in \Gamma_\r$, $R \in \GL_2(\Z)$, and $\beta \in \tDD_{\!A}$ such that $A \cdot \beta = \beta$.
Let $s_{\!R}(\beta) = \sgn(j_{\!R}(\beta))$, and suppose $s_{\!R}(\beta) \neq 0$.
Then,
\begin{equation}\label{eq:shinconj}
\sf{s_{\!R}(\beta)R\r}{R A R^{-1}}{R\cdot\beta} 
= 
\begin{cases}
\vspace{4pt}
\sf{\r}{A}{\beta} & \mbox{ if } \det(R) = 1, \\
\ol{\sf{\r}{A}{\ol{\beta}}} & \mbox{ if } \det(R) = -1.
\end{cases}
\end{equation}
\end{thm}
\begin{proof}
By \Cref{lem:jeval}, 
$j_{RA R^{-1}}(R\cdot\beta) = j_{\!A}(\beta) = \lambda > 0$, so both values of the Shintani--Faddeev modular cocycle in the statement are well-defined.
Moreover, if $s_{\!R}(\beta) = -1$, then $s_{\!-R}(\beta) = 1$, and \eqref{eq:shinconj} remains the same upon replacing $R$ by $-R$; thus, without loss of generality, we may restrict to the case of $s_{\!R}(\beta) = 1$ (that is, $j_{\!R}(\beta)>0$).

We will first prove \eqref{eq:shinconj} in the case when $\det(R) = 1$.
Write $R = \smmattwo{a}{b}{c}{d}$ and $\r = \smcoltwo{r_1}{r_2}$.
For $\tau \in \HH$, we compute
\begin{align}
\sf{R\r}{RA R^{-1}}{R\cdot\tau} 
&= \frac{\varpi_{R\r}(R\cdot(A\cdot\tau))}{\varpi_{R\r}(R\cdot\tau)} \\
&= \frac{\varpi\!\left((cr_1+dr_2)\frac{a(A\cdot\tau)+b}{c(A\cdot\tau)+d}-(ar_1+br_2),\frac{a(A\cdot\tau)+b}{c(A\cdot\tau)+d}\right)}{\varpi\!\left((cr_1+dr_2)\frac{a\tau+b}{c\tau+d}-(ar_1+br_2),\frac{a\tau+b}{c\tau+d}\right)} \\
&= \frac{\varpi\!\left(\frac{r_2(A\cdot\tau)-r_1}{c(A\cdot\tau)+d},\frac{a(A\cdot\tau)+b}{c(A\cdot\tau)+d}\right)}{\varpi\!\left(\frac{r_2\tau-r_1}{c\tau+d},\frac{a\tau+b}{c\tau+d}\right)}.
\end{align}
By definition, we also have 
\begin{equation}
\sf{\r}{A}{\tau} = \frac{\varpi\!\left(r_2(A\cdot\tau)-r_1,A\cdot\tau\right)}{\varpi\!\left(r_2\tau-r_1,\tau\right)}.
\end{equation}
Dividing, we obtain the identity
\begin{equation}
\frac{\sf{R\r}{RA R^{-1}}{R\cdot\tau}}{\sf{\det(R)\r}{A}{\tau}}
= \frac{\sfj{R}{r_2(A\cdot\tau)-r_1}{A\cdot\tau}}{\sfj{R}{r_2\tau-r_1}{\tau}}.
\end{equation}
Taking the limit as $\tau \to \beta$, and using the fact that $j_R(\beta)>0$, we obtain
\begin{equation}
\frac{\sf{R\r}{RA R^{-1}}{R\cdot\beta}}{\sf{\r}{A}{\beta}} = 1,
\end{equation}
proving \eqref{eq:shinconj} in the case of positive determinant.

Now we deal with the case when $\det(R) = -1$. It is easy to see that \eqref{eq:shinconj} for $R=R_1$ and $R=R_2$ implies the statement for $R=R_1R_2$, so since we have already proven the statement for $R \in \SL_2(\Z)$, it suffices to prove it for any one matrix of negative determinant; say, fix $R = \smmattwo{-1}{0}{0}{1}$. 
Consider $\tau \in \HH$. Then
\begin{equation}\label{eq:varpibar}
\ol{\varpi_\r(\ol\tau)} 
= \left(\ol{\ee{r_2\ol\tau-r_1}}, \ol{\ee{\ol\tau}}\right)_\infty
= \left(\ee{-r_2\tau+r_1},\ee{-\tau}\right)_\infty
= \varpi_{R\r}(R\cdot\tau).
\end{equation}
Applying \eqref{eq:varpibar} to the numerator and demoninator, we have
\begin{equation}
\ol{\sf{\!A}{\r}{\ol\tau}}
= \frac{\ol{\varpi_{\r}(A\cdot\ol\tau)}}{\ol{\varpi_\r(\ol\tau)}}
= \frac{\varpi_{R\r}(AR\cdot\tau)}{\varpi_{R\r}(R\cdot\tau)}
= \sf{\!RAR^{-1}}{R\r}{R\cdot\tau}.
\end{equation}
Sending $\tau \to \beta$ proves \eqref{eq:shinconj} is this (final) case, completing the proof of the lemma.
\end{proof}

Finally, we evaluate the real multiplication values of the Shintani--Faddeev modular cocycle for $\r \in \foh\Z$, showing that they are algebraic units of a simple form. These values are much simpler than those for $\r \in \Q \setminus \foh\Z$ that we will relate to Stark units. 
\begin{thm}\label{thm:trivrmval}
Let $A \in \SL_2(\Z)$ be non-parabolic, and let $\beta$ be a fixed point of $A$ with associated eigenvalue $\e$. Then, for $\r \in \Z^2$,
\begin{equation}\label{eq:trivrmval}
\sf{\r}{A}{\beta} 
= 
\begin{cases}
\psi(A,\sqrt{j_{\!A}})\sqrt{\e} & \mbox{if $r_2 > 0$}, \\
\psi(A,\sqrt{j_{\!A}})\frac{1}{\sqrt{\e}} & \mbox{if $r_2 \leq 0$}.
\end{cases}
\end{equation}
For $\r \in \foh \Z^2 \setminus \foh \Z$, we have $\sf{\r}{A}{\beta} \in \{\pm 1\}$.
Therefore, if $\r \in \foh\Z^2$, then $\sf{\r}{A}{\beta}$ is an algebraic unit in an abelian extension of $\Q(\beta)$.
\end{thm}
\begin{proof}
By \Cref{prop:invariance}, it suffices to evaluate $\sf{\r}{A}{\beta}$ for $\r \in \left\{\smcoltwo{0}{1}, \smcoltwo{0}{1/2}, \smcoltwo{1/2}{1}, \smcoltwo{1/2}{1/2}\right\}$. Taking the principal branch of the square root $\ep(\tau) = \sqrt{j_{\!A}(\tau)}$,
we obtain from \Cref{sec:half} the formula
\begin{equation}
    \sf{\smcoltwo{0}{1}}{A}{\beta} = \ee{\frac{-(A \cdot \beta) + \beta}{24}}\psi(A,\sqrt{j_{\!A}})\sqrt{j_{\!A}(\beta)} = \psi(A,\sqrt{j_{\!A}})\sqrt{j_{\!A}(\beta)}.
\end{equation}
By \Cref{lem:jeval}, $j_{\!A}(\beta)=\e$, proving \eqref{eq:trivrmval}. The fact the $\sf{\r}{A}{\beta} \in \{\pm 1\}$ for the half-integral characteristics $\r \in \left\{\smcoltwo{0}{1/2}, \smcoltwo{1/2}{1}, \smcoltwo{1/2}{1/2}\right\}$ follows from the last line of \Cref{sec:half}.

Moreover, $\e$ is a unit in $\Q(\beta)$ because it is an eigenvalue of an integral matrix of determinant $1$ with eigenvector $\smcoltwo{\beta}{1}$, and $\psi(A,\sqrt{j_{\!A}})$ is a $24$-th root of $1$. The maximal abelian extension $\Q(\beta)^{\rm ab}$ contains $\sqrt{\e}$ and the $24$-th roots of unity.
\end{proof}

Finally, we show an equivalence between stable values (including RM values) of the Jacobi and modular Shintani--Faddeev cocycles.
\begin{prop}
If $\r \in \Q$ and $\beta \in \C \setminus \Q$, then
\begin{equation}
\sigma[\sympt{\r}{\beta},\beta] = \shin^\r[\beta].
\end{equation}
\end{prop}
\begin{proof}
\Cref{prop:goose} gives following identity of meromorphic functions of $\tau \in \tDD_{\!A}$:
\begin{equation}
\sf{\r}{A}{\tau} = \sfj{(I-A)\r,A}{\sympt{\r}{\tau}}{\tau}.
\end{equation}
The claim then follows immediately by \Cref{prop:jacobistable}.
\end{proof}

\subsection{Conductor-lowering/level-raising relations}
We will now show that certain real multiplication values of the Shintani--Faddeev modular cocycle can be written as products of RM values at points of ``conductor 1.'' The ``conductor 1'' points correspond to ideal classes of maximal orders, at which the RM values will be directly related to Stark class invariants coming from $L$-functions. 

First, we define the conductor of a (real or complex) quadratic number.
\begin{defn}
Suppose $\beta \in \Cquad$ is the root of a quadratic polynomial $a\beta^2 + b\beta + c = 0$ with $a,b,c \in \Z$, $\gcd(a,b,c)=1$. Write $b^2 - 4ac = \Delta = f^2\Delta_0$ where $\Delta_0$ is a fundamental discriminant. We say that $f$ is the \textit{conductor} of $\beta$.

Equivalently, the conductor of $\beta$ is the conductor of the multiplier ring $\ord(\beta\Z+\Z)=\colonideal{\beta\Z+\Z}{\beta\Z+\Z}$ of the (pseudo-)lattice $\beta\Z+\Z$.
\end{defn}

We use the following notation, following Iwaniec \cite{iwaniec}, for the set of integral matrices of determinant $f$.
\begin{defn}
For $f \in \N$, define
\begin{equation}
G_f = \left\{\mattwo{a}{b}{c}{d} : a,b,c,d \in \Z, ad-bc = f\right\}.
\end{equation}
\end{defn}

It turns out that every quadratic number can be obtained from a quadratic number of conductor 1 by an integral fractional linear transformation.
\begin{lem}\label{lem:fto1}
If $\beta \in \Rquad$ is a real quadratic number of conductor $f$, then there exists some $B \in G_f$ and some $\alpha \in \Rquad$ of conductor $1$ such that $\beta = B \cdot \alpha$. 
\end{lem}
\begin{proof}
Let $F = \Q(\beta)$ and $\bb = \beta\Z + \Z$. Let $\OO$ be the order of conductor $f$ in $F$, and let $\OO_F$ be the maximal order. Write $\bb\OO_F = \alpha_1\Z + \alpha_2\Z$ for some $\alpha_1, \alpha_2 \in F$. Since $\bb$ is a sublattice of $\bb\OO_F$, we can write
\begin{equation}
\coltwo{\beta}{1} = \mattwo{a}{b}{c}{d}\coltwo{\alpha_1}{\alpha_2}.
\end{equation}
for some integral change-of-basis matrix $B = \smmattwo{a}{b}{c}{d}$. By general properties of lattices, the change of basis matrix $B$ has determinant $\det(B) = \pm [\bb\OO_F : \bb]$. Possibly reordering the basis $\{\alpha_1, \alpha_2\}$, we may assume $\det(B) = [\bb\OO_F : \bb]$.

By \Cref{prop:ordinvertible}, $\bb$ is an $\OO$-invertible fractional ideal, so the index of $\bb$ in $\bb\OO_F$ is
\begin{align}
[\bb\OO_F : \bb] 
&= [\bb\OO_F : \bb\OO] \\
&= [\OO_F : \OO] \mbox{ by \cite[Prop.\ A.2]{kopplagarias}} \\
&= f.
\end{align}
Thus, $\det(B) = f$; that is, $B \in G_f$. If we let $\alpha = \frac{\alpha_1}{\alpha_2}$, then
\begin{align}
B \cdot \alpha 
= \frac{a \frac{\alpha_1}{\alpha_2} + b}{c \frac{\alpha_1}{\alpha_2} + d}
= \frac{a \alpha_1 + b \alpha_2}{c \alpha_1 + d \alpha_2} = \frac{\beta}{1} = \beta. \tag*{\qedhere}
\end{align}
\end{proof}

The next two lemmas concern the basic properties of $G_f$ and its interactions with with congruence subgroups of $\SL_2(\Z)$.
\begin{lem}\label{lem:Gfconj}
Let $f, N \in \N$.
If $A \in \Gamma(fN)$ and $B \in G_f$, then $B^{-1}A B \in \Gamma(N)$.
\end{lem}
\begin{proof}
Certainly $\det(B^{-1}A B) = \det(A) = 1$. Moreover, we may write
\begin{equation}
A = I + fNA_1
\end{equation}
for some $2 \times 2$ integral matrix $A_1$, where $I$ is the $2 \times 2$ identity matrix. Also, since $\det(B) = f$, $fB^{-1}$ is an integral matrix. Thus,
\begin{equation}
B^{-1}A B = I + N(fB^{-1})A_1 B \equiv I \Mod{N}.
\end{equation}
So $B^{-1}A B \in \Gamma(N)$.
\end{proof}

\begin{lem}\label{lem:Gforbits}
The right $\SL_2(\Z)$-orbits in $G_f$ are represented by upper triangular matrices:
\begin{equation}
G_f/\SL_2(\Z) = \left\{\smmattwo{a}{b}{0}{d}\SL_2(\Z) : ad=f,\, 0 \leq b < a\right\}.
\end{equation}
\end{lem}
\begin{proof}
A matrix $\smmattwo{a}{b}{c}{d} \in G_f$ can be integrally column-reduced to a unique matrix of the desired form using the signed Euclidean algorithm on $c$ and $d$, multiplying on the right by products of the matrices $\smmattwo{1}{1}{0}{1}$ and $\smmattwo{0}{-1}{1}{0}$, and potentially multiplying by $\smmattwo{-1}{0}{0}{-1}$ in the last step. The column-reduction matrix has determinant $1$.
\end{proof}

We now examine the values of $\varpi_\r$ under the action of the orbit representatives from \Cref{lem:Gforbits}.
\begin{prop}\label{prop:utrel}
Let $\r = \smcoltwo{r_1}{r_2} \in \Q^2$ and $a,b,d \in \Z$ with $0 \leq b < a$. Then, for $\tau \in \HH$,
\begin{equation}
\varpi_{\r}\!\left(\frac{a\tau + b}{d}\right)
= \prod_{j=0}^{a-1}\prod_{\ell=0}^{d-1} \varpi_{\smmattwo{a}{b}{0}{d}^{-1}\smcoltwo{j+r_1}{\ell+r_2}}(\tau)
\end{equation}
\end{prop}
\begin{proof}
The proof is a direct calculation:
\begin{align}
\varpi_{\r}\!\left(\frac{a\tau + b}{d}\right)
&= \prod_{k=0}^\infty \left(1 - \ee{(k + r_2)\tfrac{a\tau+b}{d} - r_1}\right) \\
&= \prod_{k=0}^\infty \left(1 - \ee{\tfrac{a(k + r_2)}{d}\tau + \left(\tfrac{b(k + r_2)}{d} - r_1\right)}\right) \\
&= \prod_{\ell=0}^{d-1}\prod_{m=0}^\infty \left(1 - \ee{\tfrac{a(dm+\ell + r_2)}{d}\tau + \left(\tfrac{b(dm+\ell + r_2)}{d} - r_1\right)}\right) \\
&= \prod_{\ell=0}^{d-1}\prod_{m=0}^\infty \left(1 - \ee{a\left(m+\tfrac{\ell + r_2}{d}\right)\tau + \left(\tfrac{b(\ell + r_2)}{d} - r_1\right)}\right) \\
&= \prod_{\ell=0}^{d-1}\prod_{m=0}^\infty \prod_{j=0}^{a-1} \left(1 - \ee{\left(m+\tfrac{\ell + r_2}{d}\right)\tau + \tfrac{1}{a}\left(\tfrac{b(\ell + r_2)}{d} - j - r_1\right)}\right) \\
&= \prod_{j=0}^{a-1}\prod_{\ell=0}^{d-1}\prod_{m=0}^\infty \left(1 - \ee{\left(m+\tfrac{\ell + r_2}{d}\right)\tau - \tfrac{d(j + r_1) - b(\ell + r_2)}{ad}}\right) \\
&= \prod_{j=0}^{a-1}\prod_{\ell=0}^{d-1} \varpi_{\frac{d(j + r_1) - b(\ell + r_2)}{ad},\frac{\ell + r_2}{d}}(\tau) \\
&= \prod_{j=0}^{a-1}\prod_{\ell=0}^{d-1} \varpi_{\smmattwo{a}{b}{0}{d}^{-1}\smcoltwo{j+r_1}{\ell+r_2}}(\tau),
\end{align}
using $\smmattwo{a}{b}{0}{d}^{-1} = \frac{1}{ad}\smmattwo{d}{-b}{0}{a}$ in the last step.
\end{proof}

We now give a ``conductor-lowering/level-raising'' relation for the RM values of the Shintani--Faddeev modular cocycle.
\begin{thm}\label{thm:cllr}
Let $\r \in \Q/\Z$, $f \in \Z$, and $B \in G_f$. Let $A \in \ds\bigcap_{\substack{\s \in \Q^2/\Z^2 \\ B\s - \r \in \Z^2}} \Gamma_\s$. (In particular, this holds if $\r \in \frac{1}{N}\Z/\Z$ and $A \in \Gamma(fN)$.) Let $\alpha$ be a fixed point $A$. Then,
\begin{equation}
\sf{\r}{B A B^{-1}}{B \cdot \alpha} = \prod_{\substack{\s \in \Q^2/\Z^2 \\ B\s - \r \in \Z^2}} \sf{\s}{A}{\alpha}
\end{equation}
\end{thm}
\begin{proof}
By \Cref{lem:Gforbits}, we can represent $B$ as $B = \smmattwo{a}{b}{0}{d} C$ for some $a,b,d \in \Z$, $ad=f$, $0 \leq b < a$, and some $C \in \SL_2(\Z)$. Thus, by \Cref{prop:utrel}, 
\begin{equation}
\varpi_{\r}(B \cdot \tau)
= \varpi_{\r}\!\left(\frac{a(C \cdot \tau) + b}{d}\right)
= \prod_{j_1=0}^{a-1}\prod_{j_2=0}^{d-1} \varpi_{\smmattwo{a}{b}{0}{d}^{-1}(\mathbf{j}+\r)}(C \cdot \tau)
\end{equation}
for all $\tau \in \HH$, where $\mathbf{j} = \smcoltwo{j_1}{j_2}$.
Thus,
\begin{align}
\sf{\r}{B A B^{-1}}{B \cdot \tau} 
&= \frac{\varpi_{\r}(B \cdot (A \cdot \tau))}{\varpi_{\r}(B \cdot \tau)} \\
&= \prod_{j_1=0}^{a-1}\prod_{j_2=0}^{d-1} 
\frac{\varpi_{\smmattwo{a}{b}{0}{d}^{-1}(\mathbf{j}+\r)}(CA \cdot \tau)}{\varpi_{\smmattwo{a}{b}{0}{d}^{-1}(\mathbf{j}+\r)}(C \cdot \tau)} \\
&= \prod_{j_1=0}^{a-1}\prod_{j_2=0}^{d-1} 
\sf{\smmattwo{a}{b}{0}{d}^{-1}(\mathbf{j}+\r)}{CAC^{-1}}{C\cdot \tau}.
\end{align}
Sending $\tau \to \alpha$ and using \Cref{thm:shinconj} with $R=C^{-1}$, we obtain
\begin{equation}
\sf{\r}{BAB^{-1}}{B \cdot \alpha} 
= \prod_{j_1=0}^{a-1}\prod_{j_2=0}^{d-1} \sf{C^{-1}\smmattwo{a}{b}{0}{d}^{-1}(\mathbf{j}+\r)}{A}{\alpha}
= \prod_{j_1=0}^{a-1}\prod_{j_2=0}^{d-1} \sf{B^{-1}(\mathbf{j}+\r)}{A}{\alpha}.
\end{equation}
By \Cref{prop:invariance}, $\sf{B^{-1}(\mathbf{j}+\r)}{A}{\alpha}$ is periodic of period $a$ in $j_1$ and period $d$ in $j_2$. Moreover, the set of column vectors $\s \in \Q^2/\Z^2$ of the form $\s = B^{-1}(\mathbf{j}+\r)$ with $\mathbf{j} \in \Z^2$ are precisely those such that $B\s - \r \in \Z^2$. Thus,
\begin{equation}
\sf{\r}{BAB^{-1}}{B \cdot \alpha}  = \prod_{\substack{\s \in \Q^2/\Z^2 \\ B\s - \r \in \Z^2}} \sf{\s}{A}{\alpha}. \tag*{\qedhere}
\end{equation}
\end{proof}

\section{Cohomological interpretations of the Shintani--Faddeev cocycles}\label{sec:cohomology}

We will now offer some more formal cohomological perspectives on the groups of ``cocycles'' and ``cohomology classes'' introduced in the previous section. In \Cref{sec:wannabe}, we described the tuple $\shin^\r = (\shin^\r_A)_{A \in \Gamma_\r}$ of meromorphic functions as an element of a group $Z^1_\DD(\Gamma_\r,\MM_\C^\times)$ of ``$1$-cocycles,'' defining a coset class $[\shin^\r]$ in a quotient group $H^1_\DD(\Gamma_\r,\MM_\C^\times)$. A key property of this ``cohomology class'' is that the \textit{stable values} (including the \textit{real multiplication values}) are (essentially) independent of the choice of class representative; see \Cref{prop:stablewd}. We will now present several ways of describing the group $H^1_\DD(\Gamma_\r,\MM_\C^\times)$ (and variants thereof) in a matter resembling existing cohomology theories. Similar constructions can also be used for the groups of Jacobi cocycles, but we omit the details in that case.

Our cocycles are group-cohomological cocycles, with the essentially modification that they are valued in abelian groups $\MM^\times\!(\UU_{\!A})$ that vary based on the element $A \in \Gamma_\r$. 
In \Cref{sec:gengroupcohom}, we provide a simple generalization of the definition of the first cohomology of a group to accomodate this modification. In \Cref{sec:equivariant}, we provide a more complicated construction of a sequence of cohomology groups, somewhat akin to equivariant sheaf cohomology; we connect this construction to $H^1_\UU(\Gamma_\r,\MM_\C^\times)$ and to a ``parabolic'' variant in \Cref{sec:connecttoworking}.

These constructions are meant to provide a launching point for any future efforts to describe the structure of $H^1_\UU(\Gamma_\r,\MM_\C^\times)$ (or a related group) using tools such as spectral sequences to study maps between it and better-understood cohomology groups, such as Eichler cohomology. They are not intended to be a complete exposition of a new cohomology theory, nor are they meant to be a final authoritative answer to the question of what group the Shintani--Faddeev modular cocycle ``should'' live in.

\subsection{Generalized first group cohomology}\label{sec:gengroupcohom}

We first describe a simple modification of the first group cohomology that encompass the groups described in \Cref{sec:working}.

Let $\Gamma$ be any group and $M$ an abelian group (written multiplicatively) with a compatible right action of $\Gamma$ (written as exponentiation). In other words, $M$ is a right $\Z\Gamma$-module. The theory of group cohomology associates a series of comomology groups $H^n(\Gamma,M)$ to the pair $(\Gamma,M)$; see, for example, \cite[Ch.\ VII]{serre} for details. 
The first cohomology group is a quotient of two subgroups of the group of functions from $\Gamma \to M$; we denote evaluation of a function $m : \Gamma \to M$ at $m \in M$ by $m_g$.
\begin{equation}
H^1(\Gamma,M) = \frac{\left\{m : \Gamma \to M \mid m_{g_1g_2} = m_{g_1}^{g_2} m_{g_2}\right\}}{\left\{m : \Gamma \to M \mid m_g = c^g c^{-1} \mbox{ for some } c \in M\right\}}.
\end{equation}
Now, suppose we have a function $N : \Gamma \to \{\mbox{subgroups of } M\}$ denoted by $N_g$. We may define a generalized first cohomology group
\begin{equation}
\tH_N^1(\Gamma,M) = \frac{\left\{m : \Gamma \to M \mid m_g \in N_g, m_{g_1g_2} = m_{g_1}^{g_2} m_{g_2}\right\}}{\left\{m : \Gamma \to M \mid m_g = c^g c^{-1} \mbox{ for some } c \in \bigcap_{g \in \Gamma} N_g\right\}}.
\end{equation}

Now, as in \Cref{sec:working}, let $\FF$ be a sheaf of multiplicative groups of $\C$-valued functions on a topological space $X$ with a continuous action $\Gamma \times X \to X$. Let $\UU$ be a system of domains in the sense of \Cref{defn:domainscohom}. If each $\UU_A \supseteq \HH$ and the restriction maps $\FF(\UU_A) \to \FF(\HH)$ are injective for all $A \in \Gamma$, then it follows immediately from definitions that, for $N_A := \FF(\UU_A)$,
\begin{align}
\tilde{H}_N^1(\Gamma,M) \isom H_\UU^1(\Gamma,\FF).
\end{align}
In particular, this isomorphism holds when $\Gamma = \Gamma_\r$, $\UU_{\!A} = \DD_{\!A}$,
and either $\FF = \MM^\times$ or $F = \AA^\times$.

\subsection{Equivariant $(Q,V)$-cohomology of $\Gamma$-sheaves}\label{sec:equivariant}

We now describe a more sophisticated approach that that produces higher cohomology groups $H^n_{Q,V}(\Gamma,\FF,X^\circ)$ associated to a sheaf on a space with a group action, along with other data.

Let $X$ be a topological space with a continuous action of a group $\Gamma$.
Let $Q$ be a $\Gamma$-set, that is, a set with a $\Gamma$-action.
Fix open sets $V_{\q} \subseteq X$ for each $n \in \N$ and $\q \in Q^n$. Suppose that, for $A \in \Gamma$ and $\q \in Q^n$,
\begin{equation}
A \cdot V_\q = V_{\!A\cdot \q}
\end{equation}
where the action is diagonal, and also suppose that 
\begin{equation}\label{eq:vinclusion}
V_\q \subseteq \bigcap_{i=0}^{n} V_{\hat{\q}_i}
\end{equation}
for $\q = (q_0, \ldots, q_n)$
and $\hat{\q}_i = (q_0,\ldots,q_{i-1},q_{i+1},\ldots,q_n)$.

We need the following definition of a $\Gamma$-sheaf, which behaves like a sheaf of $\Z\Gamma$-modules but is more general.
\begin{defn}\label{defn:gammasheaf}
A \textit{$\Gamma$-sheaf} $\FF$ on $X$ is a sheaf of abelian groups on $X$ along with a ``$\Gamma$-action'' defined by abelian group maps $f \mapsto f^A$ from $\FF(\UU) \to \FF(A^{-1}\cdot \UU)$ commuting with restriction maps (so that $f^A|_V = (f|_V)^A$ for all open sets $V \subseteq U$) and satisfying the compatibility relations $f^I=f$ and $(f^B)^C = f^{BC}$. A map of $\Gamma$-sheaves is a map of sheaves of abelian groups that commutes with the $\Gamma$-action.
\end{defn}

We can define a ``fixed sheaf functor'' as follows; this functor is left exact.
\begin{defn}\label{defn:fixedsheaf}
For any $\Gamma$-sheaf $\FF$ on $X$, we define the \textit{fixed sheaf} $\FF^\Gamma$ of abelian groups by
\begin{align}
\FF^\Gamma(\UU) = \{f \in \FF(\UU) : f|_{\UU \cap A^{-1}\cdot \UU} = f^A|_{\UU \cap A^{-1}\cdot \UU} \mbox{ for all } A \in \Gamma\}.
\end{align}
Additionally, if $\varphi: \FF \to \FF'$ is a map of sheaves, then for $f \in \FF^\Gamma(\UU)$, we define
\begin{align}
\varphi_\UU^\Gamma(f) = \varphi_\UU(f).
\end{align}
Together, these data define a functor from the category of $\Gamma$-sheaves on $X$ to the category of sheaves of abelian groups on $X$.
\end{defn}
\begin{prop}
The fixed sheaf functor is left exact.
\end{prop}
\begin{proof}
Consider an exact sequence of $\Gamma$-sheaves
\begin{equation}
1 \to \FF_1 \xrightarrow{\varphi} \FF_2 \xrightarrow{\psi} \FF_3.
\end{equation}
We want to show that the sequence of sheaves of abelian groups
\begin{equation}
1 \to \FF_1^\Gamma \xrightarrow{\varphi^\Gamma} \FF_2^\Gamma \xrightarrow{\psi^\Gamma} \FF_3^\Gamma.
\end{equation}
is exact. To do so, we must show that $\ker\varphi^\Gamma = 1$ and $\image\varphi^\Gamma = \ker\psi^\Gamma$.

The first is easy: For each open set $\UU \subseteq X$,
\begin{equation}
(\ker\varphi^\Gamma)(\UU) = (\ker\varphi_\UU) \cap \FF_2^\Gamma(\UU) = 1
\end{equation}
because $\ker\varphi_\UU = 1$; thus, $\ker\varphi^\Gamma = 1$.

It remains to show that $\image\varphi^\Gamma = \ker\psi^\Gamma$; it's clear that $(\image\varphi^\Gamma)(\UU) \subseteq (\ker\psi^\Gamma)(\UU)$ for each open set $\UU \subseteq X$, so we must show the reverse inclusion. For each $\UU$,
\begin{equation}
(\ker\psi^\Gamma)(\UU) = (\ker\psi_\UU) \cap \FF_2^\Gamma(\UU) \subseteq \ker\psi_\UU = \image\varphi_\UU.
\end{equation}
Suppose that $f_2 \in (\ker\psi^\Gamma)(\UU)$, so in particular, $f_2 \in \image\varphi_\UU$. Write $f_2 = \varphi_\UU(f_1)$ for some $f_1 \in \FF_1$. Thus, $\varphi_\UU(f_2) \in \FF_2^\Gamma(\UU)$, so $\varphi_\UU(f_1^A) = \varphi_\UU(f_1)^A = \varphi_\UU(f_1)$. The function $\varphi_\UU$ is injective, so $f_1^A = f_1$. Thus, $f_2 \in \image\varphi_\UU^\Gamma = (\image\varphi^\Gamma)(\UU)$. We've now shown that $(\image\varphi^\Gamma)(\UU) = (\ker\psi^\Gamma)(\UU)$ for each $\UU$, so $\image\varphi^\Gamma = \ker\psi^\Gamma$.
\end{proof}

Let $\FF$ be a $\Gamma$-sheaf on $X$ with the group operation written multiplicatively.
Define the $\Gamma$-sheaf $\CC^n$ with underlying sheaf of abelian groups
\begin{equation}
\CC^n = \CC^n_{Q,V,\FF} = \prod_{\q \in Q^{n+1}} \FF_{V_\q}
\end{equation}
and group action defined for $f \in \CC^n(\UU)$ by
\begin{equation}
\left((f_\q)_{\q \in Q^{n+1}}\right)^{A} = (f_{\!A\cdot\q}^A)_{\q \in Q^{n+1}}.
\end{equation}
To check that $\CC^n$ is a $\Gamma$-sheaf, observe that
\begin{equation}
(f^A)_\q = f_{A\cdot\q}^A \in \FF(A^{-1}(V_{A\cdot\q} \cap \UU) = \FF(A^{-1}\cdot V_{A\cdot\q} \cap A^{-1}\cdot \UU) = \FF(V_{\q} \cap A^{-1}\cdot \UU).
\end{equation}
Now consider the complex of sheaves
\begin{equation}
1 
\to 
\CC^0
\xrightarrow{\partial_0}
\CC^1
\xrightarrow{\partial_1}
\CC^2
\xrightarrow{\partial_2} 
\cdots,
\end{equation}
where the boundary maps are defined (on each $\CC^n(\UU)$) by
\begin{equation}
(\partial_n f)_{q_0,\ldots,q_{n+1}} = 
\prod_{j=0}^n f_{q_0,\ldots,q_{j-1},q_{j+1},\ldots,q_{n+1}}^{(-1)^j}.
\end{equation}
We may apply the fixed sheaf functor to obtain a complex of sheaves of abelian groups
\begin{equation}
1 
\to 
\left(\CC^0\right)^{\!\Gamma}
\xrightarrow{\partial_0^\Gamma}
\left(\CC^1\right)^{\!\Gamma}
\xrightarrow{\partial_1^\Gamma}
\left(\CC^2\right)^{\!\Gamma}
\xrightarrow{\partial_2^\Gamma}
\cdots.
\end{equation}
Finally, we choose a particular open set $X^\circ \subseteq X$ and apply the left exact functor $\CC \mapsto \CC(X^\circ)$, $\varphi \mapsto \varphi_{X^\circ}$ from the category of sheaves of abelian groups to the category of abelian groups. Taking the induced maps to be $d_n = (\partial_n^\Gamma)_{X^\circ}$, we obtain a complex of abelian groups
\begin{equation}\label{eq:abgpcomplex}
1 
\to 
\left(\CC^0\right)^{\!\Gamma}\!\!(X^\circ)
\xrightarrow{d_0}
\left(\CC^1\right)^{\!\Gamma}\!\!(X^\circ)
\xrightarrow{d_1}
\left(\CC^2\right)^{\!\Gamma}\!\!(X^\circ)
\xrightarrow{d_2}
\cdots.
\end{equation}
\begin{defn}\label{defn:QV}
The \textit{equivariant $(Q,V)$-cohomology of $\FF$ on $X^\circ$} is defined to be the cohomology of \eqref{eq:abgpcomplex}, that is,
\begin{equation}
H_{Q,V}^n(\Gamma,\FF,X^\circ) = \frac{\ker(d_n)}{\image(d_{n-1})}.
\end{equation}
\end{defn}

We will now compute the zeroth and first cohomology groups more explicitly. While not needed for the definitions above, we will now add the assumption that the $V_q$ cover $X^\circ$.
\begin{prop}
Suppose that $\bigcup_{q \in Q} V_q \supseteq X^\circ$.
Then,
\begin{equation}
H_{Q,V}^0(\Gamma,\FF,X^\circ) \isom \FF^\Gamma\!(X^\circ).
\end{equation}
\end{prop}
\begin{proof}
By \Cref{defn:QV}, the zeroth $(Q,V)$-cohomology group is
\begin{equation}
H_{Q,V}^0(\Gamma,\FF,X^\circ) = \frac{\ker(d_0)}{\image(d_{-1})} = \ker(d_0),
\end{equation}
where the $d_n$ are as in \eqref{eq:abgpcomplex}.
The condition that $f \in \left(\CC^0\right)^{\!\Gamma}\!\!(X^\circ)$ is equivalent to the condition that $f \in \prod_{q\in Q}\FF(X^\circ \cap V_q)$ and $f_q = (f^A)_q = f_{A\cdot q}^A$.
The condition that $d_0(f)=0$ means that $f_{q_0}=f_{q_1}$ for every pair $(q_0,q_1) \in Q^2$.
Therefore, fixing any choice of $q_0 \in Q$,
\begin{equation}
\ker(d_0) = \left\{f \in \prod_{q\in Q}\FF(X^\circ \cap V_q) : f_q = f_{q_0} = f_{q_0}^A\right\}.
\end{equation}
By gluing, 
the diagonal map defines an isomorphism $\FF^\Gamma\!\left(\bigcup_{q\in Q} (X^\circ \cap V_q)\right) \isom \ker(d_0)$; moreover, by hypothesis, $\FF^\Gamma\!\left(\bigcup_{q\in Q} (X^\circ \cap V_q)\right)=\FF^\Gamma(X^\circ)$.
\end{proof}
For the computation of the first cohomology, we will also assume that $\Gamma$ acts transitively on $Q$. We give a presentation of the first cohomology that generalizes the usual presentation of standard group cohomology.
\begin{prop}\label{prop:HQV1calc}
Suppose that $\bigcup_{q \in Q} V_q \supseteq X^\circ$ and that $\Gamma$ acts transitively on $Q$. Fix any $q_0 \in Q$, let $\Gamma_0 = \stab_\Gamma(q_0)$, and choose for each $q \in Q$ some $A_q \in \Gamma$ with $A_q \cdot q = q_0$.
Then,
\begin{equation}
H_{Q,V}^1(\Gamma,\FF,X^\circ) \isom
\frac{Z_{Q,V}^1(\Gamma,\FF,X^\circ)}{B_{Q,V}^1(\Gamma,\FF,X^\circ)} \label{eq:HQVasZQVoverBQV}
\end{equation}
where
\begin{align}
Z_{Q,V}^1(\Gamma,\FF,X^\circ) 
&= \left\{g \in \prod_{q \in Q} \FF(V_{q,q_0}) : g_{A_1A_2\cdot q}=g_{A_2\cdot q}^{A_1^{-1}}g_{A_1\cdot q_0}\right\};\label{eq:zqvdef} \\
B_{Q,V}^1(\Gamma,\FF,X^\circ) 
&= \left\{g \in \prod_{q \in Q} \FF(V_{q,q_0}) : 
\begin{array}{ll}
\exists \ h \in \FF^{\Gamma_0}(V_{q_0}) \\
\text{such that } g_{q}=h^{A_q} h^{-1}
\end{array}
\right\}.\label{eq:bqvdef}
\end{align}
These groups can also be written as
\begin{align}
Z_{Q,V}^1(\Gamma,\FF,X^\circ) 
&\isom \left\{w \in \prod_{A \in \Gamma} \FF(V_{A^{-1}\cdot q_0,q_0}) : 
\begin{array}{ll}
w_{A_1A_2}=w_{A_1}^{A_2}w_{A_2}; \\
w_A=1 \mbox{ for } A \in \Gamma_0
\end{array}
\right\};\label{eq:zqvdefA} \\
B_{Q,V}^1(\Gamma,\FF,X^\circ) 
&\isom \left\{w \in \prod_{A \in \Gamma} \FF(V_{A^{-1}\cdot q_0,q_0}) : 
\begin{array}{ll}
\exists \ h \in \FF^{\Gamma_0}(V_{q_0}) \\
\text{such that } w_{A}=h^{A} h^{-1}
\end{array}
\right\}.\label{eq:bqvdefA}
\end{align}
\end{prop}
\begin{proof}
By \Cref{defn:QV}, the first $(Q,V)$-cohomology group is
\begin{equation}\label{eq:HQVaskeroverim}
H_{Q,V}^1(\Gamma,\FF,X^\circ) = \frac{\ker(d_1)}{\image(d_0)},
\end{equation}
where the $d_n$ are as in \eqref{eq:abgpcomplex}.
The condition that $f \in \left(\CC^1\right)^{\!\Gamma}\!\!(X^\circ)$ is equivalent to the condition that $f \in \prod_{(q_1,q_2)\in Q^2}\FF(X^\circ \cap V_{q_1,q_2})$ and $f_{q_1,q_2} = (f^A)_{q_1,q_2} = f_{A\cdot q_1, A\cdot q_2}^A$.
The condition that $d_1(f)=0$ means that $f_{q_1,q_3}=f_{q_1,q_2}f_{q_2,q_3}$ for every pair $(q_1,q_2,q_3) \in Q^3$.
If $g_q = f_{q,q_0}$ for $f \in \ker(d_1)$, then
\begin{equation}
g_{A_1A_2\cdot q} 
= f_{A_1A_2\cdot q,q_0}
= f_{A_1A_2\cdot q,A_1\cdot q_0} f_{A_1\cdot q_0,q_0}
= f_{A_2\cdot q,q_0}^{A_1^{-1}} f_{A_1\cdot q_0,q_0}
= g_{A_2\cdot q}^{A_1^{-1}} g_{A_1\cdot q_0}.
\end{equation}
Thus, with $Z_{Q,V}^1(\Gamma,\FF,X^\circ)$ as defined in \eqref{eq:zqvdef}, there is a homomorphism
\begin{align}
\ker(d_1) &\xrightarrow{\varphi} Z_{Q,V}^1(\Gamma,\FF,X^\circ) \\
f &\mapsto (f_{q,q_0})_{q\in Q},
\end{align}
and the hypothesis that $\Gamma$ acts transitively on $Q$ ensures that $\varphi$ is an isomorphism.

Moreover, $f \in \image(d_0)$ if and only if $f_{q_1,q_2} = h_{q_1}h_{q_2}^{-1}$ for some $h \in \left(\CC^0\right)^{\!\Gamma}\!\!(X^\circ)$. By transitivity of the group action, this is equivalent to the condition that $f_{q,q_0} = h_{q_0}^A h_{q_0}^{-1}$ for every $A \in \Gamma$ such that $A \cdot q = q_0$. In turn, setting $h = h_{q_0}$, this is equivalent to the condition that $f_{q,q_0} = h^{A_q}h^{-1}$ and $h \in \FF^{\Gamma_0}(V_{q_0})$.
Thus, $\varphi(\image(d_0)) = B_{Q,V}^1(\Gamma,\FF,X^\circ)$ as defined in \eqref{eq:bqvdef}.
Therefore, \eqref{eq:HQVasZQVoverBQV} follows from \eqref{eq:HQVaskeroverim}.

Finally, the map sending $g \mapsto w$ where $w_A = g_{A^{-1} \cdot q_0}$ defines the isomorphisms \eqref{eq:zqvdefA} and \eqref{eq:bqvdefA}, having inverse maps defined by $g_q = w_{A_q}$.
\end{proof}

\subsection{From the working definition to equivariant $(Q,V)$-cohomology}\label{sec:connecttoworking}

As in both \Cref{sec:working} and \Cref{sec:equivariant}, let $X$ is a topological space with a continuous action of a group $\Gamma$. As in \Cref{sec:working}, let $\FF$ be a sheaf of multiplicative groups of $\C$-values functions on $X$. For $A \in \Gamma$ and $f \in \FF(\UU)$, write $f^A \in \FF(A^{-1} \cdot \UU)$ for the function defined by $f^A(u) = f(A\cdot u)$. We see by inspection that $\FF$ is a $\Gamma$-sheaf in the sense of \Cref{defn:gammasheaf}.

The cohomology group $H_\UU^1(\Gamma,\FF)$ defined in \Cref{sec:working} is identified with an equivariant $(\Gamma,V)$-cohomology group by the following proposition.

\begin{prop}
Let $(\UU_A)_{A \in \Gamma}$ be a system of domains in the sense of \Cref{defn:domainscohom}, covering an open subset $X^\circ$ of $X$, and suppose $\UU_I = X^\circ$. For any $A,B \in \Gamma$, set $V_{A} = A \cdot X^\circ$ and $V_{A,B} = B \cdot \UU_{A^{-1}B}$. For $n \geq 2$, define $V_{\mathbf{A}} = \bigcap_{i=0}^{n-1} V_{A_i,A_{i+1}}$ for $\mathbf{A} = (A_0, \ldots, A_n) \in \Gamma^{n+1}$. Then,
\begin{equation}
H^1_\UU(\Gamma,\FF) \isom H^1_{\Gamma,V}(\Gamma,\FF,X^\circ).
\end{equation}
\end{prop}
\begin{proof}
We check directly that $A\cdot V_{B} = AB \cdot X^\circ = V_{AB}$, $A\cdot V_{B,C} = AC\cdot \UU_{B^{-1} C} = AC\cdot \UU_{(AB)^{-1} (AC)} = V_{AB,AC}$, and thus $A \cdot V_{\mathbf{B}} = V_{A\mathbf{B}}$ for $\mathbf{B} \in \Gamma^{n+1}$; the inclusion \eqref{eq:vinclusion} is also verified directly.

We now apply \Cref{prop:HQV1calc} with $Q=\Gamma$ and $q_0=I$. We have $V_{A^{-1}\cdot q_0, q_0} = V_{A^{-1},I} = \UU_A$ and $\Gamma_0 = \stab_\Gamma(I) = \{I\}$.
Then \eqref{eq:zqvdefA} and \eqref{eq:bqvdefA} define the same groups of cocycles and coboundaries as defined in \Cref{defn:domainscohom}:
\begin{align}
Z^1_\UU(\Gamma,\FF) \isom Z^1_{\Gamma,V}(\Gamma,\FF,X^\circ); \\
B^1_\UU(\Gamma,\FF) \isom B^1_{\Gamma,V}(\Gamma,\FF,X^\circ).
\end{align}
The proposition follows.
\end{proof}

Writing $\r = \smcoltwo{p_1/q_1}{p_2/q_2}$, the Shintani--Faddeev modular cocycle satisfies an triviality condition with respect to the matrix $T^{q_2}=\smmattwo{1}{q_2}{0}{1}$; specifically, $\sf{\r}{T^{q_2}}{\tau)}=1$. 
Suppose that $\r \nin \frac{1}{2}\Z^2$, so that $-I \nin \Gamma_\r$.
Under the fractional linear transformation action of $\Gamma_\r$, let
\begin{equation}
\Gamma_{\r,\infty} = \stab_\Gamma(\infty) = \langle T^{q_2} \rangle.
\end{equation}
The triviality condition can be expressed by stating that $\shin^\r$ is a member of the group 
\begin{align}
Z^1_{\DD,\text{par-}\infty}(\Gamma,\FF) 
&:= \left\{w \in \prod_{A \in \Gamma_\r} \FF(\DD_A) : 
\begin{array}{ll}
w_{A_1A_2} = w_{A_1}^{A_2}w_{A_2}; \\
w_{A} = 1 \mbox{ for } A \in \Gamma_{\r,\infty}
\end{array}
\right\} \\
&= \left\{w \in \prod_{A \in \Gamma_\r} \FF(\DD_A) : w_{A_1A_2} = w_{A_1}^{A_2}w_{A_2}; w_{T^{q_2}} = 1\right\}
\end{align}
with $\FF = \MM^\times$.
We may additionally define
\begin{align}
B^1_{\DD,\text{par-}\infty}(\Gamma,\FF) &= 
\left\{w \in \prod_{A \in \Gamma_\r} \FF(\DD_A) : 
\begin{array}{ll}
\exists \ h \in \FF^{\Gamma_{\r,\infty}}(\C) \\
\mbox{such that } w_A = h^Ah^{-1}
\end{array}
\right\}\\
H^1_{\DD,\text{par-}\infty}(\Gamma,\FF) &=
\frac{Z^1_{\DD,\text{par-}\infty}(\Gamma,\FF)}{B^1_{\DD,\text{par-}\infty}(\Gamma,\FF)}. \label{eq:h1par}
\end{align}
Note that this ``parabolic at $\infty$'' cohomology group is not the same as parabolic cohomology imposing vanishing conditions at every cusp.

For $q,q' \in \Gamma_\r \cdot \infty$ with $q \neq q'$, define
\begin{equation}
\sgn_\r(q,q') = 
\begin{cases}
+1 & \mbox{if } (\exists \smmattwo{a}{b}{c}{d} \in \Gamma_\r) \ (-\tfrac{d}{c}, \infty) \in \Gamma_\r\cdot(q,q') \mbox{ and } c\geq 0 \\ 
-1 & \mbox{if } (\exists \smmattwo{a}{b}{c}{d} \in \Gamma_\r) \ (-\tfrac{d}{c}, \infty) \in \Gamma_\r\cdot(q,q') \mbox{ and } c<0
\end{cases}
\end{equation}
This sign is well-defined because $-I \nin \Gamma_\r$, so in particular $-I \nin \Gamma_{\r,\infty}$.
Comparing \eqref{eq:h1par} with \Cref{prop:HQV1calc} and especially \eqref{eq:zqvdefA} and \eqref{eq:bqvdefA}, we immediately obtain the following proposition, identifying the ``parabolic at $\infty$'' cohomology group with an equivariant $(Q,V)$-cohomology group for $Q = \Gamma_\r \cdot \infty$.
\begin{prop}
Let $\r \in \Q^2 \setminus \foh\Z^2$.
Let $Q = \Gamma_\r \cdot \infty$ (which is a subset of $\Q \cup \{\infty\} \isom \Pj^1(\Q)$).
For $q, q' \in Q$, set $V_{q} = (\C \cup \{\infty\}) \setminus \{q\}$ and 
\begin{equation}
V_{q,q'} = 
\begin{cases}
(\C \cup \{\infty\}) \setminus [q',q] & \mbox{if } q \neq q' \mbox{ and } \sgn_\r(q,q') = +1, \\
(\C \cup \{\infty\}) \setminus [q,q'] & \mbox{if } q \neq q' \mbox{ and } \sgn_\r(q,q') = -1, \\
V_{q} & \mbox{if } q=q'.
\end{cases}
\end{equation}
(Here, for $x, y \in \R$, the notation $[x,y]$ is the usual closed interval from $x$ to $y$ if $x \leq y$, and $[x,y] := [x,\infty)\cup\{\infty\}\cup(-\infty,y]$ if $x > y$. We define $[\infty,x] := \{\infty\}\cup(-\infty,x]$ and $[x,\infty] := [x,\infty)\cup\{\infty\}$.)
For $n \geq 2$, define $V_{\q} = \bigcap_{i=0}^{n-1} V_{q_i,q_{i+1}}$ for $\q = (q_0, \ldots, q_n) \in \Gamma_\r^{n+1}$. Then,
\begin{equation}
H^1_{\DD,{\rm par\text{-}}\infty}(\Gamma_\r,\FF) \isom H^1_{Q,V}(\Gamma_\r,\FF,\C).
\end{equation}
\end{prop}

The definition of equivariant $(Q,V)$-cohomology should be more broadly applicable beyond the propositions stated and proven in this section. In particular, the Shintani--Faddeev Jacobi cocycle will be identified with an equivariant $(Q,V)$-cohomology class for the Jacobi group $\Z^2 \semidirect \SL_2(\Z)$. More generally, one hopes that the definition will apply to $(r_1+r_2-1)$-cocycles for $\Z^n \semidirect \SL_n(\Z)$ encoding Kronecker limit formulas for more general number fields $F$ with $[F : \Q] = n$ and $F \tensor \R \isom \R^{r_1} \times \C^{r_2}$. This seems particularly feasible in the totally real case ($r_2=0$, multiple sine functions) and the almost totally real case ($r_2=1$, multiple elliptic gamma functions, including the complex cubic case \cite{bcg}).

It would also be interesting (in future work) to use this framework to consider \textit{additive} cocycles for the weight $k$ action $(f|_k{A})(\tau) := j_{\!A}(\tau)^k f(A\cdot\tau)$ of finite-index subgroups of $\SL_2(\Z)$, such as cocycles of holomorphic quantum modular forms, generalizing Eichler cohomology. 
More generally, one could consider additive cocycles for the ``weight $w$'' action $(f|_w{A})(\tau) := w_{\!A}(\tau) f(A\cdot\tau)$ for any multiplicative cocycle $w_{\!A}$, including $w_{\!A} = \shin_{\!A}^{\r}$.

\section{Partial zeta functions}\label{sec:partial}

In this section, we introduce two partial zeta functions, the \textit{ray class partial zeta function} and the \textit{ideal coset partial zeta function}, and we prove that they are closely related. We also relate the ray class partial zeta function to the \textit{Galois-theoretic partial zeta function} appearing in Tate's refinement of the Stark conjectures. We also discuss statements of the Stark conjectures in terms of \textit{differenced ray class partial zeta functions}.

\subsection{Ray class partial zeta functions}

Let $F$ be a number field and $\OO$ be an order in $F$. Let $\mm$ be an ideal of $\OO$ and $\rS$ a subset of the set of real embeddings of $F$.

To a ray class in the \rcmia, we associate a zeta function. (The dependence of $\OO$ is implicit in this definition, via the dependence on $\A$.)
\begin{defn}[Ray class partial zeta function]
Let $\A \in \Clt_{\mm,\rS}(\OO)$. For $\re(s)>1$, define 
\begin{equation}\label{eq:rayseries}
\zeta_{\mm,\rS}(s,\A) = \sum_{\substack{\aa \in \A \\ \aa \subseteq \OO}} \Nm(\aa)^{-s}.
\end{equation}
\end{defn}

In the case of the maximal order, the ray class partial zeta function can be reduced to the ray class group case.
\begin{prop}\label{prop:changemodzeta}
Let $\A \in \Clt_{\mm,\rS}(\OO_F)$. Choose an ideal $\cc \in \rJ_\mm^\ast(\OO_F)$ such that $\cc\A = [(\gamma_0)]$ for some $\gamma_0 \in \OO_F$, and let $\dd = \mm + \gamma_0\OO_F$. Then
\begin{equation}
\zeta_{\mm,\rS}(s,\A) = \Nm(\dd)^{-s} \zeta_{\mm',\rS}(s,\A') \label{eq:changemodzeta}
\end{equation}
where $\mm' = \dd^{-1}\mm$ and $\A' = [\gamma\cc^{-1}\dd^{-1}] \in \Cl_{\mm,\rS}(\OO_F)$.
\end{prop}
\begin{proof}
By definition of $\Clt_{\mm,\rS}(\OO_F)$, for an integral ideal $\mm$, we have that
\begin{align}
\aa \in \A 
&\iff \cc\aa = \gamma\OO_F \mbox{ with } \gamma \con \gamma_0 \Mod{\mm} \mbox{ and } \rho(\gamma)>0 \mbox{ for } \rho \in \rS
\\
&\implies \cc\aa = \gamma\OO_F \mbox{ with } \gamma \in \dd 
\\ 
&\implies \dd | \cc\aa
\\
&\implies \dd | \aa \mbox{ because $\cc$ is coprime to $\mm$ and thus to $\dd$}.
\end{align}
Additionally, $\gamma\dd^{-1}$ is coprime to $\mm'$, so $\A' = [\gamma\cc^{-1}\dd^{-1}]$ defines a class in $\Cl_{\mm,\rS}(\OO_F)$.
Moreover, for integral ideals $\aa \in \A$, the ideals $\aa\dd^{-1}$ run over all ideals in the class $\A'$. Equation \eqref{eq:changemodzeta} follows.
\end{proof}

The ray class partial zeta function has an analytic continuation to $\C \setminus \{1\}$ with a simple pole at $s=1$ \cite[Ch.\ VII, Thm.\ 5.9]{neukirch}. Taking a difference of two ray class partial zeta functions cancels the poles and leads to a nicer constant term at $s=1$. This behavior led Stark \cite{stark3} to study certain differences of partial zeta functions.
\begin{defn}[Differenced ray class partial zeta function]
Let $\sR$ be the element of $\Clt_{\mm,\rS}(\OO)$ defined by
\begin{equation}
\sR := \{\alpha\OO : \alpha \equiv -1 \Mod{\mm} \mbox{ and } \rho(\alpha)>0 \mbox{ for all } \rho \in \rS\}.
\end{equation} 
For any ray class $\A \in \Clt_{\mm,\rS}(\OO)$, define 
\begin{equation}
Z_{\mm,\rS}(s,\A) = \zeta_{\mm,\rS}(s,\A) - \zeta_{\mm,\rS}(s,\sR\A).
\end{equation}
\end{defn}

\subsection{Ideal coset partial zeta functions}

To a subset $\rS$ of the set of real embeddings of $F$ and a function $(\rho \mapsto \sn_\rho) : \rS \to \{\pm 1\}$, associate the cone
\begin{equation}
\CC_{\rS}(\sn) = \{\alpha \in F : \sgn(\rho(\alpha))=\sn_\rho \mbox{ for all } \rho \in \rS\}.
\end{equation}

\begin{defn}[Ideal coset partial zeta function]
Let $\bb$ be an ideal of $\OO$ coprime to $\mm$, and let $\gamma_0 \in \bb$. Define
\begin{equation}
C_{\mm,\rS}(\bb,\gamma_0,\sn) := (\gamma_0+\bb\mm) \cap \CC_{\rS}(\sn).
\end{equation}
The \textit{ideal coset partial zeta function} is
\begin{equation}
\xi_{\mm,\rS}^\OO(s,\bb,\gamma_0,\sn) := \sum_{\gamma \in C_{\mm,\rS}(\bb,\gamma_0,\sn)/\U_{\mm,\rS}(\OO)} \Nm(\gamma)^{-s}.
\end{equation}
\end{defn}

In the case when $\mm$ is an ideal of a larger order $\OO'$, the zeta functions $\zeta_{\mm,\rS}^\OO$ and $\zeta_{\mm,\rS}^{\OO'}$ are the same under ideal extension of $\bb$. Specifically, $\U_{\mm,\rS}(\OO) = \U_{\mm,\rS}(\OO)$ by definition, and 
\begin{align}
C_{\mm,\rS}(\bb,\gamma_0,\sn) 
= (\gamma_0+\bb\mm) \cap \CC_{\rS}(\sn)
= (\gamma_0+\bb\OO'\mm) \cap \CC_{\rS}(\sn)
= C_{\mm,\rS}(\bb\OO',\gamma_0,\sn),
\end{align}
so $\xi_{\mm,\rS}^\OO(s,\bb,\gamma_0,\sn) = \xi_{\mm,\rS}^{\OO'}(s,\OO'\bb,\gamma_0,\sn)$.

\subsection{Equivalence of partial zeta functions}

From \Cref{prop:exmonoid} we have the exact sequence of multiplicative monoids
\begin{equation}
\left(\OO/\mm,\times\right)\times\{\pm 1\}^{\rS} \to \Clt_{\mm,\rS}(\OO) \xrightarrow{\phi} \Cl(\OO) \to 1.
\end{equation}
For each $\A_0 \in \Cl(\OO)$, the exact sequence gives a surjection
\begin{equation}
\left(\OO/\mm,\times\right)\times\{\pm 1\}^{\rS} \xsurj{\psi_{\!\A_0}} \phi^{-1}(\A_0).
\end{equation}
This surjection is described by the formula
$\psi_{\!\A_0}(\gamma_0,\sn_\rho) = [\gamma\bb^{-1}]$, where $\bb \in \A_0$ is any representative coprime to $\mm$, and $\gamma \in \OO$ satisfying $\gamma \equiv \gamma_0 \Mod{\mm}$ and $\sgn(\rho(\gamma)) = \sn_\rho$ for $\rho \in \rS$.

\begin{prop}\label{prop:partialzetaeq}
Let $\A_0 \in \Cl(\OO)$, and choose an integral ideal $\bb \in \A_0^{-1}$ such that $\bb$ is coprime to $\mm$. 
Let $\A \in \Clt_{\mm,\rS}(\OO)$ such that $\phi(\A) = \A_0$. 
Choose some $\gamma_0 \in \bb$ and $\sn \in \{\pm 1\}^{\rS}$
such that $\psi_{\!\A_0}(\ol\gamma_0,\sn) = \A$, where $\ol\gamma_0$ is the reduction of $\gamma_0 \Mod{\mm}$.
Then,
\begin{equation}\label{eq:partialzetaeq}
\Nm(\bb)^{-s} \zeta_{\mm,\rS}(s,\A) 
= \left[\U_{\colon{\mm}{\mm+\gamma_0\OO},\rS}(\OO):\U_{\mm,\rS}(\OO)\right]^{-1} \xi_{\mm,\rS}^{\OO}(s,\bb,\gamma_0,\sn).
\end{equation}
\end{prop}
\begin{proof}
For any $\aa \in \A$, write $\aa\bb = \gamma\OO$ for some $\gamma \in \bb$ such that $\gamma \con \gamma_0 \Mod{\mm}$ and $\sgn(\rho(\gamma))=\sn_\rho$. 
The choice of $\gamma$ is determined up to a unit that stabilizes $\gamma_0 \Mod{\mm}$ and is positive at the real places in $\rS$. The group of such units is precisely $\U_{\colon{\mm}{\mm+\gamma_0\OO},\rS}(\OO)$.
The map $\aa \mapsto \gamma$ defines a bijection
\begin{align}
\A &\xrightarrow{\sim} C_{\mm,\rS}(\bb,\gamma_0,\sn)/\U_{\colon{\mm}{\mm+\gamma_0\OO},\rS}(\OO),\\
\aa &\mapsto \gamma.
\end{align}
In terms of zeta functions, we have
\begin{align}
\Nm(\bb)^{-s} \zeta_{\mm,\rS}(s,\A) 
&= \sum_{\aa \in \A} \Nm(\aa\bb)^{-s} \\
&= \sum_{\gamma \in \frac{C_{\mm,\rS}(\bb,\gamma_0,\sn)}{\U_{\colon{\mm}{\mm+\gamma_0\OO},\rS}(\OO)}} \Nm(\gamma)^{-s} \\
&= \left[\U_{\colon{\mm}{\mm+\gamma_0\OO},\rS}(\OO):\U_{\mm,\rS}(\OO)\right]^{-1} \xi_{\mm,\rS}(s,\bb,\gamma_0,\sn). \tag*{\qedhere}
\end{align}
\end{proof}

\subsection{Galois-theoretic partial zeta functions}\label{sec:gal}

In this section, we introduce the Galois-theoretic partial zeta functions used in Tate's formulation of the rank $1$ abelian Stark conjecture in \cite[Sec.\ 4]{tate}.

\begin{defn}[Galois-theoretic partial zeta function]
Let $H/F$ be an abelian Galois extension of number fields. Let $S$ be a finite set of places of $F$ containing all the places that ramify in $H$ as well as all the infinite places of $F$, and let $S=S_{\rm fin} \sqcup S_\infty$ for a set of finite places $S_{\rm fin}$ and a set of infinite places $S_\infty$. For any $\sigma \in \Gal(H/F)$ and $\re(s)>1$, define
\begin{equation}\label{eq:galseries}
    \zeta_S^{\Gal}(\sigma,s) = \sum_{\substack{\aa \subseteq \OO_F \\ (\forall \pp \in S_{\rm fin}) \aa+\pp=\OO_F \\ \Art([\aa])=\sigma}} \Nm(\aa)^{-s}.
\end{equation}
\end{defn}

In the case when $H = H_{\mm,\rS}^{\OO_F}$ is a ray class field of the maximal order, the Galois theoretic partial zeta function is equal to the ray class partial zeta function.
\begin{thm}\label{thm:zetagaleq}
Let $\A \in \Cl_{\mm,\rS}(\OO_F)$ for $\mm$ a nonzero $\OO_F$-ideal and $\rS$ a subset of the real places of $F$. Let $\sigma = \Art(\A) \in \Gal(H_{\mm,\rS}^{\OO_F}/F)$. Let $S_{\rm fin}$ be the set of primes of $\OO_F$ dividing $\mm$, let $S_\infty$ be the set of infinite places of $F$, and let $S = S_{\rm fin} \sqcup S_\infty$.
Then,
\begin{equation}
   \zeta_S^{\Gal}(\sigma,s) = \zeta_{\mm,\rS}(s,\A).
\end{equation}
\end{thm}
\begin{proof}
An integral ideal $\aa \subseteq \OO_F$ is coprime to every $\pp \in S_{\rm fin}$ if and only if it is coprime to $\mm$, and $\Art([\aa])=\sigma$ if and only if $\aa \in \A$.
Thus, the Dirichlet series in \eqref{eq:galseries} has the same terms as the Dirichlet series is \eqref{eq:rayseries}.
\end{proof}
\Cref{thm:zetagaleq} is essentially a restatement of Artin reciprocity, which we have used implicitly here, in terms of zeta functions.

\subsection{The rank 1 abelian Stark--Tate conjecture}

We now state Tate's refinement of the rank 1 abelian Stark conjecture. 
The following conjecture appears as \cite[(4.2) Conjecture ${\rm St}(S,K/k)$]{tate}, but we remove several equivalent statements for conciseness. Our notation differs from Tate's only in that we denote field extension as $H/F$ rather than $K/k$.
\begin{conj}[Stark--Tate conjecture ${\rm ST}(S,H/F)$]\label{conj:stark1}
Let $H/F$ be an abelian extension of number fields, and let $W$ be the number of roots of unity in $H$. Let $S$ be a finite set of places of $F$ containing all the places that ramify in $H$ as well as all the infinite places of $F$, satisfying $\abs{S}\geq 2$. Suppose that $S$ contains a place $\pp$ (finite or infinite) that splits completely in $F$, and let $T = S \setminus \{\pp\}$. Let $U_{S,H}^T$ denote the set of elements $\alpha \in H^\times$ such that its $\QQ$-adic valuations at places $\QQ$ of $H$ satisfy
\begin{align}
& \abs{\alpha}_{\QQ} = 1 \mbox{ for } \QQ | \qq \nin S, &&  \\
& \abs{\alpha}_{\QQ} = 1 \mbox{ for } \QQ | \qq \in T, && \mbox{ if } \abs{T} \geq 2, \mbox{ and}  \\
& \abs{\alpha}_{\QQ} = a \mbox{ for } \QQ | \qq \mbox{ and $a$ constant}, && \mbox{ if } T=\{\qq\}.
\end{align}
Then, there is an element $\e \in U_{S,H}^T$ such that
\begin{equation}
\log\abs{\sigma(\e)}_\PP = -W\zeta_S'(\sigma,0) \mbox{ for each } \sigma \in \Gal(H/F) \mbox{ and } \PP|\pp
\end{equation}
and such that $H(\e^{1/W})$ is abelian over $F$.
\end{conj}

While Tate attributes the full conjecture to Stark in the case when $\pp$ is Archimedean, his statement is stronger that Stark's published conjectures even in that case, so we use this refinement even though only the Archimedean case is important to this paper. Precisely, Stark conjectured that $H(\e^{1/W})$ is normal over $F$, without claiming in print that the group $\Gal(H(\e^{1/W})/F)$ is always abelian \cite{stark4}. We will need Tate's refinement to conclude that the RM values of the Shintani--Faddeev modular cocycle live in abelian extensions of $F$.

We now give an alternatively statement of a special case of the Stark--Tate conjecture.
\begin{conj}\label{conj:stark2}
Let $F$ be a real quadratic field and $\mm$ a nonzero integral $\OO_F$-ideal such that $\mm \neq \OO_F$.
Let $\{\infty_1, \infty_2\}$ be the two real places of $F$ and $\{\rho_1, \rho_2\}$ the corresponding real embeddings.
Let $H = H_{\mm\infty_2}^{\OO_F}$, and let $\tilde\rho_1 : H \to \R$ be any embedding extending $\rho_1$.
Then, for all $\A \in \Cl_{\mm\infty_2}(\OO_F)$, 
there are elements $\e_{\!\A} \in \OO_H^\times$ such that 
\begin{equation}
\tilde\rho_1(\e_{\!\A}) = \exp\!\left(-2\zeta_{\mm\infty_2}'(0,\A)\right),
\end{equation}
$(\Art(\B))(\e_{\!\A}) = \e_{\!\A\B}$ for $\B \in \Cl_{\mm\infty_2}(\OO_F)$, $\abs{\tilde\rho_2(\e_{\!\A})}=1$ for all $\tilde\rho_2 \mbox{ extending } \rho_2$, and $H(\e^{1/2})$ is abelian over $F$.
\end{conj}

\begin{prop}\label{stark1implies2}
\Cref{conj:stark1} implies \Cref{conj:stark2}.
\end{prop}
\begin{proof}
Assume \Cref{conj:stark1}.
The real place $\infty_1$ is unramified in $H/F$, so $H$ has real embeddings; thus, the number of roots of unity in $H$ is $W=2$. 

Let $S = \{\mbox{prime ideals } \pp | \mm \} \cup \{\infty_1, \infty_2\}$, and let $T = S \setminus \{\infty_1\}$. Then $\abs{T} \geq 2$, so
\begin{equation}
U^T_{S,H} = \{\eta \in \OO_H^\times : \abs{\tilde\rho_2(\eta)}=1 \mbox{ for all } \tilde\rho_2 \mbox{ extending } \rho_2\}.
\end{equation}
By \Cref{conj:stark1}, there is an element $\e$ of this group such that
\begin{equation}\label{eq:starkimp1}
\log\abs{\tilde\rho_1(\sigma(\e))} = -2\zeta_S'(\sigma,0) \mbox{ for each } \sigma \in \Gal(H/F)
\end{equation}
and such that $H(\e^{1/2})$ is abelian over $F$. This statement remains true if $\e$ is replaced by $-\e$, so choosing $\e$ appropriately, we may assume $\tilde\rho_1(\e) > 0$. Denote also by $\tilde\rho_1$ an extension of $\tilde\rho_1$ to $H(\e^{1/2})$, and let $\nu = \pm \e^{1/2}$ so that $\tilde\rho_1(\nu) > 0$. We have $\e = \nu^2$, and $\sigma(\e) = \sigma(\nu)^2$ for any $\sigma \in \Gal(H(\e^{1/2})/F)$, and thus $\tilde\rho_1(\sigma(\e)) = \tilde\rho_1(\sigma(\nu))^2 > 0$. Since any $\sigma \in \Gal(H/F)$ may be extended to $H(\e^{1/2})$, it follows that $\tilde\rho_1(\sigma(\e))>0$ in \eqref{eq:starkimp1}, and thus
\begin{equation}
\tilde\rho_1(\sigma(\e)) = \exp\!\left(-2\zeta_S'(\sigma,0)\right) \mbox{ for each } \sigma \in \Gal(H/F).
\end{equation}
For $\A \in \Cl_{\mm\infty_2}(\OO_F)$, set $\e_\A = (\Art(\A))(\e)$. Thus, we have $(\Art(\B))(\e_{\!\A}) = \e_{\!\A\B}$, and $\zeta_S'(\sigma,0) = \zeta_{\mm\infty_2}'(0,\A)$ by \Cref{thm:zetagaleq}, completing the proof.
\end{proof}

We will actually use an essentially equivalent formulation in terms of the differenced zeta function of a class in the ray class monoid. We show that \Cref{conj:stark2} (and thus \Cref{conj:stark1}) implies the statement we need. 
\begin{prop}\label{prop:stark3}
Assume \Cref{conj:stark2}.
Let $F$ be a real quadratic field and $\mm$ a nonzero integral $\OO_F$-ideal.
Let $\{\infty_1, \infty_2\}$ be the two real places of $F$ and $\{\rho_1, \rho_2\}$ the corresponding real embeddings.
Let $H = H_{\mm\infty_2}$, and let $\tilde\rho_1 : H \to \R$ be any embedding extending $\rho_1$.
Then, for all $\A \in \Clt_{\mm\infty_2}(\OO_F)$, 
there are elements $\e_{\!\A} \in \OO_H^\times$ such that 
\begin{equation}
\tilde\rho_1(\e_{\!\A}) = \exp\!\left(-Z_{\mm\infty_2}'(0,\A)\right), \label{eq:stark3}
\end{equation}
$(\Art(\B))(\e_{\!\A}) = \e_{\!\A\B}$ for $\B \in \Cl_{\mm\infty_2}(\OO_F)$, $\abs{\tilde\rho_2(\e_{\!\A})}=1$ for all $\tilde\rho_2 \mbox{ extending } \rho_2$, and $H(\e_{\A}^{1/2})$ is abelian over $F$.
For $\A \in \Cl_{\mm\infty_2}(\OO_F)$, these are the same $\e_\A$ as in \Cref{conj:stark2}.
\end{prop}
\begin{proof}
If $\A \in \ZClt_{\mm\infty_2}(\OO_F)$ (including the case $\mm = \OO_F$), then $\zeta_{\mm\infty_2}(s,\A)=\zeta_{\mm\infty_2}(s,\sR\A)$, so $Z_{\mm\infty_2}'(0,\A)=0$, and we may take $\e_{\!\A}=1$. Henceforth, assume $\A \nin \ZClt_{\mm\infty_2}(\OO_F)$.

We now consider $\A \in \Cl_{\mm\infty_2}(\OO_F)$. Standard results in the theory of $L$-functions give
\begin{align}
\zeta_{\mm\infty_2}(0,\sR\A) &= \zeta_{\mm\infty_2}(0,\A); \label{eq:zetaatzero1} \\
\zeta_{\mm\infty_2}'(0,\sR\A) &= -\zeta_{\mm\infty_2}'(0,\A). \label{eq:zetaatzero2}
\end{align}
Specifically, \eqref{eq:zetaatzero2} is proven as \cite[Prop.\ 5]{tangedal} by means of writing the partial zeta functions in terms of $L$-functions of finite-order Hecke characters, and \eqref{eq:zetaatzero1} is proven similarly.
Thus, under the assumption of \Cref{conj:stark2},
\begin{align}
\tilde\rho_1(\e_{\!\A}) 
&= \exp\!\left(-2\zeta_{\mm\infty_2}'(0,\A)\right) 
= \exp\!\left(-\zeta_{\mm\infty_2}'(0,\A) + \zeta_{\mm\infty_2}'(0,\sR\A)\right)
= \exp\!\left(-Z_{\mm\infty_2}'(0,\A)\right).
\end{align}

Finally, consider a general $\A \in \Clt_{\mm\infty_2}(\OO_F)$.
In the case when $\A = [\aa]$ with $\mm | \aa$, \eqref{eq:stark3} is trivially true with $\e_\A = 1$, and the Galois action is trivial. Otherwise, choose an ideal $\cc \in \rJ_\mm^\ast(\OO_F)$ such that $\cc\A = [(\gamma_0)]$ for some $\gamma_0 \in \OO_F$, and let $\dd = \mm + \gamma_0\OO_F$, $\mm' = \dd^{-1}\mm$, and $\A' = [\gamma\cc^{-1}\dd^{-1}] \in \Cl_{\mm,\rS}(\OO_F)$. 
By \Cref{prop:changemodzeta}, we have
\begin{equation}
Z_{\mm,\rS}(s,\A) = \Nm(\dd)^{-s} Z_{\mm',\rS}(s,\A').
\end{equation}
By \eqref{eq:zetaatzero1}, we also have $Z_{\mm',\rS}(0,\A')=0$; it follows that $Z_{\mm,\rS}'(0,\A) = Z_{\mm',\rS}'(0,\A')$.
Since $\A$ is \textit{not} of the form $\A = [\aa]$ with $\mm | \aa$, we know that $\mm' \neq \OO_F$.
Applying \Cref{conj:stark2} to $\A'$ and taking $\e_\A = \e_{\!\A'}$ proves the rest of the proposition.
\end{proof}

\section{Partial zeta functions at $s=0$ and continued fractions}\label{sec:partialzero}

In this section, we relate this value of a real quadratic ideal coset partial zeta function at $s=0$ to an RM value of the Shintani--Faddeev modular cocycle. That is, we demonstrate that a formula like that given in \Cref{thm:main} is true. We will need to tie up some loose ends to complete the proof of \Cref{thm:main} in \Cref{sec:completing}.

Our proof relies on Shintani's Kronecker limit formula for real quadratic fields, originally proven in \cite{shintani}. Many variants of this formula exist. Our presentation most closely follows Tangedal \cite{tangedal}, whose formula builds on earlier work of Zagier \cite{zagier, zagierfrench}, Arakawa \cite{arakawa}, Hayes \cite{hayes}, and Sczech \cite{sczechstark,sczech}. Yamamoto \cite{yamamoto1} independently developed a similar approach to Tangedal's; see Onodera \cite{onodera} for a comparison. 

In this section, to ease the burden of the notation, we will consider our real quadratic field $F$ to be embedded in $\R$ using the real embedding $\rho_1$, so $\rho_1(\beta) = \beta$ and $\rho_2(\beta) = \beta'$ for $\beta \in F$.

\subsection{Hirzebruch--Jung continued fractions}\label{sec:continuedfractions}

Tangedal's formulation of Shintani's formula uses a particular type of continued fraction expansion. Before discussing zeta functions, we establish some fundamental results about these continued fractions and their connection to expressions for elements of $\SL_2(\Z)$ in terms of the matrices $S = \smmattwo{0}{-1}{1}{0}$ and $T = \smmattwo{1}{1}{0}{1}$.

\begin{defn}
For $a_0, a_1, a_2, \ldots, a_k, \ldots$ real numbers, we denote
\begin{equation}
[a_0, a_1, a_2, \ldots, a_k]_- := a_0 - \dfrac{1}{a_1-\dfrac{1}{a_2-\dfrac{1}{\ddots-\dfrac{1}{a_k}}}}
\end{equation}
and
\begin{equation}
[a_0, a_1, a_2, \ldots]_- := \lim_{k\to\infty} [a_0, a_1, a_2, \ldots, a_k]_- = a_0 - \dfrac{1}{a_1-\dfrac{1}{a_2-\dfrac{1}{\ddots}}}
\end{equation}
provided the limit exists. In the special case when the $a_j$ are integers and $a_j \geq 2$ for all $j \geq 1$,
we call these expressions \textit{Hirzebruch--Jung continued fractions} or \textit{HJ continued fractions}.
For an irrational real number $\alpha$, its unique such expression of the form
\begin{equation}
\alpha = [a_0,a_1,a_2,\ldots]_-
\end{equation}
is its \textit{Hirzebruch--Jung (HJ) continued fraction expansion}. We will also use the notation $[a_0, a_1, \ldots, a_k, \ol{a_{k+1}, a_{k+2}, \ldots, a_{k+\ell}}]_-$ for a periodic Hirzebruch--Jung continued fraction.
\end{defn}
Our terminology follows Popescu--Pampu \cite{popescu}. These continued fractions play a crucial role in Hirzebruch's work on Hilbert modular surfaces \cite{hirzebruch} and were studied earlier by Jung in a related context \cite{jung}.
In the literature, HJ continued fractions are also called \textit{``-'' continued fractions} \cite{katok} or \textit{minus continued fractions} \cite{katoknotes}, \textit{backwards continued fractions} \cite{afback,blback}, \textit{negative-regular continued fractions} \cite{negreg}, \textit{reduced regular continued fractions} \cite{redreg}, and \textit{by-excess continued fractions} \cite{lm}. They are closely connected to the $\SL_2(\Z)$ reduction theory of indefinite integral binary quadratic forms. Proofs of the fundamental properties of HJ continued fractions and their connection to reduction theory are available in the work of Svetlana Katok \cite{katok, katoknotes}.

\begin{prop}\label{prop:quadhj}
Let $\beta$ be a real quadratic number. Then, $\beta$ has a periodic HJ continued fraction expansion
\begin{equation}
\beta = [a_0, \ldots, a_k, \ol{b_1, \ldots, b_\ell}]_-.
\end{equation}
This expression is purely periodic if and only if $0<\beta'<1<\beta$.
\end{prop}
\begin{proof}
See \cite[Thm.\ 1.3 and Thm.\ 1.4]{katoknotes}.
\end{proof}

\begin{defn}\label{defn:fracpart}
If $\r = \smcoltwo{r_1}{r_2}$, set $\{\r\} := \smcoltwo{r_1-\floor{r_1}-1}{r_2-\floor{r_2}}$. That is, $\{\r\} \con \r \Mod{1}$, and $\{\r\} = \r$ if and only if $-1 \leq r_1 < 0$ and $0 \leq r_2 < 1$.
\end{defn}

\begin{defn}\label{defn:cycledata}
Let $\beta$ be a real quadratic number satisfying $0<\beta'<1<\beta$, and let $\r \in \frac{1}{N}\Z^2 \setminus \Z^2$ satisfying $\{\r\}=\r$. The \textit{associated HJ cycle data} are:
\begin{itemize}
\item positive integers $k, \ell$,
\item integers $b_n \geq 2$, $n \in \Z/\ell\Z$,
\item real quadratic numbers $\beta_n$ for $n \in \Z/\ell\Z$,
\item matrices $P \in \SL_2(\Z)$, $A \in \Gamma_\r$,
\item matrices $A_{m,n} \in \SL_2(\Z)$ for $m,n \in \Z$,
\item $\r_n \in \frac{1}{N}\Z^2 \setminus \Z^2$ for $n \in \Z/k\ell\Z$, and
\item real quadratic numbers $w_n$ for $n \in \Z/k\ell\Z$.
\end{itemize}
(Sometimes we will use the term ``HJ cycle data'' while only requiring a subset of this data.) They are defined as follows.
The number $\beta$ has a purely periodic Hirzeburch--Jung continued fraction expansion
\begin{equation}
\beta = [\ol{b_0, b_1, \ldots, b_{\ell-1}}]_-
\end{equation}
of period $\ell$; treat the indices as elements of $\Z/\ell\Z$, so $b_{n+\ell} = b_n$.
Let $\beta=\beta_0, \beta_1, \ldots, \beta_{\ell-1}$ be the real numbers with Hirzebruch--Jung continued fraction expansions given by the cyclic permutations of the Hirzebruch--Jung continued fraction expansion of $\beta$; that is,
\begin{equation}
\beta_n = [\ol{b_n, b_{n+1} \ldots, b_{n+\ell-1}}]_-.
\end{equation}
Let $P = T^{b_0}ST^{b_1}S \cdots T^{b_{\ell-1}}S$, so $P\cdot\beta = \beta$ and $\langle -I, P \rangle$ is the stabilizer of $\beta$ under of the $\SL_2(\Z)$-action by fractional linear transformations. 
Choose $k \in \N$ so that $A = P^k$ is the smallest positive power of $P$ in $\Gamma_\r$. 
Define the matrices $A_{m,n}$ for $m,n \in \Z$ by
\begin{equation}
A_{m,n} = 
\begin{cases}
T^{b_m}ST^{b_{m+1}}S \cdots T^{b_{n-1}}S & \mbox{for } m < n, \\
I & \mbox{for } m=n, \\
\left(T^{b_n}ST^{b_{n+1}}S \cdots T^{b_{m-1}}S\right)^{-1} & \mbox{for } m > n.
\end{cases}
\end{equation}
Thus, $A_{n_1,n_2}A_{n_2,n_3} = A_{n_1,n3}$, $A_{0,\ell}=P$, $A_{0,k\ell}=A$, $\beta_m = A_{m,n}\cdot\beta_n$, and in particular $\beta_n = A_{n,0}\cdot\beta = A_{0,n}^{-1}\cdot\beta$.

Let $\r_n = \{A_{n,0}\r\}$, and let $w_n = \sympt{\r_n}{\beta_n}$. The index $n \in \Z/k\ell\Z$.
\end{defn}

\begin{lem}\label{lem:betajs}
Let $\beta$ be a real quadratic number with $0<\beta'<1<\beta$, and let $\beta_n$ and $A_{m,n}$ be associated HJ cycle data, as in $\Cref{defn:cycledata}$. Then,
\begin{equation}\label{eq:betaprodj}
j_{\!A_{m,n}}(\beta_n) = 
\begin{cases}
\beta_{m+1} \cdots \beta_{n} & \mbox{if } m < n, \\
1 & \mbox{if } m=n, \\
\left(\beta_{n+1} \ldots \beta_m\right)^{-1} & \mbox{if } m>n.
\end{cases}
\end{equation}
\end{lem}
\begin{proof}
If $m=n$, \eqref{eq:betaprodj} holds because both sides are equal to $1$. We will prove the recursion $j_{\!A_{m,n}}(\beta_n) = \beta_{m+1}j_{\!A_{m+1,n}}(\beta_n)$; then \eqref{eq:betaprodj} will follow by induction on $n-m$ (for $m<n$) and induction on $m-n$ (for $m>n$).

To prove the recursion, use the fact that $A_{m,n} = T^{b_m} S A_{m+1,n}$ and the cocycle property of $j$. We have:
\begin{align}
j_{\!A_{m,n}}(\beta_n) 
&= j_{T^{b_m} S A_{m+1,n}}(\beta_n) \\
&= j_{T^{b_m} S}(A_{m+1,n}\cdot\beta_n)j_{\!A_{m+1,n}}(\beta_n) \\
&= j_{\smmattwo{b_m}{-1}{1}{0}}j_{\!A_{m+1,n}}(\beta_n) \\
&= \beta_{m+1}j_{\!A_{m+1,n}}(\beta_n). \tag*{\qedhere}
\end{align}
\end{proof}

\begin{lem}\label{lem:betaunit}
Let $\beta$ be a real quadratic number with $0<\beta'<1<\beta$, and let $\beta_n$ be associated HJ cycle data, as in $\Cref{defn:cycledata}$. Let $\e>1$ be the smallest totally positive unit greater than 1 in $\Q(\beta)$ such that $\e(\beta\Z + \Z) = \beta\Z+\Z$ (that is, $\e$ is the fundamental totally positive unit of the muliplier ring of $\beta\Z+\Z$). Then, the product
\begin{equation}
\beta_0 \cdots \beta_{\ell-1} = \e.
\end{equation}
\end{lem}
\begin{proof}
Follows from \Cref{lem:betajs} and \Cref{prop:jevalreal}.
(Note that this lemma is also stated in \cite[p.\ 215]{hirzebruch}.)
\end{proof}

\begin{prop}\label{prop:betatogamma}
Let $\beta$ be a real quadratic number, and write
\begin{equation}\label{eq:betahj}
\beta = [a_0, \ldots, a_k, \ol{b_0, \ldots, b_{\ell-1}}]_-.
\end{equation}
Then the 2-by-2 matrix
\begin{equation}
A = (T^{a_0}S \cdots T^{a_k}S) (T^{b_0}S \cdots T^{b_{\ell-1}}S) (T^{a_0}S \cdots T^{a_k}S)^{-1}
\end{equation}
defines an element of $\SL_2(\Z)$ 
with attracting fixed point $\beta$.
If HJ-expansion \eqref{eq:betahj} for $\beta$ is primitive, then $\stab_{\SL_2(\Z)}(\beta) = \langle -I, A \rangle$.
\end{prop}
\begin{proof}
For any $\alpha \in \R$, a direct calculation shows that
\begin{equation}\label{eq:onetee}
T^n S\smcoltwo{\alpha}{1} = \alpha\smcoltwo{[n,\alpha]_-}{1}.
\end{equation}
Set $b_i := b_{i \Mod{\ell}}$ for all $i \in \Z$, and define the real numbers
\begin{align}
\alpha_n &= [a_n, \ldots, a_k, \ol{b_0, \ldots, b_{\ell-1}}]_- \mbox{ and}\\ 
\beta_n &= [\ol{b_n, \ldots, b_{n+\ell-1}}]_-.
\end{align}
Applying \eqref{eq:onetee} repeatedly (using induction), 
\begin{align}
A \smcoltwo{\beta}{1} 
&= (\alpha_1\cdots\alpha_{k+1})(\beta_1\cdots\beta_{\ell})(\alpha_{k+1}^{-1}\cdots\alpha_1^{-1})\smcoltwo{\beta}{1} 
= \e\smcoltwo{\beta}{1},
\end{align}
where $\e = \beta_1\cdots\beta_{\ell} = \beta_0\cdots\beta_{\ell-1} > 1$ is a generator of the totally positive unit group of $\OO^\times$ for the real quadratic order $\OO = \colonideal{\beta\Z+\Z}{\beta\Z+\Z}$ by \Cref{lem:betaunit}. 
The group $\stab_{\SL_2(\Z)}(\beta)$ is a discrete subgroup of
\begin{equation}
\stab_{\SL_2(\R)}(\beta) = \left\{\pm\smmattwo{\beta}{\beta'}{1}{1}\smmattwo{e^v}{0}{0}{e^{-v}}\smmattwo{\beta}{\beta'}{1}{1}^{-1} : v \in \R\right\}.
\end{equation}
The eigenvalues of $A$ are $\e$ and $\e'=\e^{-1}$, corresponding to $v=\log{\e}$.
A general element of the stabilizer looks like
\begin{equation}
\pm\smmattwo{\beta}{\beta'}{1}{1}\smmattwo{e^v}{0}{0}{e^{-v}}\smmattwo{\beta}{\beta'}{1}{1}^{-1} 
= \frac{\pm 1}{\beta - \beta'}\smmattwo{\beta e^v - \beta' e^{-v}}{-\beta\beta'(e^v-e^{-v})}{e^v-e^{-v}}{-(\beta' e^v - \beta e^{-v})},
\end{equation}
so if it is contained in $\SL_2(\Z)$, then $e^v - e^{-v} = (\beta - \beta')m$ for some $m \in \Z$, which is only possible if $v \in (\log \e)\Z$. 
Thus, $\stab_{\SL_2(\Z)}(\beta) = \pm \langle A \rangle = \langle -I, A \rangle$.
\end{proof}

The following lemma relates the HJ continued fraction expansion of a real quadratic number $\beta$ with that of its nontrivial Galois conjugate $\beta'$.
\begin{lem}\label{lem:hjconj}
Let $\beta$ be a real quadratic number with purely periodic Hirzebruch--Jung continued fraction expansion
\begin{equation}
\beta = [\ol{\{2\}^{n_0}, m_1+3, \{2\}^{n_1}, \ldots, m_k+3, \{2\}^{n_k}}]_-,
\end{equation}
for integers $m_j, n_j \geq 0$, where the notation $\{2\}^n$ stands for $n$ consecutive $2$s.
Then, the Hirzebruch--Jung continued fraction expansion of $\beta'$ is
\begin{equation}
\beta' = [1,n_k+2,\ol{\{2\}^{m_k},n_{k-1}+3,\{2\}^{m_{k-1}}, \ldots, n_1+3, \{2\}^{m_1},n_0+n_k+3}]_-.
\end{equation}
\end{lem}
\begin{proof}
By \Cref{prop:betatogamma}, $\beta$ is the attracting fixed point of 
\begin{equation}
A = (T^2S)^{n_0}\prod_{i=1}^k T^{m_i+3}S(T^2S)^{n_i}.
\end{equation}
The Galois conjugate $\beta'$ is the attracting fixed point of $A^{-1}$. To prove the lemma, we express $A^{-1}$ in the form given in the conclusion of \Cref{prop:betatogamma}.

Let $B = TST = \smmattwo{1}{0}{1}{1}$, and rewrite $A$ as
\begin{equation}
A = TB^{n_0+1}\left(\prod_{i=1}^k (ST)^{-1}T^{m_i+1}(ST)B^{n_i+1}\right)T^{-1}.
\end{equation}
Furthermore, using the relation $B=ST^{-1}S^{-1}$
and setting $C=STS$, we may write
\begin{equation}
A = TST^{-(n_0+1)}\left(\prod_{i=1}^k C^{-1}T^{m_i+1}CT^{-(n_i+1)}\right)TS^{-1}T^{-1}
\end{equation}
We may then use the relations $C = T^{-1}ST^{-1}$ and $C^{-1} = TS^{-1}T$ to write 
\begin{equation}
A = TST^{-n_0}\left(\prod_{i=1}^k S^{-1}T^{m_i+1}ST^{-(n_i+1)}\right)S^{-1}T^{-1}
\end{equation}
Inverting $A$ will switch the role of $T$ and $T^{-1}$, whereas $S^{-1} = -S$, and we can now see that this will essentially switch the role of the $m_j$ and the $n_j$. Specifically,
\begin{align}
A^{-1} 
&= TS\left(\prod_{i=0}^{k-1} T^{n_{k-i}+1}S^{-1}T^{-(m_{k-i}+1)}S\right)T^{n_0}S^{-1}T^{-1} \\
&= (TSTT^{n_k}S^{-1}T^{-1}S^{-1}T^{-1}) D (TSTST^{-n_k}T^{-1}S^{-1}T^{-1}),
\end{align}
where
\begin{align}
D
&= TST^{-m_k}\left(\prod_{i=1}^{k-1}ST^{n_{k-i}+1}S^{-1}T^{-(m_{k-i}+1)}\right)ST^{n_0+n_k+1}S^{-1}T^{-1}S^{-1}T^{-1} \\
&= TST^{-m_k}\left(\prod_{i=1}^{k-1}S^{-1}T^{n_{k-i}+1}ST^{-(m_{k-i}+1)}\right)ST^{n_0+n_k+1}S^{-1}T^{-1}S^{-1}T^{-1} \\
&= (T^2S)^{m_k} \left(\prod_{i=1}^{k-1} T^{n_{k-i}+3}S(T^2S)^{m_{k-i}}\right)T^{n_0+n_k+3}S.
\end{align}
We also simplify
\begin{align}
TSTT^{n_k}S^{-1}T^{-1}S^{-1}T^{-1} &= -TST^{n_k+2}S; \\
TSTST^{-n_k}T^{-1}S^{-1}T^{-1} &= -(TST^{n_k+2}S)^{-1}.
\end{align}
Thus, $A^{-1} = PDP^{-1}$ for $P=TST^{n_k+2}S$.
Using \Cref{prop:betatogamma} again, the lemma follows.
\end{proof}

\subsection{Shintani's limit formula and Stark--Tangedal--Yamamoto class invariants}

Shintani's formula expresses the derivative of a zeta value at zero as a product of values of the double sine function. In Shintani's original formulation \cite{shintani}, the product is not uniquely determined but depends on certain choices. We present a formulation following Tangedal \cite{tangedal} that gives a canonical product (in the real quadratic case) whose terms are determined by a certain Hirzebruch--Jung continued fraction expansion.

\begin{defn}
Let $s \in \C$ with $\re(s)>2$, $x_1, x_2 > 0$ and $\omega_1, \omega_2 > 0$.
The \textit{Shintani zeta function} in dimension 2 is the function
\begin{equation}
z_2(s,(x_1,x_2),(\omega_1,\omega_2)) := \sum_{m=0}^\infty \sum_{n=0}^\infty \left((x_1+m\omega_1+n)(x_2+m\omega_2+n)\right)^{-s}.
\end{equation}
\end{defn}
More generally, one may allow the parameters $x_1, x_2, \omega_1, \omega_2$ to take complex values, but we do not need to do so in this paper. (As a warning, our use of the variables $x_i$ deviates from Shintani's \cite{shintani} and Tangedal's \cite{tangedal}, despite the similar notation.)

The following proposition is a special case of Shintani decomposition, generalized slightly to the case of arbitrary orders in real quadratic fields. See \cite[Chap.\ VII]{neukirch} for an exposition of Shintani decomposition for the maximal order in a general totally real number fields.
\begin{prop}\label{prop:shintanidecomp}
Let $\OO$ be an order in a real quadratic field $F$, and let $\mm$ be an integral ideal of $\OO$ and $\rS=\{\infty_1,\infty_2\}$.
Let $\bb$ be a fractional ideal of $\OO$. Write $\bb\mm = \alpha(\beta\Z+\Z)$ for some $\alpha,\beta \in F$ satisfying $0 < \beta' < 1 < \beta$, and such that $\sgn(\rho(\alpha)) = \sn_\rho$ for $\rho \in \{\rho_1, \rho_2\}$. 
Let $w_0 = \sympt{\r_0}{\beta}$ for some $\r_0 \in \Q^2 \setminus \Z^2$ with $\{\r_0\}=\r_0$. Associate to $\beta$ and $\r_0$ the HJ cycle data $k,\ell \in \Z$, $b_n \in \Z$, $\beta_n \in F$, $A_{m,n} \in \SL_2(\Z)$, and $w_n \in F$.
For each $n\in \Z$, let $\alpha_n = \alpha (\beta_{n+1} \cdots \beta_0)$ if $n < 0$ and $\alpha_n = \alpha(\beta_1 \cdots \beta_n)^{-1}$ if $n \geq 0$. 
Then,
\begin{align}
&\left[\U_{\mm,\rS}(\OO) : \U_{\colonideal{\mm}{\mm+\alpha w_0\OO},\rS}(\OO)\right]^{-1}\xi_{\mm,\rS}(s,\bb,\alpha w_0,\sn) \\
&= \sum_{n=0}^{k\ell-1} \Nm(\alpha_n)^{-s}z_2(s,(w_n,w_n'),(\beta_n,\beta_n')).
\end{align}
\end{prop}
\begin{proof}
For each $n\in\Z$, $\alpha_{n-1}\Z+\alpha_n\Z = \alpha_n(\beta_n\Z+\Z)$. Moreover, $\b_{n} = b_{n} - \beta_{n+1}^{-1}$, so
\begin{align}
\alpha_n(\beta_n\Z+\Z) 
&= \alpha_n\left(\left(b_n-\beta_{n+1}^{-1}\right)\Z+\Z\right) \\
&= \alpha_{n+1}\left(\left(b_n\beta_{n+1}-1\right)\Z+\beta_{n+1}\Z\right) \\
&= \alpha_{n+1}(\beta_{n+1}\Z+\Z).
\end{align}
By an induction argument, $\alpha_{n-1}\Z+\alpha_n\Z$ is independent of $n$, so $\alpha_{n-1}\Z+\alpha_n\Z = \alpha_{-1}\Z+\alpha_0\Z = \alpha(\beta\Z+\Z)$.

We may write the cone $\CC_\rS(\sn)$ as a disjoint union of subcones
\begin{equation}
\CC_\rS(\sn) = \bigsqcup_{n=-\infty}^\infty (\alpha_{n-1}\Q_{\geq 0} + \alpha_{n}\Q_{> 0}).
\end{equation}
Thus, the cone $C_{\mm,\rS}(\bb,\alpha w_0,\sn)$ may be written as
\begin{equation}\label{eq:coned2}
C_{\mm,\rS}(\bb,\alpha w_0,\sn) = \bigsqcup_{n=-\infty}^\infty (\alpha w_0+\bb \cap \mm) \cap (\alpha_{n-1}\Q_{\geq 0} + \alpha_n\Q_{> 0}).
\end{equation}
We have
\begin{align}
w_0-j_{\!A_{0,n}}(\beta)^{-1}w_n
&= \sympt{\r_0}{\beta_0}-j_{\!A_{0,n}}(\beta)^{-1}\sympt{\r_n}{\beta_n} \\
&= \sympt{\{A_{0,n}\r_n\}}{A_{0,n}\beta_n}-\sympt{A_{0,n}\r_n}{A_{0,n}\beta_n} \\
&= \sympt{\{A_{0,n}\r_n\}-A_{0,n}\r_n}{\beta} \in \beta\Z+\Z. 
\end{align}
Since $j_{\!A_{0,n}}(\beta) = \beta_1\cdots\beta_n$ by \Cref{lem:betajs}, and $\alpha(\beta\Z+\Z) = \alpha_{n-1}\Z+\alpha_n\Z$,
\begin{align}
w_0 &\in \left(\beta_1\cdots\beta_n\right)^{-1}w_n + \beta\Z + \Z, \mbox{ so} \\
\alpha w_0 &\in \alpha_n w_n + \alpha_{n-1}\Z + \alpha_n\Z.
\end{align}
Moreover, $\alpha_n w_n = \alpha_n\sympt{\r_n}{\beta_n} = \alpha_n(r_{n2}\beta_n-r_{n1}) = r_{n2}\alpha_{n-1}+(-r_{n1})\alpha_n$ with $0 \leq r_{n2} < 1$ and $0 < -r_{n1} \leq 1$.
Thus, 
$(\alpha w_0+\bb \cap \mm) \cap (\alpha_{n-1}\Q_{\geq 0} + \alpha_n\Q_{> 0})
= \alpha_n w_n + \alpha_{n-1}\Z_{\geq 0} + \alpha_{n}\Z_{\geq 0} = \alpha_n\left(w_n + \beta_n\Z_{\geq 0} + \Z_{\geq 0}\right)$, so we may rewrite the cone decomposition \eqref{eq:coned2} as
\begin{equation}
C_{\mm,\rS}(\bb,\alpha w_0,\sn) = \bigsqcup_{n=-\infty}^\infty \alpha_n (w_n + \beta_n\Z_{\geq 0} + \Z_{\geq 0}).
\end{equation}
Since $\e^k$ is the smallest totally positive unit in $\OO$ for which $\e^k \alpha w_0 \con \alpha w_0 \Mod{\mm}$, a fundamental domain for the action of $\U_{\colonideal{\mm}{\mm+\alpha w_0},\rS}(\OO)$ on $C_{\mm,\rS}(\bb,\alpha w_0,\sn)$ is given by $\bigsqcup_{n=0}^{k\ell-1} \alpha_n (w_n + \beta_n\Z_{\geq 0} + \Z_{\geq 0})$. Thus,
\begin{align}
\frac{\xi_{\mm,\rS}(s,\bb,\alpha w_0,\sn) }{\left[\U_{\mm,\rS}(\OO) : \U_{\colonideal{\mm}{\mm+\alpha w_0\OO},\rS}(\OO)\right]}
&= \sum_{\alpha \in C_{\mm,\rS}(\bb,\alpha w_0,\sn)/\U_{\colonideal{\mm}{\mm+\alpha w_0\OO},\rS}(\OO)} \Nm(\alpha)^{-s} \\
&= \sum_{n=0}^{k\ell-1} \sum_{m_1=0}^\infty\sum_{m_2=0}^\infty \Nm(\alpha_n(w_n+m_1\beta_n+m_2))^{-s} \\
&= \sum_{n=0}^{k\ell-1} \Nm(\alpha_n)^{-s}z_2(s,(w_n,w_n'),(\beta_n,\beta_n')). \tag*{\qedhere}
\end{align}
\end{proof}

The following result is proven in Shintani's original work on his Kronecker limit formula for real quadratic fields. Shintani's full proof, by manipulation of a contour integral expression for the Shintani zeta function, is omitted from this paper.
\begin{lem}\label{lem:z2primeof0}
The Shintani zeta function has a meromorphic continuation in $s$ that is analytic at $s=0$.
Let $x_1,x_2,\omega_1,\omega_2 > 0$, and set $r_1 = \frac{\omega_2x_1-\omega_1x_2}{\omega_1-\omega_2}$ and $r_2=\frac{x_1-x_2}{\omega_1-\omega_2}$ (that is, $x_1=r_2\omega_1-r_1$ and $x_2=r_2\omega_2-r_1$).
The evaluation of the Shintani zeta function at $s=0$ is given by
\begin{align}
z_2(0,(x_1,x_2),(\omega_1,\omega_2))
&= \tfrac{1}{4}\left(\tfrac{1}{\omega_1}+\tfrac{1}{\omega_2}\right)B_2(-r_1) + B_1(-r_1)B_1(r_2)+\tfrac{1}{4}(\omega_1+\omega_2)B_2(r_2). \label{eq:z2of0}
\end{align}
The evaluation of the first derivative (in $s$) of the Shintani zeta function at $s=0$ is given by
\begin{align}
z_2'(0,(x_1,x_2),(\omega_1,\omega_2))
&= \log\!\left(\frac{\Gamma_2(x_1,\omega_1)\Gamma_2(x_2,\omega_2)}{\rho_2(\omega_1)\rho_2(\omega_2)}\right) + \frac{\omega_1-\omega_2}{4\omega_1\omega_2}\log\!\left(\frac{\omega_2}{\omega_1}\right)B_2(-r_1). \label{eq:z2primeof0}
\end{align}
In both equations, $B_1$ and $B_2$ denote the Bernoulli polynomials
\begin{align}
B_1(r) &= r-\foh\mbox{ and} \\
B_2(r) &= r^2-r+\tfrac{1}{6}.
\end{align}
\end{lem}
\begin{proof}
The meromorphic continuation of $z_2$ is proven as \cite[Prop.\ 1]{shintanievaluation}.
\Cref{eq:z2primeof0} is \cite[Prop.\ 3]{shintani}, and \eqref{eq:z2of0} is proven within Shintani's proof of that proposition, on \cite[p.\ 177]{shintani}. We have translated both equations into our preferred notation.
\end{proof}

The expressions in \Cref{lem:z2primeof0} become simpler when one takes a certain difference of two $z_2$-values.
\begin{lem}\label{lem:z2primeof0diff}
For $0<x_1<\omega_1$ and $0<x_2<\omega_2$,
\begin{equation}\label{eq:z2of0diff}
z_2(0,(\omega_1+1-x_1,\omega_2+1-x_2),(\omega_1,\omega_2)) - z_2(0,(x_1,x_2),(\omega_1,\omega_2)) = 0
\end{equation}
and
\begin{align}
&z_2'(0,(\omega_1+1-x_1,\omega_2+1-x_2),(\omega_1,\omega_2)) - z_2'(0,(x_1,x_2),(\omega_1,\omega_2)) \\
&= \log\!\left(\Sin_2(x_1,\omega_1)\Sin_2(x_2,\omega_2)\right). \label{eq:z2primeof0diff}
\end{align}
\end{lem}
\begin{proof}
As in \Cref{lem:z2primeof0}, set $r_1 = \frac{\omega_2x_1-\omega_1x_2}{\omega_1-\omega_2}$ and $r_2=\frac{x_1-x_2}{\omega_1-\omega_2}$. We have
\begin{align}
x_1 &= r_2\omega_1-r_1, & && 
\omega_1+1-x_1 &= (1-r_2)\omega_1-(-1-r_1), \\
x_2 &= r_2\omega_2-r_1, &
\text{and} &&
\omega_2+1-x_2 &= (1-r_2)\omega_2-(-1-r_1).
\end{align}
Using the relations $B_1(1-r) = -B_1(-r)$ and $B_2(1-r) = B_2(r)$ together with \eqref{eq:z2of0}, we obtain \eqref{eq:z2of0diff}.

Again using the relation $B_2(1-r) = B_2(r)$ (specifically, $B_2(1+r_1)=B_2(-r_2)$), we obtain cancellation of the second summand of \eqref{eq:z2primeof0} in the expression
\begin{align}
&z_2'(0,(\omega_1+1-x_1,\omega_2+1-x_2),(\omega_1,\omega_2)) - z_2'(0,(x_1,x_2),(\omega_1,\omega_2))  \\
&= \log\!\left(\tfrac{\Gamma_2(\omega_1+1-x_1,\omega_1)\Gamma_2(\omega_2+1-x_2,\omega_2)}{\rho_2(\omega_1)\rho_2(\omega_2)}\right) - \log\!\left(\tfrac{\Gamma_2(x_1,\omega_1)\Gamma_2(x_2,\omega_2)}{\rho_2(\omega_1)\rho_2(\omega_2)}\right) \\
&= \log\!\left(\tfrac{\Gamma_2(\omega_1+1-x_1,\omega_1)\Gamma_2(\omega_2+1-x_2,\omega_2)}{\Gamma_2(x_1,\omega_1)\Gamma_2(x_2,\omega_2)}\right) \\
&= \log\!\left(\Sin_2(x_1,\omega_1)\Sin_2(x_2,\omega_2)\right),
\end{align}
proving \eqref{eq:z2primeof0diff}.
\end{proof}

If $\A$ is a ray class in $\Cl_{\mm\infty_1\infty_2}(\OO_F)$, Tangedal uses the double sine function to define class invariants $U^{(i)}_{\mm}(\A) \in \R$ for $i \in \{1,2\}$ \cite{tangedal}. He shows that, in the unproven cases of the real quadratic abelian Stark conjectures, these invariants coincide with positive square roots of Stark units. Tangedal's invariants (and similar invariants defined by Yamamoto \cite{yamamoto1}) generalize in a straightforward manner to both nonmaximal orders and imprimitive classes, so we describe them in that context.

\begin{defn}\label{defn:tangedalu}
Let $\OO$ be an order in a real quadratic field $F$, and let $\mm$ be an integral ideal of $\OO$.
Let $\A \in \Clt_{\mm\infty_1\infty_2}(\OO)$. 
Let $(\r,\beta)$ be a reduced representative of $\tilde\Upsilon_\mm(\A)$.
Associate to $(\r,\beta)$ the HJ cycle data $k, \ell \in \Z$ and $\beta_i, w_i \in F$. Define the \textit{Stark--Tangedal--Yamamoto invariant}
\begin{align}
U^{(j)}_{\mm}(\A)
&= \prod_{i=0}^{k\ell-1} \Sin_2(\rho_j(w_i), \rho_j(\beta_i)) & \mbox{for $j \in \{1,2\}$}.
\end{align}  
\end{defn}

The following proposition is a slight generalization of \cite[Prop.\ 1]{tangedal}.
\begin{prop}\label{prop:Zprimesum}
Let $\OO$ be an order in a real quadratic field $F$, and let $\mm$ be an integral ideal of $\OO$.
If $\A \in \Clt_{\mm\infty_1\infty_2}(\OO)$, then
\begin{equation}\label{eq:Zprimesum}
Z_{\mm\infty_1\infty_2}'(0,\A) 
= -\log\!\left(U_\mm^{(1)}(\A)U_\mm^{(2)}(\A)\right).
\end{equation}
\end{prop}
\begin{proof}
Let $\bb, \alpha, \beta, \r, k, \ell, \beta_i$, and $w_i$ be as in \Cref{defn:tangedalu}.
By \Cref{prop:partialzetaeq},
\begin{equation}
\Nm(\bb)^{-s} \zeta_{\mm,\rS}(s,\A) 
= \frac{\xi_{\mm,\rS}(s,\bb,\alpha w_0,\sn)}{\left[\U_{\mm,\rS}(\OO) : \U_{\colonideal{\mm}{\mm+\alpha w_0\OO},\rS}(\OO)\right]}.
\end{equation}
For each $n\in \Z$, let  
$\alpha_n = \alpha (\beta_{n+1} \cdots \beta_0)$ if $n < 0$ and $\alpha_n = \alpha(\beta_1 \cdots \beta_n)^{-1}$ if $n \geq 0$. 
By \Cref{prop:shintanidecomp},
\begin{align}
\zeta_{\mm,\rS}(s,\A) 
&= \frac{\Nm(\bb)^{s} \xi_{\mm,\rS}(s,\bb,\alpha w_0,\sn)}{\left[\U_{\mm,\rS}(\OO) : \U_{\colonideal{\mm}{\mm+\alpha w_0\OO},\rS}(\OO)\right]} \\
&= \sum_{n=0}^{k\ell-1} \Nm(\alpha_n\bb^{-1})^{-s}z_2(s,(w_n,w_n'),(\beta_n,\beta_n')).
\end{align}
Similarly, we apply the same results to $\zeta_{\mm,\rS}(s,\mathfrak{R}\A)$.
The associated HJ cycle data include 
$\tilde{w}_n = \beta+1-w_n$, and one obtains
\begin{align}
\zeta_{\mm,\rS}(s,\mathfrak{R}\A) 
&= \frac{\Nm(\bb)^{s} \xi_{\mm,\rS}(s,\bb,\alpha \tilde{w}_0,\sn)}{\left[\U_{\mm,\rS}(\OO) : \U_{\colonideal{\mm}{\mm+\alpha w_0\OO},\rS}(\OO)\right]} \\
&= \sum_{n=0}^{k\ell-1} \Nm(\alpha_n\bb^{-1})^{-s}z_2(s,(\tilde{w}_n,\tilde{w}_n'),(\beta_n,\beta_n')).
\end{align}
Subtracting, 
and using the fact that $Z_{\mm,\rS}(s,\A)=0$ by \eqref{eq:zetaatzero1}, we have
\begin{align}
Z_{\mm,\rS}'(0,\A) 
&= \sum_{n=0}^{k\ell-1} \left(z_2'(0,(w_n,w_n'),(\beta_n,\beta_n'))-z_2'(0,(\tilde{w}_n,\tilde{w}_n'),(\beta_n,\beta_n'))\right) \\
&= \sum_{n=0}^{k\ell-1} -\log\!\left(\Sin_2(w_n,\beta_n)\Sin_2(w_n',\beta_n')\right) \mbox{ by \Cref{lem:z2primeof0diff}} \\
&= -\log\!\left(U_\mm^{(1)}(\A)U_\mm^{(2)}(\A)\right),
\end{align}
proving the proposition.
\end{proof}

The following theorem may be compared to \cite[Thm. 1]{tangedal}.
\begin{thm}\label{thm:ZprimeU}
Let $\OO$ be an order in a real quadratic field $F$, and let $\mm$ be an integral $\OO$-ideal. 
Consider the surjective monoid map $\phi: \Clt_{\mm\infty_1\infty_2}(\OO) \to \Clt_{\mm\infty_2}(\OO)$.
If $\A \in \Clt_{\mm\infty_2}(\OO) \setminus \ZClt_{\mm\infty_2}(\OO)$,
and $\tilde\A \in \Clt_{\mm\infty_1\infty_2}(\OO)$ such that $\phi(\tilde\A) = \A$,
then
$U^{(1)}(\A) := U^{(1)}(\tilde\A)$ is well-defined. Moreover,
\begin{equation}\label{eq:ZprimeU}
Z_{\mm\infty_2}'(0,\A) 
= -t\log\!\left(U_\mm^{(1)}(\A)\right),
\end{equation}
where $t = \abs{\phi^{-1}(\A)} \in \{1,2\}$.
Additionally, $U^{(1)}(\A)$ depends only on $\Upsilon_\mm(\A)$.
\end{thm}
\begin{proof}
Define the following classes in $\Cl_{\mm\infty_1\infty_2}(\OO)$:
\begin{align}
\sR_{+-} &= \{\gamma\OO : \gamma \con 1 \Mod{\mm}, \rho_1(\gamma)>0, \rho_2(\gamma)<0\}, \\
\sR_{-+} &= \{\gamma\OO : \gamma \con 1 \Mod{\mm}, \rho_1(\gamma)<0, \rho_2(\gamma)>0\}, \mbox{ and}\\
\sR_{--} &= \{\gamma\OO : \gamma \con 1 \Mod{\mm}, \rho_1(\gamma)<0, \rho_2(\gamma)<0\} = \sR_{+-}\sR_{-+}.
\end{align}
In $\Cl_{\mm\infty_2}(\OO)$, we also have
\begin{align}
\sR &= \{\gamma\OO : \gamma \con 1 \Mod{\mm}, \rho_2(\gamma)<0\} = \sR_{+-} \cup \sR_{--}.
\end{align}
Choose some $\tilde{\A} \in \Cl_{\mm\infty_1\infty_2}(\OO)$ such that $\phi(\tilde{\A}) = \A$. It may be shown by the same method as \cite[Prop.\ 7]{arakawa} and \cite[Cor.\ 1]{tangedal} that
\begin{align}\label{eq:Uconjrels}
U_{\mm}^{(1)}(\sR_{-+}\tilde{\A}) &= U_{\mm}^{(1)}(\tilde{\A})
&&\mbox{and}&
U_{\mm}^{(2)}(\sR_{-+}\tilde{\A}) &= \left(U_{\mm}^{(2)}(\tilde{\A})\right)^{-1};
\end{align}
we omit the details. We complete the proof by cases.

\textit{Case 1: Assume that $t=1$.}
Then $\sR_{-+}\A = \A$ and $\sR_{--}\A = \sR_{+-}\A$.
By \eqref{eq:Uconjrels}, it follows that $\left(U_{\mm}^{(2)}(\sR_{-+}\tilde{\A})\right)^2 = 1$; the properties of the double sine function also imply that $U_{\mm}^{(2)}(\tilde{\A}) > 0$ (because all the terms in the product defining it are positive), so $U_{\mm}^{(2)}(\tilde{\A}) = 1$.
We have $Z_{\mm\infty_2}(s,\A) = Z_{\mm\infty_1\infty_2}(s,\tilde\A)$, so by \Cref{prop:Zprimesum}
\begin{equation}
Z_{\mm\infty_2}'(0,\A)
= -\log\!\left(U_{\mm}^{(1)}(\tilde{\A})U_{\mm}^{(2)}(\tilde{\A})\right) 
= -\log\!\left(U_{\mm}^{(1)}(\tilde{\A})\right).
\end{equation}

\textit{Case 2: Now assume that $t=2$.}
Write $\A = \tilde{\A} \cup \sR_{-+}\tilde{\A}$, $\sR\A = \sR_{+-}\tilde{\A} \cup \sR_{--}\tilde{\A}$, and
\begin{align}
Z_{\mm\infty_2}(s,\A)
&= Z_{\mm\infty_1\infty_2}(s,\tilde{\A}) + Z_{\mm\infty_1\infty_2}(s,\sR_{-+}\tilde{\A}).
\end{align}
Thus, by \Cref{prop:Zprimesum},
\begin{align}
Z_{\mm\infty_2}'(0,\A)
&= -\log\!\left(U_{\mm}^{(1)}(\tilde{\A})U_{\mm}^{(2)}(\tilde{\A})\right) - \log\!\left(U_{\mm}^{(1)}(\sR_{-+}\tilde{\A})U_{\mm}^{(2)}(\sR_{-+}\tilde{\A})\right) \\
&= -\log\!\left(U_{\mm}^{(1)}(\tilde{\A})U_{\mm}^{(2)}(\tilde{\A})U_{\mm}^{(1)}(\sR_{-+}\tilde{\A})U_{\mm}^{(2)}(\sR_{-+}\tilde{\A})\right).
\end{align}
Therefore, it follows that $Z_{\mm\infty_2}'(0,\A) = -2\log\!\left(U_{\mm}^{(1)}(\tilde{\A})\right)$.

We've now shown that $Z_{\mm\infty_2}'(0,\A) = -t\log\!\left(U_{\mm}^{(1)}(\tilde{\A})\right)$ in all cases. It follows that $U_{\mm}^{(1)}(\tilde{\A})$ depends only on $\Upsilon_\mm(\A)$, and in particular, $U_{\mm}^{(1)}(\A) := U_{\mm}^{(1)}(\tilde{\A})$ is well-defined.
\end{proof}

\subsection{Telescoping the product}

We need the following identity for the double sine function.
\begin{lem}
$\SS_2(\omega_1+\omega_2-z;\omega_1,\omega_2) = \SS_2(z;\omega_1,\omega_2)^{-1}$.
\end{lem}
\begin{proof}
Follows directly from the definition of the double sine function, specifically \eqref{eq:dsinegamma}.
\end{proof}

We also need the following identities for $\varpi_\r$, whose proofs are straightforward.
\begin{lem}\label{lem:taut}
Let $A \in \SL_2(\R)$ and $\r = \smcoltwo{r_1}{r_2} \in \R^2$. Then,
\begin{equation}
\varpi_{\r}(A\cdot\tau) = \varpi\left(\frac{\sympt{A^{-1}\r}{\tau}}{j(A,\tau)},A\cdot\tau\right).
\end{equation}
\end{lem}
\begin{proof}
We use the identity $\sympt{\r}{A\cdot\tau} = \frac{\sympt{A^{-1}\r}{\tau}}{j(A,\tau)}$. Specifically,
\begin{align}
\varpi_\r(A\cdot\tau) 
&= \varpi\left(\sympt{\r}{A\cdot\tau},A\cdot\tau\right)
= \varpi\left(\frac{\sympt{A^{-1}\r}{\tau}}{j(A,\tau)},A\cdot\tau\right).
\end{align}
This completes the proof.
\end{proof}

\begin{lem}\label{lem:deltafrac1}
Let $b \in \Z$ and $\r = \smcoltwo{r_1}{r_2} \in \R^2$. If $z = \sympt{\r}{\tau}$, then
\begin{equation}
\frac{\varpi_{T^bS\r}(T^bS\cdot\tau)}{\varpi_{\r}(\tau)} = 
\ee{\frac{\tau-3+\tau^{-1}}{24} + \frac{(\tau-z)(1-z)}{4\tau}} \left(1-\ee{\frac{z}{\tau}}\right) 
\SS_2(z,\tau)^{-1}.
\end{equation}
\end{lem}
\begin{proof}
We have
$T^bS = \smmattwo{b}{-1}{1}{0}$. 
Thus, by \Cref{lem:taut},
\begin{align}
\varpi_{T^bS\r}(T^bS\cdot\tau) 
&= \varpi\left(\frac{\sympt{(T^{b}S)^{-1}T^{b}S\r}{\tau}}{\tau}, -\frac{1}{\tau} + c\right) 
= \varpi\left(\frac{z}{\tau}, -\frac{1}{\tau}\right),
\end{align}
where in the last step, we simplified the elliptic term and used periodicity in the modular coordinate.
We also have
$
\varpi_{\r}(\tau) = \varpi(z, \tau).
$
The claim now follows by \Cref{thm:shin5}.
\end{proof}

We actually need to transform the formula in \Cref{lem:deltafrac1} further. We will want to translate $z$ to some $\widetilde{z}$ in the fundamental domain of the (pseudo)lattice $\tau \Z + \Z$ so that we may compare the factors in Shintani's formula. We also replace $\r$ by $\{\r\}$.
\begin{lem}\label{lem:deltafrac2}
Let $\tau \in \HH$, $b \in \Z$, and $\r \in \R^2$. If $\widetilde{z} = \sympt{\{\r\}}{\tau}$, then
\begin{equation}
\frac{\varpi_{\{T^bS\r\}}(T^bS \cdot \tau)}{\varpi_{\{\r\}}(\tau)} 
= 
\ee{
\frac{\tau-3+\tau^{-1}}{24} + 
\frac{(\tau-\widetilde{z})(1-\widetilde{z})}{4\tau}
}
\SS_2(\widetilde{z},\tau)^{-1}. 
\end{equation}
\end{lem}
\begin{proof}
Write $\r = \smcoltwo{r_1}{r_2}$.
Note that $\{T^bS\r\} = T^bS\tilde\r$ for some $\tilde\r \con \r \Mod{\Z^2}$; it thus suffices without loss of generality to prove the lemma under the assumption that $\{T^bS\r\}=T^bS\r$; we make that assumption henceforth. Since $T^bS\r = \smcoltwo{br_1-r_2}{r_1}$, it follows that $0 \leq r_1 < 1$. Let $m=\floor{r_2}$; then, $\{\r\} = \smcoltwo{r_1-1}{r_2-m}$.

As in \Cref{lem:deltafrac1}, let $z = \sympt{\r}{\tau}$.
Using the 1-quasiperiodicity of the double sine function (from \Cref{prop:quasiperiodicity}), we rewrite the formula in \Cref{lem:deltafrac1} as follows.
\begin{align}
\frac{\varpi_{T^bS\r}(T^bS\tau)}{\varpi_{\r}(\tau)} 
&= 
\ee{\tfrac{\tau-3+\tau^{-1}}{24} + \tfrac{(\tau-z)(1-z)}{4\tau}} \left(1-\ee{\tfrac{z}{\tau}}\right) 
\SS_2(z,\tau)^{-1} \\
&=
\ee{\tfrac{\tau-3+\tau^{-1}}{24} + \tfrac{(\tau-z)(1-z)}{4\tau}} \left(1-\ee{\tfrac{z}{\tau}}\right) 
\frac{\SS_2(z+1,\tau)^{-1}}{2\sin\left(\frac{\pi z}{\tau}\right)} \\
&=
\ee{\tfrac{\tau-3+\tau^{-1}}{24} + \tfrac{(\tau-z)(1-z)}{4\tau}} \left(1-\ee{\tfrac{z}{\tau}}\right) 
\frac{\SS_2(z+1,\tau)^{-1}}{-i\left(\ee{\frac{z}{2\tau}}-\ee{-\frac{z}{2\tau}}\right)} \\
&=
\ee{\tfrac{\tau-3+\tau^{-1}}{24} + \tfrac{(\tau-z)(1-z)}{4\tau}} \ee{\tfrac{z}{2\tau}-\tfrac{1}{4}}
\SS_2(z+1,\tau)^{-1} \\
&=
\ee{\tfrac{\tau-3+\tau^{-1}}{24} + \tfrac{(\tau-(z+1))(1-(z+1))}{4\tau}}
\SS_2(z+1,\tau)^{-1}.
\end{align}
Assume $m \geq 0$. The case $m < 0$ may be handled similarly and is omitted.
Note that $z+1 = \tilde{z}+m\tau$.
By repeated use of the $\tau$-quasiperiodicity property of the double sine function (from \Cref{prop:quasiperiodicity}), we have
\begin{align}\label{eq:deltop}
\frac{\varpi_{T^bS\r}(T^bS\tau)}{\varpi_{\r}(\tau)}
&= 
\ee{\tfrac{\tau-3+\tau^{-1}}{24} + \tfrac{(\tau-(z+1))(1-(z+1))}{4\tau}}
\left(\prod_{k=0}^{m-1} 2\sin\left(\pi(\widetilde{z}+k\tau)\right)\right)
\SS_2(\widetilde{z},\tau)^{-1}. 
\end{align}
We also need to relate $\varpi_{\r}(\tau)$ to $\varpi_{\{\r\}}(\tau)$, which we do using the $q$-product:
\begin{align}
\frac{\varpi_{\{\r\}}(\tau)}{\varpi_{\r}(\tau)}
&= \prod_{k=0}^{m-1} \left(1 - \ee{\widetilde{z}+k\tau}\right) \\
&= \prod_{k=0}^{m-1} \left(-i \ee{\foh \left(\widetilde{z}+k\tau\right)}\right)\frac{\ee{\foh (\widetilde{z}+k\tau)} - \ee{-\foh (\widetilde{z}+k\tau)}}{i} \\
&= \prod_{k=0}^{m-1} \ee{\tfrac{2\widetilde{z}-1}{4} + \tfrac{k\tau}{2}}\left(2\sin\left(\pi(\widetilde{z}+k\tau)\right)\right) \\
&= \ee{\tfrac{m(2\widetilde{z}-1)}{4} + \tfrac{m(m-1)\tau}{4}} \prod_{k=0}^{m-1} 2\sin\left(\pi(\widetilde{z}+k\tau)\right). \label{eq:delbot}
\end{align}
Dividing \eqref{eq:deltop} by \eqref{eq:delbot}, and using the relation  
$z+1 = \widetilde{z}+m\tau$, 
we obtain (after some algebra)
\begin{align}
\frac{\varpi_\r(T^cS\tau)}{\varpi_{\{\r T^cS\}}(\tau)}
&= 
\ee{\tfrac{\tau-3+\tau^{-1}}{24} + \tfrac{(\tau-(z+1))(1-(z+1))}{4\tau}
- \tfrac{m(2\widetilde{z}-1)}{4} - \tfrac{m(m-1)\tau}{4}
}
\SS(\widetilde{z};1,\tau)^{-1} \\
&= 
\ee{
\tfrac{\tau-3+\tau^{-1}}{24} + 
\tfrac{(\tau-\widetilde{z})(1-\widetilde{z})}{4\tau}
} 
\SS(\widetilde{z};1,\tau)^{-1}. 
\end{align}
We have now proved the lemma.
\end{proof}

We are now in a position to relate the Stark--Tangedal--Yamamoto invariant to an RM value of the Shintani--Faddeev modular cocycle.
\begin{prop}\label{prop:almost}
Let $\OO$ be an order in a real quadratic field $F$, and let $\mm$ be an integral ideal of $\OO$.
Let $\A \in \Clt_{\mm\infty_1\infty_2}(\OO)$, 
and let $(\r,\beta)$ be a reduced representative of $\tilde\Upsilon_\mm(\A)$.
Associate to $\beta$ and $\r$ the HJ cycle data $k, \ell \in \Z$, $A \in \Gamma_\r$, and $\beta_i, w_i \in F$.
Identify $\beta$ with $\rho_1(\beta)$.
Then,
\begin{align}
U^{(1)}_{\mm}(\A)^{-1}
&= \ee{-\tfrac{1}{24}\gamma(A)-\tfrac{1}{4}\lambda_\r(A)}\sf{\r}{A}{\beta},
\end{align}
where
\begin{equation}
\gamma(A) := \sum_{n=0}^{k\ell-1}\left(\beta_n - 3 + \beta_n^{-1}\right) = k\sum_{n=0}^{\ell-1}\left(\beta_n - 3 + \beta_n^{-1}\right)
\end{equation}
and 
\begin{equation}
\lambda_\r(A) := \sum_{n=0}^{k\ell-1} \frac{(\beta_n-w_n)(1-w_n)}{\beta_n}.
\end{equation}
\end{prop}
\begin{proof}
By \Cref{defn:tangedalu}, 
\begin{align}
U^{(1)}_{\mm}(\A)
= \prod_{n=0}^{k\ell-1} \Sin_2(w_n, \beta_n)
= \prod_{n=1}^{k\ell} \Sin_2(w_n, \beta_n).
\end{align}
We have $\beta_n = A_{n,k\ell}\cdot\beta$, $T^{b_{n-1}}S\cdot \beta_n = \beta_{n-1}$, and $T^{b_{n-1}}S\cdot \r_n = \r_{n-1}$. Thus, by \Cref{lem:deltafrac2},
\begin{equation}
\Sin_2(w_n,\beta_n)^{-1} = \ee{-\frac{\beta_n-3+\beta_n^{-1}}{24} - \frac{(\beta_n-w_n)(1-w_n)}{4\beta_n}}\lim_{\tau \to \beta}\frac{\varpi_{\r_{n-1}}(A_{n-1,k\ell}\cdot\tau)}{\varpi_{\r_n}(A_{n,k\ell}\cdot\tau)},
\end{equation} 
where the limit is taken over $\tau \in \HH$. Taking the product from $n=1$ to $n=k\ell$ and using the identities $A_{0,k\ell} = A$ and $A_{k\ell,k\ell}=I$, we have
\begin{align}
U^{(1)}_{\mm}(\A)^{-1}
&= \ee{-\tfrac{1}{24}\gamma(A) - \tfrac{1}{4}\lambda_\r(A)}\lim_{\tau \to \beta} \prod_{n=1}^{k\ell} \frac{\varpi_{\r_{n-1}}(A_{n-1,k\ell}\cdot\tau)}{\varpi_{\r_n}(A_{n,k\ell}\cdot\tau)} \\
&= \ee{-\tfrac{1}{24}\gamma(A) - \tfrac{1}{4}\lambda_\r(A)}\lim_{\tau \to \beta} 
\frac{\varpi_{\r}(A\cdot\tau)}{\varpi_{\r}(\tau)} \\
&= \ee{-\tfrac{1}{24}\gamma(A) - \tfrac{1}{4}\lambda_\r(A)}\sf{\r}{A}{\beta},
\end{align}
proving the proposition.
\end{proof}

To prove \Cref{thm:main} from \Cref{prop:almost}, we need to show that the root of unity factors in front agree up to a plus or minus sign. This is done in \Cref{sec:global} for the ``global phase factor'' $\ee{\gamma(A)}$ and in \Cref{sec:local} for the ``$\r$-dependent phase factor'' $\ee{\lambda_\r(A)}$.

We also need to show that the condition that $0 < \beta' < 1 < \beta$ can be removed. This is done as part of the final proof of \Cref{thm:main} in \Cref{sec:proof}.

\subsection{The global phase factor}\label{sec:global}

We relate the global phase factor to the eta-multiplier character $\psi$.

\begin{prop}\label{prop:globalphase}
Let $\beta$ be a real number purely periodic Hirzeburch--Jung continued fraction expansion
\begin{equation}
\beta = [\ol{b_0, b_1, \ldots, b_{\ell-1}}],
\end{equation}
and treat the indices as elements of $\Z/\ell\Z$, so $b_{n+\ell} = b_n$.
Let $\beta_0, \beta_1, \ldots, \beta_{\ell-1}$ be the real numbers with Hirzebruch--Jung continued fraction expansions given by the cyclic permutations of the Hirzebruch--Jung continued fraction expansion of $\beta = \beta_0$; that is,
\begin{equation}
\beta_n = [b_n, b_{n+1} \ldots, b_{n+\ell-1}].
\end{equation}
Let $\ell'$ be the length of the periodic part of the HJ continued fraction of the nontrivial Galois conjugate $\beta'$ of $\beta$, and let $P = T^{b_0}ST^{b_1}S\cdots T^{b_{\ell-1}}S$.
We have the following identities:
\begin{equation}
\sum_{n=0}^{\ell-1}\left(\beta_n - 3 + \beta_n^{-1}\right) 
= -3\ell + \sum_{n=0}^{\ell-1} b_n 
= \ell' - \ell
= \Phi(P)-3
= \Psi(P,\sqrt{j_P}).
\end{equation}
Here $\Phi$ and $\Psi$ are the integer-valued functions defined in \Cref{sec:eta} and appearing in the transformation law of the logarithm of the Dedekind eta function.
\end{prop}
\begin{proof}
First, note that the $\beta_n$ satisfy the recurrence $\beta_{n} = b_n - \beta_{n+1}^{-1}$, that is, $\beta_{n} + \beta_{n+1}^{-1} = b_n$. Thus, we may pair up the appropriate summands as follows.
\begin{align}
\sum_{n=0}^{\ell-1}\left(\beta_n - 3 + \beta_n^{-1}\right) 
&= -3\ell + \sum_{n=0}^{\ell-1}\beta_n + \sum_{n=0}^{\ell-1}\beta_n^{-1} \\ 
&= -3\ell + \sum_{n=0}^{\ell-1}\beta_n + \sum_{n=0}^{\ell-1}\beta_{n+1}^{-1} \\
&= -3\ell + \sum_{n=0}^{\ell-1}b_n. \label{eq:malt1}
\end{align}
This proves the first identity. To prove the second, consider the Hirzebruch--Jung continued fraction expansion of the nontrivial Galois conjugate $\beta'$; let $\ell'$ be the length the periodic part that expansion. By \Cref{lem:hjconj}, each $b_n$ contributes $b_n-2$ entries to the periodic part of the expansion of $\beta'$, so
\begin{equation}
\ell' = \sum_{n=0}^{\ell-1} b_n-2 = -2\ell + \sum_{n=0}^{\ell-1}b_n. \label{eq:malt2}
\end{equation}
By combining \eqref{eq:malt1} and \eqref{eq:malt2}, we have
\begin{equation}
\sum_{n=0}^{\ell-1}\beta_n - 3 + \beta_n^{-1} 
= -3\ell + (\ell' + 2\ell) = \ell' - \ell,
\end{equation}
proving the second relation.
For the relation to $\Phi$ and $\Psi$, we appeal to results of Meyer and of Zagier. Writing $P = \smmattwo{a}{b}{c}{d}$ and combining formulas on \cite[p.\ 155]{zagier} and \cite[p.\ 162]{zagier}, we obtain the formula
\begin{equation}
\ell' - \ell = \frac{a+d-2(d,c)}{c}-3,
\end{equation}
where $(d,c) := 6c s(d,c)$ and $s(d,c)$ denotes the Dedekind sum in our notation from \Cref{sec:eta}.
Zagier cites a proof by Meyer \cite{meyer}.
A straightforward induction on $\ell$ shows that $c>0$. Thus, by \eqref{eq:Phi} and \eqref{eq:Psi}, we obtain the two relations involving $\Psi$ and $\Phi$.
\end{proof}

\subsection{The $\r$-dependent phase factor}\label{sec:local}

We relate the $\r$-dependent phase factor to the theta-multiplier character $\chi_\r$.

\begin{prop}\label{prop:lambdaminus}
Let $\r$ and $A$ be as in \Cref{prop:almost}. Then,
\begin{equation}
\lambda_{\r}(A) = \lambda_{-\r}(A).
\end{equation}
\end{prop}
\begin{proof}
If $\r=\smcoltwo{-1}{0}$, the claim is trivial, so we assume $\r \neq \smcoltwo{-1}{0}$. 
Write
\begin{equation}
\lambda_{\r}(A) = \sum_{n=0}^{\ell-1} a_n \mbox{ where } a_n = \frac{(\beta_n-w_n)(1-w_n)}{\beta_n},
\end{equation}
and write 
\begin{equation}
\lambda_{-\r}(A) = \sum_{n=0}^{\ell-1} \tilde{a}_n \mbox{ where } \tilde{a}_n = \frac{(\beta_n-\tilde{w}_n)(1-\tilde{w}_n)}{\beta_n},
\end{equation}
in a similar manner. 
We have $w_n = \sympt{\r_n}{\beta_n}$ with $\r_n = \{A_{n,0}\r\}$ and $\tilde{w}_n = \sympt{\tilde{\r}_n}{\beta_n}$ with $\tilde{\r}_n = \{-A_{n,0}\r\}$. 
It follows from the assumption that $\r \neq \smcoltwo{-1}{0}$ that $\r_n$ and $\tilde{\r}_n$ are also not equal to $\smcoltwo{-1}{0}$.
We'll now compare the summands $a_n$ and $a_n'$ in three cases.

\textit{Case 1:} If $r_{n1} \neq -1$ and $r_{n2} \neq 0$, then $\tilde{w}_n = \beta_n+1-w_n$, so
\begin{align}
\tilde{a}_n 
&= \frac{(\beta_n-\tilde{w}_n)(1-\tilde{w}_n)}{\beta_n} 
= \frac{(-1+w_n)(-\beta_n+w_n)}{\beta_n} 
= \frac{(\beta_n-w_n)(1-w_n)}{\beta_n}
= a_n. \label{eq:aneq1}
\end{align}

\textit{Case 2:} If $r_{n1} \neq -1$ and $r_{n2} = 0$, then $\tilde{w}_n = 1-w_n$, so
\begin{equation}
\tilde{a}_n
= \frac{(\beta_n-\tilde{w}_n)(1-\tilde{w}_n)}{\beta_n} 
= \frac{(\beta_n-1+w_n)w_n}{\beta_n} 
= \frac{w_n(\beta_n-1+w_n)}{\beta_n}.
\end{equation}
A further algebraic calculation shows that
\begin{align}
\tilde{a}_n - a_n 
&= \frac{w_n(\beta_n-1+w_n) - (\beta_n-\tilde{w}_n)(1-\tilde{w}_n)}{\beta_n}
= \frac{-\beta_n+2\beta_n w_n}{\beta_n}
= -1+2w_n.
\end{align}
Moreover, since $r_{n2}=0$, it follows that $w_n = r_{n2}\beta_n-r_{n1} = -r_{n1}$. Thus, $\tilde{a}_n = a_n - (1+2r_{n1})$.

\textit{Case 3:} If $r_{n1} = -1$ and $r_{n2} \neq 0$, then $\tilde{w}_n = \beta_n+2-w_n$, so
\begin{align}
\tilde{a}_n
&= \frac{(\beta_n-\tilde{w}_n)(1-\tilde{w}_n)}{\beta_n} 
= \frac{(-2+w_n)(-\beta_n-1+w_n)}{\beta_n}
= \frac{(2-w_n)(\beta_n+1-w_n)}{\beta_n}.
\end{align}
A further algebraic calculation shows that
\begin{align}
\tilde{a}_n - a_n 
&= \frac{w_n(\beta_n-1+w_n) - (\beta_n-\tilde{w}_n)(1-\tilde{w}_n)}{\beta_n}
= \frac{\beta_n-2-2w_n}{\beta_n}
= 1+2(1-w_n)\beta_n^{-1}.
\end{align}
Moreover, since $r_{n1}=-1$, it follows that $w_n = r_{n2}\beta_n-r_{n1} = r_{n2}\beta_n+1$, so $(1-w_n)\beta_n^{-1} = -r_{n2}$. Thus, $\tilde{a}_n = a_n + (1-2r_{n2})$.

Since $r_{(n+1)1} = r_{n2}-1$, the instances of \textit{Case 2} and \textit{Case 3} occur in pairs: $r_{n2}=0 \iff r_{(n+1)1}=-1$. When this happens, we have
\begin{equation}
\tilde{a}_n + \tilde{a}_{n+1} = (a_n + a_{n+1}) - 2(r_{n1}+r_{(n+1)2}).
\end{equation}
However, $r_{(n+1)2} = \{-r_{n1}+b_n r_{n2}\} = \{-r_{n1}\} = -r_{n1}$ (because $r_{n1} \neq 1$). So, in fact,
\begin{equation}\label{eq:aneq2}
\tilde{a}_n + \tilde{a}_{n+1} = a_n + a_{n+1}.
\end{equation}

Using \eqref{eq:aneq1} and \eqref{eq:aneq2}, we see that
\begin{equation}
\lambda_{-\r}(A) = \sum_{n=0}^{\ell-1} \tilde{a}_n = \sum_{n=0}^{\ell-1} a_n = \lambda_{\r}(A),
\end{equation}
as desired.
\end{proof}

\begin{prop}\label{prop:lambdachi}
Let $\r$ and $A$ be as in \Cref{prop:almost}, and suppose $\r \nin \Z^2$. Then,
\begin{equation}
\ee{\tfrac{1}{2}\lambda_\r(A)} = \chi_\r(A).
\end{equation}
\end{prop}
\begin{proof}
By \Cref{thm:shincharacter}, we have
\begin{equation}
\sf{\r}{A}{\beta}\sf{-\r}{A}{\beta} = \psi^2(A)\chi_\r(A). \label{eq:gum1}
\end{equation}
By \Cref{prop:almost} (and using \Cref{prop:globalphase} and \Cref{prop:lambdaminus}), we also have
\begin{align}
\sf{\r}{A}{\beta}\sf{-\r}{A}{\beta}
&= \left(\ee{\tfrac{1}{24}\gamma(A)+\tfrac{1}{4}\lambda_\r(A)}U^{(1)}_\mm(\A)^{-1}\right)\left(\ee{\tfrac{1}{24}\gamma(A)+\tfrac{1}{4}\lambda_{-\r}(A)}U^{(1)}_\mm(\sR\A)^{-1}\right) \\
&= \ee{\tfrac{1}{24}\gamma(A)}^2\ee{\tfrac{1}{4}\lambda_{\r}(A)+\tfrac{1}{4}\lambda_{-\r}(A)}\left(U^{(1)}_\mm(\A)U^{(1)}_\mm(\sR\A)\right)^{-1} \\
&= \psi^2(A)\ee{\tfrac{1}{4}\lambda_{\r}(A)+\tfrac{1}{4}\lambda_{-\r}(A)} \left(U^{(1)}_\mm(\A)U^{(1)}_\mm(\sR\A)\right)^{-1} \mbox{ by \Cref{prop:globalphase}} \\
&= \psi^2(A)\ee{\tfrac{1}{2}\lambda_{\r}(A)} \left(U^{(1)}_\mm(\A)U^{(1)}_\mm(\sR\A)\right)^{-1} \mbox{ by \Cref{prop:lambdaminus}}. \label{eq:gum2}
\end{align}
The left-hand side of \eqref{eq:gum2} is on the unit circle by \eqref{eq:gum1}, and $U^{(1)}_\mm(\A)$ and $U^{(1)}_\mm(\sR\A)$ are positive real numbers, so $U^{(1)}_\mm(\A)U^{(1)}_\mm(\sR\A)=1$. Equating \eqref{eq:gum1} and \eqref{eq:gum2}, we have
\begin{equation}
\psi^2(A)\ee{\tfrac{1}{2}\lambda_{\r}(A)} = \psi^2(A)\chi_\r(A),
\end{equation}
and thus $\ee{\tfrac{1}{2}\lambda_{\r}(A)} = \chi_\r(A)$.
\end{proof}

\section{Culmination of proofs of main theorems and concluding remarks}\label{sec:completing}

We now have everything we need to prove \Cref{thm:main} and \Cref{thm:field}.
We also simplify the statement of \Cref{thm:main} using the notation established throughout the paper. The proofs consist mainly of piecing together our various results on modular cocycles and partial zeta functions.

\subsection{Completing the proof of \Cref{thm:main}}\label{sec:proof}

We will now bring together our calculations on modular cocycles and zeta functions to prove \Cref{thm:main}. We restate the theorem in terms of a modified version of (the square of) the shin cocycle that allows us to leave the characters $\psi$ and $\chi_\r$ out of the statement.
\begin{defn}
Let $\r \in \Q$ and $A \in \Gamma_\r$. For $\tau \in \tDD_{\!A}$, define the \textit{samech modular cocycle}
\begin{equation}
\samech^\r_{\!A}(\tau) = (\psi^{-2}\chi_\r^{-1})(A) \, \sf{\r}{\!A}{\tau}^2.
\end{equation}
\end{defn}

\begin{thm}[Restatement of \Cref{thm:main}]\label{thm:mainrestate}
Let $\OO$ be an order in a real quadratic field $F \subset \R$, and let $\mm$ be a nonzero $\OO$-ideal.
Let $\A \in \Clt_{\mm\infty_2}(\OO) \setminus \ZClt_{\mm\infty_2}(\OO)$, and write
\begin{equation}
\Upsilon_\m(\A) = \GL_2(\Z) \cdot (\r,\beta)
\end{equation}
in the notation of \Cref{thm:correspondence}.
Let $n = \frac{2}{\phi^{-1}(\A)}$, where $\phi : \Clt_{\mm\infty_1\infty_2}(\OO) \to \Clt_{\mm\infty_2}(\OO)$ is the natural quotient map.
Then 
\begin{align}\label{eq:mainrestate}
\exp\!\left(n Z_{\mm\infty_2}'(0,\A)\right)
& = \samech^\r[\beta].
\end{align}
\end{thm}
\begin{proof}
\Cref{thm:shinconj} implies the following relation for $R \in \GL_2(\Z)$:
\begin{equation}\label{eq:samechgl}
\samech^{s_R(\beta)R\r}[R\cdot\beta]
= 
\begin{cases}
\vspace{4pt}
\samech^{\r}[\beta] & \mbox{if } \det(R) = 1, \\
\ol{\samech^{\r}[\beta]} & \mbox{if } \det(R) = -1.
\end{cases}
\end{equation}
If \eqref{eq:mainrestate} holds for any pair $(\r,\beta)$ in a $\GL_2(\Z)$-orbit, then it follows that both sides of \eqref{eq:samechgl} must be real, so it will follow that
\begin{equation}\label{eq:samechgl2}
\samech^{s_R(\beta)R\r}[R\cdot\beta]
= \samech^{\r}[\beta].
\end{equation}
Since $\A \nin \ZClt_{\mm,\rS}(\OO)$, we have $\r \nin \Z^2$; thus,
$\samech^{\r}[\beta]$ only depends on the class of $\r \Mod{\Z^2}$ by \Cref{prop:invariance}.
Thus, it suffices to prove \eqref{eq:mainrestate} for one pair $(\r,\beta) \in \Z \times \Fquad$ in each $(\Z^2 \semidirect \GL_2(\Z))$-orbit. Henceforth, we assume $(\r,\beta)$ is \textit{reduced} in the sense of \Cref{defn:reduced}.

By \Cref{thm:ZprimeU}, we have
\begin{equation}
Z_{\mm\infty_2}(0,\A) 
= -t\log\!\left(U^{(1)}_\mm(\A)\right)
= t\log\!\left(U^{(1)}_\mm(\A)^{-1}\right),
\end{equation}
where $U^{(1)}_\mm(\A)$ is the Stark--Tangedal--Yamamoto invariant, $t=1$ if $\OO$ has a unit $u \con 1 \Mod{\mm}$ with $\rho_1(u)>0$ and $\rho_2(u) <0$, and $t=2$ otherwise. The former condition (giving $t=1$) is equivalent to $\OO$ having a unit of negative norm (either $u$ or $-u$) that is congruent to $1 \Mod{\mm}$.
Taking $n = \frac{2}{t}$, we have
\begin{equation}\label{eq:nz}
nZ_{\mm\infty_2}(0,\A) 
= 2\log\!\left(U^{(1)}_\mm(\A)^{-1}\right)
= \log\!\left(U^{(1)}_\mm(\A)^{-2}\right),
\end{equation} 
where $n=2$ if $\OO$ has a unit of negative norm that is congruent to $1 \Mod{\mm}$, and $n=1$ otherwise.
Let $A = A_\beta^+ \in \Gamma_\r$ be the canonical generator of the stabilizer of $\beta$.
By \Cref{prop:almost}, we have
\begin{align}\label{eq:u1}
U^{(1)}_{\mm}(\A)^{-1}
&= \ee{-\tfrac{1}{24}\gamma(A)-\tfrac{1}{4}\lambda_\r(A)}\sf{\r}{A}{\beta},
\end{align}  
where $\gamma(A)$ and $\lambda_\r(A)$ are the rational invariants defined in that proposition. Moreover, these invariants are related to metaplectic characters, as shown in \Cref{sec:global} and \Cref{sec:local}. Specifically, \Cref{prop:globalphase} says that
$\gamma(A) = \Psi(A,\sqrt{j_A})$, and thus $\ee{\tfrac{1}{24}\gamma(A)}^2 = \psi^2(A)$; \Cref{prop:lambdachi} says that $\ee{\tfrac{1}{2}\lambda_\r(A)} = \chi_\r(A)$. Squaring \eqref{eq:u1},
\begin{align}
\left(U^{(1)}_{\mm}(\A)\right)^2
&= \ee{\tfrac{1}{24}\gamma(A)}^{-2}\ee{\tfrac{1}{2}\lambda_\r(A)}^{-1}\left(\sf{\r}{A}{\beta}\right)^2
= (\psi^{-2}\chi_\r^{-1})(A)\left(\sf{\r}{A}{\beta}\right)^2
= \samech^\r[\beta].
\end{align}
Finally, exponentiating \eqref{eq:nz}, we have
\begin{equation}
\exp(nZ_{\mm\infty_2}(0,\A)) = \left(U^{(1)}_\mm(\A)\right)^{-2} = \samech^\r[\beta]. \tag*{\qedhere}
\end{equation}
\end{proof}

We now give a corollary stating the key functional equations and basic properties satisfied by RM values the samech cocycle.
\begin{cor}\label{cor:samechrm}
Suppose that $\r \in \Q^2 \setminus \Z^2$ and $\beta \in \Rquad$. The following hold.
\begin{itemize}
\item[(1)] $\samech^\r[\beta]\samech^{-\r}[\beta] = 1$. 
\item[(2)] $\samech^\r[\beta] = \samech^{\r+\n}[\beta]$ for any $\n \in \Z^2$
\item[(3)] $\samech^{s_R(\beta)R\r}[R\cdot\beta] = \samech^{\r}[\beta]$ for any $R \in \GL_2(\Z)$.
\item[(4)] $\samech^\r[\beta]$ is a positive real number.
\end{itemize}
\end{cor}
\begin{proof}
Property (1) follows from \Cref{thm:shincharacter}. Property (2) follows from \Cref{prop:invariance}. Property (3) follows from \Cref{thm:mainrestate} via \eqref{eq:samechgl2}. Property (4) follows from \Cref{thm:mainrestate} because \eqref{eq:mainrestate} expresses $\samech^\r[\beta]$ as the exponential of a real number.
\end{proof}

\subsection{Stark units and conditional results}\label{sec:stark}

We now discuss the conditional implications of our results under the assumption of Tate's refinement of the Stark conjectures. In particular, we will complete the proof of \Cref{thm:field}.

\begin{proof}[Proof of \Cref{thm:field}]
For $\r \in \Z^2$ (and indeed $\r \in \foh\Z^2$), the conclusion follows from \Cref{thm:trivrmval} (even without assuming \Cref{conj:stark2}), so we may assume $\r \nin \Z^2$ henceforth.

Let $f$ be the conductor of $\beta$ (that is, $b^2-4ac = f^2\Delta_0$ for a fundamental discriminant $\Delta_0$ and a positive integer $f$). 
Set $A = A^+_\beta \in \Gamma_\r$. 
By \Cref{lem:fto1}, there is some $B \in G_f$ and some $\alpha$ of conductor $1$ such that $\beta = B\cdot\alpha$. Choose $n \in \N$ so that 
\begin{equation}
C := B^{-1}A^nB \in \bigcap_{\substack{\s \in \Q^2/\Z^2 \\ B\s - \r \in \Z^2}} \Gamma_\s.
\end{equation}
Then, by \Cref{thm:cllr}, we have
\begin{equation}
\shin^\r[\beta]^n
= \sf{\r}{A^n}{\beta} 
= \sf{\r}{B C B^{-1}}{B \cdot \alpha}
= \prod_{\substack{\s \in \Q^2/\Z^2 \\ B\s - \r \in \Z^2}} \sf{\s}{C}{\alpha}.
\end{equation}
The values $\sf{\s}{C}{\alpha}$ are integral powers of $\shin^\s[\alpha]$, and $\alpha$ has conductor $1$. Thus, (2) implies (1). It suffices to prove (2) on the assumption of \Cref{conj:stark2}. We henceforth assume $f=1$.

Let $w = \sympt{\r}{\beta} = r_2\beta-r_1$. We have defined $\mm$ to be the kernel of the map $\OO_F \to (w\OO_F+\cc)/\cc$ given by $1 \mapsto w$. In other words, $\mm$ is the largest $\OO_F$-ideal with the property that $(\r,\beta) \in \M_{\OO_F,\mm}$. Choose some $\A \in \Cl_{\mm\infty_2}(\OO_F)$ such that $\Upsilon(\A) = \GL_2(\Z) \cdot (\r + \Z^2,\beta)$. By \Cref{thm:mainrestate}, for some $n \in \{1,2\}$ and $A = A^+_\beta$, we have
\begin{equation}
\exp\!\left(n Z_{\mm\infty_2}'(0,\A)\right)
= \samech^\r[\beta]
= (\psi^{-2}\chi_\r^{-1})(A) \shin^\r[\beta]^2.
\end{equation}
Let $\e_\A = \exp\!\left(-Z_{\mm\infty_2}'(0,\A)\right)$.
By \Cref{prop:stark3}, it follows from \Cref{conj:stark2} that $\e_\A \in \OO_H^\times$, where $H = H_{\mm\infty_2}$, and $\e_\A^{1/2}$ is a unit in an abelian extension of $F$. Thus, $\samech^\r[\beta] = \e_\A^n \in \OO_H^\times$, and $\shin^\r[\beta] = \pm \sqrt{(\psi^{2}\chi_\r)(A)}\,\e_\A^{-n/2}$ is a unit in an abelian extension of $F$.
\end{proof}

In \Cref{conj:field}, we no longer define $\mm$ to be the kernel of the map from $\OO \to (w\OO_f+\cc)/\cc$ sending $1 \mapsto w$ (taking here $\OO = \ord(\cc)$). We avoid this because that kernel may not be an $\OO$-invertible ideal, so that the associated $\M_{\OO,\mm}$ would not appear in \Cref{thm:correspondence}. We may need to consider a noncanonical choice of a smaller $\OO$-invertible ideal $\mm$, as in \Cref{eg:notraygroupimage}. (This is a conservative choice---we have not disproved the possibly-stronger conjecture that $\mm$ can be taken to be the kernel of the map from $\OO \to (w\OO_f+\cc)/\cc$ sending $1 \mapsto w$.)

\subsection{Example}\label{sec:example}

For illustrative purposes, we give an example of a nontrivial RM value of the Shintani--Faddeev modular cocycle. This example is related by a Galois automorphism over $\Q(\sqrt{3})$ to the running example in \cite{koppindef,koppklf} and by Galois automorphisms over $\Q$ to the SIC-POVM in dimension $d=5$ \cite{koppsic}.

Consider the order $\OO = \Z[\sqrt{3}]$, which is the maximal order of $\Q(\sqrt{3})$ and has class number $1$. For the modulus $5\infty_2$, the ray class group $\Cl_{5\infty_2}(\OO) \isom \Z/8\Z$. Consider the ray class $\A = [\sqrt{3}\OO] \in \Cl_{5\infty_2}(\OO)$. 
Using \texttt{PARI}, we calculate to high precision the exponentiated derivative partial zeta value 
\begin{equation}\label{eq:nunumber}
\exp(Z_{5\infty_2}'(0,\A)) 
= \exp(2 \zeta_{5\infty_2}'(0,\A)) 
\approx 5.54060902431686855379\ldots
\end{equation}
by writing $\zeta_{5\infty_2}'(0,\A)$ as a linear combination of $L$-functions of finite-order Hecke characters.
Specifically, we take \texttt{F = bnfinit(x{\textasciicircum}2-3)} (a ``Buchmann's number field'' object representing the field $F$ with certain auxiliary data) and \texttt{C = bnrinit(F,[5,[1,0]],1)} (a ``Buchmann's number rays'' object representing the ray class group $\Cl_{5\infty_2}(\OO)$ with certain auxiliary data). The command \texttt{Lvals = bnrL1(C,0,6)} produces a list containing the values of $L'(0,\chi^j)$ for an order $8$ Hecke character $\chi$ of modulus $5\infty_2$ satisfying $\chi([\sqrt{3}\OO]) = \ee{\tfrac{1}{8}}$. We may recover $\zeta_{5\infty_2}'(0,\A) = \frac{1}{8}\sum_{j=0}^7 \ee{-\tfrac{j}{8}}L'(0,\chi^j)$.

In the notation of \Cref{thm:main}, we take $\mm=5\OO$, $\A_0 = [\OO]$, $\bb = \OO$, $\bb\mm = \alpha(\beta\Z+\Z)$ with $\alpha = 5$ and $\beta = \sqrt{3}$, and $\r = \smcoltwo{0}{-1/5}\equiv \smcoltwo{0}{4/5} \Mod{\Z^2}$.
By \Cref{thm:main} and a sign computation, we obtain
\begin{align}\label{eq:shinexample1}
\shin^{\tiny \smcoltwo{0}{4/5}}\!\left[\sqrt{3}\right] 
= \sf{\tiny \smcoltwo{0}{4/5}}{\tiny \smmattwo{26}{45}{15}{26}}{\sqrt{3}} 
= e^{-7\pi i/20} \sqrt{\exp(Z_{5\infty_2}'(0,\A))}.
\end{align}
Conditional on the Stark conjectures,
\begin{align}\label{eq:shinexample2}
\shin^{\tiny \smcoltwo{0}{4/5}}\!\left[\sqrt{3}\right] 
\stackrel{?}{=} e^{-7\pi i/20} \sqrt{\nu} 
\end{align}
where $\nu \approx 5.54060902431686855379\ldots$ is a root of the polynomial
\begin{align}
x^8 
&- (8 + 5\sqrt{3})x^7 
+ (53 + 30\sqrt{3})x^6 
- (156 + 90\sqrt{3})x^5 \\
&+ (225 + 130\sqrt{3})x^4
- (156 + 90\sqrt{3})x^3 
+ (53 + 30\sqrt{3})x^2 
- (8 + 5\sqrt{3})x 
+ 1.
\end{align}
We have done additional numerical checks of \eqref{eq:shinexample1} and \eqref{eq:shinexample2} by computing $\sf{\tiny \smcoltwo{0}{4/5}}{\tiny \smmattwo{26}{45}{15}{26}}{\sqrt{3}}$ in Mathematica, both a limit of $q$-Pochhammer symbols and in terms of a product of double sine values, applying numerical integration to \Cref{prop:sinhintegral} in the latter case.

For computational purposes, it is more efficient to compute RM values of $\shin^\r$ via a product of double sine values than as a limit of $q$-Pochhammer symbols; the comparison is comparable to that described in \cite[Sec.~3]{koppklf}. 
The comparative asymptotic efficiency of using double sine to using the methods for $L$-functions (of finite-order Hecke characters) implemented in PARI requires further investigation.
More will be said about computational methods in \cite{afk}.

\subsection{Asymptotics of the $q$-Pochhammer symbol}\label{sec:asymptotics}

We conclude with a few remarks on the asymptotics of the $q$-Pochhammer symbol near the real line. Specifically, the analytic continuation properties of the $\shin$-function explain the behavior of the function $\varpi_\r(\tau) = (e^{2\pi i (r_2\tau-r_1)},e^{2\pi i \tau})_\infty$ as $\tau \to \beta \in \Rquad$ along modular geodesics.

For $\beta \in \Q$, the behavior of $\varpi_\r(\tau)$ as $\tau \to \beta$ is intimately connected to the dilogarithm function, as well as to the \textit{cyclic quantum dilogarithm}, as discussed in \Cref{sec:rational}. A much more comprehensive study of the behavior of the $q$-Pochhammer symbol near roots of unity may be found in the PhD thesis of Campbell Wheeler \cite{wheelerthesis}.

For real quadratic irrationals, the behavior of asympotic behavior as $\tau \to \beta$ along 
the modular geodesic between $\beta'$ and $\beta$ is given by the following theorem.
\begin{thm}
Let $\r \in \Q^2$ and $\beta \in \Rquad$. Let $\nu_{\r,\beta} = \shin^\r[\beta]$, which is related to a Stark class invariant by \Cref{thm:main}. Let $u$ be the smallest totally positive unit of the quadratic order $\OO = \colonideal{\beta\Z+\Z}{\beta\Z+\Z}$ such that $u>1$ and $(u-1)(r_2\beta-r_1) \in \beta\Z+\Z$. Then,
\begin{equation}\label{eq:varpiasympt}
\varpi_\r\!\left(\frac{\beta + i \beta' e^{-\log(u)t}}{1+i e^{-\log(u)t}}\right)
= \nu_{\r,\beta}^t(f_{\r,\beta}(t)+o(1)) \mbox{ as } t \to \infty.
\end{equation}
where $f_{\r,\beta}$ is a smooth function satisfying $f_{\r,\beta}(t+1) = f_{\r,\beta}(t)$. 
\end{thm}
\begin{proof}
Let $A = A_{\beta}^+$. Diagonalize $A = R\smmattwo{u}{0}{0}{u^{-1}}R^{-1}$ (where $u$ is the unit specified in the statement), and define $A^t = R\smmattwo{u^t}{0}{0}{u^{-t}}R^{-1}$. If $\omega := \frac{\beta+i\beta'}{1+i}$, then for $t \in \R$,
\begin{equation}
A^t \cdot \omega = 
\frac{\beta + i \beta' e^{-2\log(u)t}}{1+i e^{-2\log(u)t}}
\end{equation} 
Thus,
$\varpi_\r(A^{t+1}\cdot\omega) =
\sf{\r}{A}{A^t\cdot\omega}\varpi_\r(A^t\cdot\omega)$, and $\lim_{t \to \infty}\sf{\r}{A}{A^t\cdot\omega} = \nu_{\r,\beta}$. 
Set $g_{\r,\beta}(t) = \nu_{\r,\beta}^{-t}\omega_{\r}(A^t\cdot\omega)$. Then,
\begin{equation}
g_{\r,\beta}(t+1) 
= \nu_{\r,\beta}^{-t-1} \sf{\r}{A}{A^t\cdot\omega}\varpi_\r(A^t\cdot\omega)
= \nu_{\r,\beta}^{-1} \sf{\r}{A}{A^t\cdot\omega} g_{\r,\beta}(t).
\end{equation}
Setting $f_{\r,\beta}(t) = \ds\lim_{\substack{n \to \infty \\ n \in \Z}} g_{\r,\beta}(t+n)$, \eqref{eq:varpiasympt} follows.
\end{proof}

In fact, one may replace $\beta'$ by any $x \in \R$ in \eqref{eq:varpiasympt} (shifting $f_{\r,\beta}$ by a constant depending on $x$). Similarly, along the vertical geodesic between $i\infty$ and $\beta$, we have the asymptotic formula
\begin{equation}\label{eq:vertical}
\varpi_\r\!\left(\beta + i e^{-\log(u)t}\right)
= \nu_{\r,\beta}^t(\tilde{f}_{\r,\beta}(t)+o(1)) \mbox{ as } t \to \infty,
\end{equation}
where $\tilde{f}_{\r,\beta} = f_{\r,\beta}(t+c)$ for some $c \in \R$.
Equation \eqref{eq:vertical} is illustrated in \Cref{fig:qpochasymptotic} for the example in \Cref{sec:example}.
We omit a formal proof of these more general asymptotics. They hold because modular geodesics approach $\beta$ approximately vertically from above, and differences between values of $\varpi_\r$ on nearby points on different geodesics may be bounded appropriately.

\begin{figure}[ht!]
\begin{center}
\begin{tabular}{r}
\includegraphics[scale=0.70]{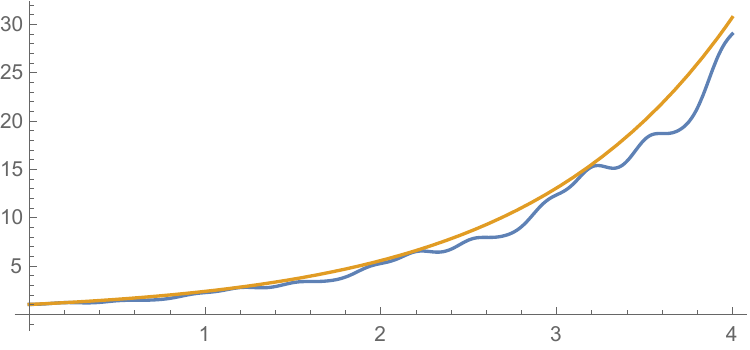}
\\
\includegraphics[scale=0.72]{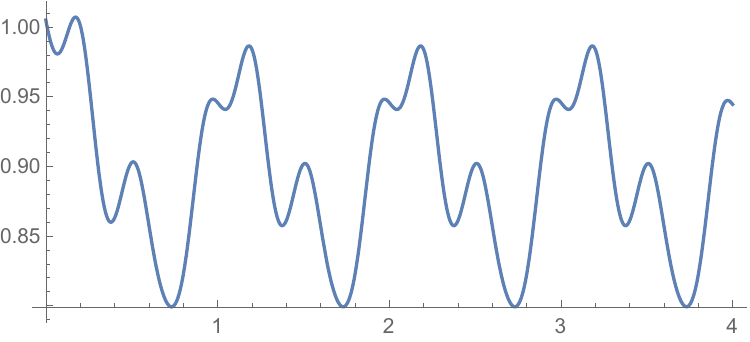}
\\
\includegraphics[scale=0.74]{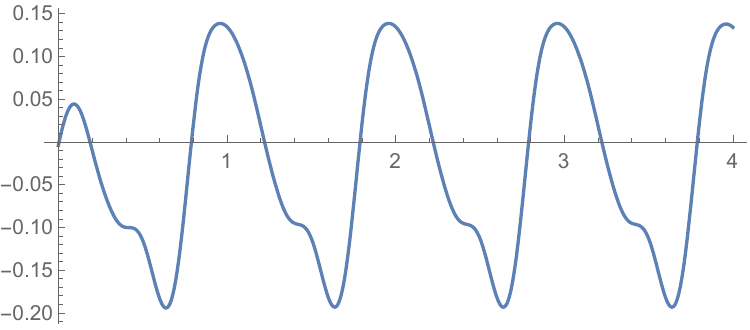}
\end{tabular}
\end{center}
\caption{
The top plot compares $y=\abs{\varpi_{\smcoltwo{0}{4/5}}\!\left(\sqrt{3}+i e^{-6\log(2+\sqrt{3})t}\right)}$ and $y=\mu^t$ for $\mu = e^{-\frac{7\pi i}{20}}\sqrt{\nu}$ and $\nu$ as in \eqref{eq:nunumber}.
The middle and bottom plots show graphs of $y=\abs{\mu^{-t} \,\varpi_{\smcoltwo{0}{4/5}}\!\left(\sqrt{3}+i e^{-6\log(2+\sqrt{3})t}\right)}$ and $y=\arg\!\left(\mu^{-t} \,\varpi_{\smcoltwo{0}{4/5}}\!\left(\sqrt{3}+i e^{-6\log(2+\sqrt{3})t}\right)\right)$, respectively.
}
\label{fig:qpochasymptotic}
\end{figure}

Further investigating the behavior of $\varpi_\r(\tau)$ near the real line could be interesting in several ways. Analytically, such investigations might lead to improved asymptotic formulas, or even exact formulas, for counting integer partitions into parts with congruence restrictions. Algebraically, they could bring Stark units into the ``quantum modular'' universe, perhaps connecting them to the many interesting objects already linked to quantum modular forms, including as $3$-manifolds and conformal field theories.

\section{Typesetting note}\label{sec:typesetting}

The \textit{shin} character $\shin$ and the \textit{samech} character $\samech$ may used as \texttt{{\char`\\}shin} and \texttt{{\char`\\}samech} in \LaTeX\ after adding the following lines to the preamble.
{
\begin{verbatim}
\DeclareFontFamily{U}{rcjhbltx}{}
\DeclareFontShape{U}{rcjhbltx}{m}{n}{<->rcjhbltx}{}
\DeclareSymbolFont{hebrewletters}{U}{rcjhbltx}{m}{n}
\DeclareMathSymbol{\shin}{\mathord}{hebrewletters}{152}
\DeclareMathSymbol{\samech}{\mathord}{hebrewletters}{115}
\end{verbatim}
}

\section{Acknowledgments}\label{sec:acknowledgements}

The author is supported by NSF DMS grant \#2302514 and has been supported by the Heilbronn Institute for Mathematical Research while conducting the research for this paper.

The author thanks Marcus Appleby, Kairi Black, Samit Dasgupta, Steven T.~Flammia, Edna Jones, Jeffrey C.~Lagarias, Owen Patashnick, and David Solomon for helpful conversations.
He also thanks Brett Tangedal for help checking the \texttt{PARI} computation in \Cref{sec:example} and Eleanor McSpirit, Brett Tangedal, and Bora Yalkinoglu for pointing out several important references. 

The author is grateful to the developers of the computer algebra packages Mathematica and PARI, which he has used to compute the example in \Cref{sec:example}.

\bibliographystyle{abbrv}
\bibliography{references}

\end{document}